\documentclass[a4paper,reqno]{amsart}

\usepackage[
  a4paper,
  textwidth=15cm,
  centering, % <--- Centra il blocco di testo (3cm sx e 3cm dx)
  marginparwidth=3cm,
  marginparsep=5mm,
  top=3cm,
  bottom=3cm
]{geometry}
\usepackage[english]{babel}
\usepackage{yfonts}

\usepackage[dvipsnames]{xcolor}
\usepackage{graphicx}
\usepackage{subcaption}
\usepackage{tikz}
\usepackage{pgfplots}
\pgfplotsset{compat=1.18}
\usetikzlibrary{arrows.meta, shadings}

\usepackage{mathtools} 
\usepackage{amssymb}
\usepackage{mathrsfs}
\usepackage{bm}
\usepackage{cancel}
\usepackage{esint}

\usepackage{csquotes}
\usepackage{multicol}
\usepackage{enumitem}
\usepackage[normalem]{ulem}
\usepackage[textsize=tiny]{todonotes}
\usepackage{comment}
\newcommand{\stkout}[1]{\ifmmode\text{\sout{\ensuremath{#1}}}\else\sout{#1}\fi}

\newcommand{\irchi}[2]{\raisebox{\depth}{\(\chi\)}}

\newcommand{\mc}{\mathcal}
\newcommand{\mb}{\mathbb}

\newcommand{\la}{\lambda}
\newcommand{\e}{\varepsilon}

\DeclarePairedDelimiter{\norm}{\lVert}{\rVert}

\newcommand{\pd}[2]{\frac{\partial#1}{\partial#2}}

\newcommand{\ps}[3]{\left( #2, #3 \right)_{#1}}

\newcommand{\R}{\mb{R}}

\newcommand{\N}{\mb{N}}

\newcommand{\diff}{\,\mathrm{d}}
\newcommand{\dint}[1]{\diff #1}
\newcommand{\dx}{\diff x}
\newcommand{\ds}{\diff \sigma}
\newcommand{\dt}{\diff t}
\newcommand{\dvol}{\diff\mathrm{vol}}

\newcommand{\andd}{\text{and}}
\newcommand{\inn}{\text{in }}

\newcommand{\nnu}{\bm{\nu}}

\newenvironment{bvp}{\left\{\begin{aligned}}{\end{aligned}\right.}

\numberwithin{equation}{section}
\usepackage{thmtools} 

\usepackage[hypertexnames=false]{hyperref}
\hypersetup{
    linktocpage=true,
    colorlinks=true,
    linkcolor=blue!70!black,
    citecolor=orange!40!red,
    urlcolor=magenta!80!black,
}

\usepackage[nameinlink, capitalize]{cleveref} 

\declaretheorem[name=Theorem, numberwithin=section]{theorem}
\declaretheorem[name=Lemma, sibling=theorem]{lemma}
\declaretheorem[name=Corollary, sibling=theorem]{corollary}

\declaretheorem[name=Proposition, sibling=theorem]{proposition}

\declaretheorem[name=Definition, sibling=theorem, style=definition]{definition}
\declaretheorem[name=Example, sibling=theorem, style=definition]{example}
\declaretheorem[name=Remark, sibling=theorem, style=definition]{remark}

\mathtoolsset{showonlyrefs=true} % La via moderna e nativa di mathtools.

\allowdisplaybreaks
\begin{document}
\title[On the stability of eigenvalues of varying bilinear forms in abstract Hilbertian settings and applications]
{On the stability of eigenvalues of varying bilinear forms in abstract Hilbertian settings and applications}

\author[A. Bisterzo, R. Ognibene, P. Roychowdhury, G. Siclari]{Andrea Bisterzo, Roberto Ognibene,\\ Prasun Roychowdhury and Giovanni Siclari}

\address{Andrea Bisterzo
  \newline \indent Centro di Ricerca Matematica Ennio De Giorgi
  \newline \indent
Scuola Normale Superiore
\newline\indent Piazza dei Cavalieri 3, 56126, Pisa, Italy.}
\email{andrea.bisterzo@sns.it}

\address{Roberto Ognibene
  \newline \indent Dipartimento di Matematica e Applicazioni
  \newline \indent
Università degli Studi di Milano-Bicocca
\newline\indent Via Roberto Cozzi 55, 20126, Milano, Italy.}
\email{roberto.ognibene@unimib.it}

\address{Prasun Roychowdhury
  \newline \indent Department of Mathematics
  \newline \indent Indian Institute of Technology Hyderabad
  \newline\indent Kandi, Sangareddy, Telangana, 502285, India}
\email{prasunrc@math.iith.ac.in}

\address{Giovanni Siclari 
  \newline \indent  Centro di Ricerca Matematica Ennio De Giorgi
  \newline \indent
Scuola Normale Superiore
\newline\indent Piazza dei Cavalieri 3, 56126, Pisa, Italy.}
\email{giovanni.siclari@sns.it}

\date{\today}

\begin{abstract}
The aim of the present paper is to develop a spectral perturbation theory from a higher perspective. More precisely, we consider a one-parameter family of varying bilinear forms, each of them defined on a (possibly) different Hilbert space. Assuming the stability of the corresponding spectra, our first main result establishes a quantification of the rate of convergence. A key feature is the explicit variational characterization of the first term in the asymptotic expansion of the perturbed eigenvalues, which only depends on the ``data'', i.e. the limit eigenspace and the magnitude of the perturbation, through the resolution of a minimization problem. Remarkably, we make no assumptions on the perturbed eigenelements (besides, naturally, the spectral stability). Moreover, we cover both the cases of simple and multiple limit eigenvalues in full generality. In the second part, we explore some concrete applications of our abstract results. First, we consider eigenvalue problems for the Laplace-Beltrami operator with varying measure weights (also motivated by optimization in spectral geometry); secondly, we investigate the Neumann approximation of the Steklov eigenvalues of the Laplacian; finally, we focus on how the spectrum of the Laplace-Beltrami operator on a Riemannian manifold changes when a second small manifold is glued on a small portion of it.
\end{abstract}

\maketitle

\tableofcontents

\hskip10pt{\footnotesize {\bf Keywords.} Spectral stability, singular perturbation, eigenvalues of bilinear forms, Neumann-to-Steklov approximation, connected sum of manifolds.

\medskip 

}

\section{Introduction}\label{sec_intro}

The study of the asymptotic behavior of the spectrum of sequences of linear operators has deep roots in mathematical analysis and partial differential equations, and a quite long history. The need for this kind of results has a naive interpretation, which is essentially the core of the infinitesimal calculus: given a family of linear operators $\{T_\e\}_{\e\in[0,1]}$, to deduce properties of the \emph{perturbed} operator $T_\e$ (which might be hard to be studied) from the properties of a certain \emph{limit} operator $T_0$ (which might be easier to handle), assuming that $T_\e$ is close, in some sense, to $T_0$ when $\e>0$ is small. To be more precise, we want to understand \emph{how close} the spectrum $\sigma(T_\e)$ of $T_\e$ is to the spectrum $\sigma(T_0)$ of $T_0$, assuming that $T_\e$ is converging in some sense (which we are going to specify later) to the limit $T_0$, as $\e\to 0^+$.

\medskip

In order to avoid technicalities in this part, we now introduce the main ideas of the present work in an informal (and possibly not rigorous) way, and we postpone the statement of the precise assumptions to \Cref{sec_assumptions} and of the main results to \Cref{sec:main_results}, with applications provided in \Cref{sec_manifolds_measures}, \Cref{sec_conc_boundary}, \Cref{sec_sum_mani}.

\medskip

We consider a family of Hilbert spaces $\{H_\e\}_{\e\in[0,1]}$ and a family of self-adjoint linear operators $\{T_\e\}_{\e\in[0,1]}$. For any $\e\in[0,1]$, the operator $T_\e$ has a domain $D[T_\e]\subseteq H_\e$ and we assume it possesses a purely discrete spectrum, i.e. there holds
\begin{equation}\sigma(T_\e)=\{\lambda_{\e,n}\}_{n\in\N\setminus\{0\}},\end{equation}
with $\sigma(\cdot)$ denoting the spectrum of an operator. We recall that this holds true if and only if $T_\e$ has a compact resolvent. We now assume the convergence, as $\e\to 0^+$, of the spaces $H_\e$ and of the operators $T_\e$ to the corresponding limit space $H_0$ and operator $T_0$, respectively. In particular, we assume the continuity of every eigenvalue with respect to $\e$, i.e.
\begin{equation}
    \lambda_{\e,n}\to \lambda_{0,n}\quad\text{as }\e\to 0^+,\quad\text{for all }n\in\N.
\end{equation}
The purpose of the present paper is to investigate the asymptotic behavior of the eigenvalue variation 
$\la_{\e,n}-\la_{0,n}$ as $\e \to 0^+$, under suitable assumptions on the family $\{T_\e\}_{\e \in [0,1]}$. In other words, for any fixed $n\in\N$, we look for a function $\rho_n(\e)$ such that $\lim_{\e\to 0^+}\rho_n(\e)=0$ and
\begin{equation}\label{eq:intr_asy}
    \lambda_{\e,n}=\lambda_{0,n}+\rho_n(\e)+o(\rho_n(\e))\quad\text{as }\e\to 0^+.
\end{equation}
The question of \emph{quantifying} the spectral stability is not new and has been investigated in the past. Among the main references, we surely find the comprehensive work of Kato, see e.g. \cite[Chapter 8, § 2, paragraph 3, Theorem 2.6]{kato}. We point out that such a result deals with operators perturbed at the first order, i.e. of the type
\begin{equation}\label{eq:first_order_pert}
    T_\e=T_0+\e T,
\end{equation}
for some operator $T$. In particular, this framework requires a common operator domain and, in some sense, restricts the range of possible perturbations. To be more precise, it seems difficult to apply the theory developed in \cite{kato} even to some simple examples of \emph{singular} perturbations: indeed, a first-order perturbation of the type \eqref{eq:first_order_pert} naturally yields a differentiability result, i.e.
\begin{equation}\label{eq:kato_diff}
    \frac{d}{d\e}_{|\e=0}\lambda_{\e,n}=\lim_{\e\to 0^+}\frac{\lambda_{\e,n}-\lambda_{0,n}}{\e}\in\R,
\end{equation}
while many notable applications (such as the Laplacian in perforated domains, see e.g. \cite{FLO_neumann,FLO2}) display a higher order rate of convergence (i.e. $\e^k$ with $k>1$). Hence, in these cases, \eqref{eq:kato_diff} is not enough to detect the sharp rate of convergence of the eigenvalue variation and only provides non-sharp estimates. On the other hand, we emphasize that our results yield asymptotic estimates given in terms of a quantity ($\rho_n(\e)$ in \eqref{eq:intr_asy}) which depends on the ``size'' of the perturbation and on the limit eigenelements. In turn, in many applications, one can analyze the asymptotic behavior of such a quantity, thus providing an explicit rate of convergence in terms of a power of $\e$, which is typically sharp.

To extend the literature review, we mention the series of works by O. Post and collaborators, see e.g. \cite{Post2013,Post2020} and references therein, in which the authors deal with an abstract framework similar to ours and investigate the question of the convergence of operators defined on varying Hilbert spaces. However, to the best of our knowledge, there are no developments in the direction of detecting the explicit rate of convergence of eigenvalues. 
It is also worth mentioning the survey \cite{BLL} in which the authors give an overview of some spectral stability results for differential operators in general frameworks. Another work in the topic of quantitative spectral stability is \cite{BS_eigen}, in which the authors treat the case of simple limit eigenvalues, but allowing more general perturbations than the ones in \cite{kato}. We finally mention the works by Kozlov and coauthors (see \cite{kozlov2006,kozlov2013,KT,kozlov2020} and references therein),  which deal with an abstract setting and obtain quantitative stability results similar to ours (see e.g. \cite[Theorem 1]{kozlov2006} or \cite[Theorem 3.2]{KT}). Nevertheless, these results seem to be ineffective when considering singular perturbations, such as the excision of a small hole from a domain. In the present work, we improve all the previously cited results by considering a general Hilbertian setting (which allows us to cover fairly general, possibly \emph{singular}, perturbations) and we also consider multiplicities in the limit spectrum.

\medskip

The interest in results of the type \eqref{eq:intr_asy} in general Hilbertian settings lies in the large variety of possible applications. Indeed, our unified approach allows to recover a series of results already established in the literature, among which we find:
\begin{itemize}
    \item the classical Hadamard formulas for the eigenvalues of the Dirichlet and Neumann Laplacian (see e.g. \cite{Henry,KT});
    \item the asymptotic expansion of the eigenvalues of the Laplacian in domains with small holes, with Dirichlet \cite{courtois,flucher}, Neumann \cite{FLO_neumann,FLO2,bucur_GAFA} or Robin \cite{FRS_robin} boundary conditions;
    \item asymptotic estimates of the Dirichlet eigenvalues in domains with a thin tube attached to the boundary \cite{AO,FO};
    \item expansion of eigenvalues of Aharonov-Bohm operators with moving poles, see e.g. \cite{FNOS,FRS_AB} and references therein.
\end{itemize}
Besides recovering known results, our tools allow to obtain a wide range of new spectral asymptotic expansions in more concrete settings. In particular, in the present paper we choose to deepen three applications, which we here briefly anticipate.
\begin{enumerate}
    \item \emph{Eigenvalues of the Laplacian with varying measure weights.} We consider the following eigenvalue problem
    \begin{equation}\label{eq_intr_meas_eq}
        \begin{bvp}
            -\Delta^g u+(\alpha_\e+\beta_\e) u&=\lambda\beta_\e u,&&\text{in }\Omega, \\
            \partial_{\nnu} u&=0, &&\text{on }\partial\Omega,
        \end{bvp}
    \end{equation}
    where $(\overline{\Omega},g)$ is a compact, connected Riemannian manifold with Lipschitz boundary $\partial \Omega$, $\Omega=\overline{\Omega}\setminus \partial \Omega$ is the interior of $\overline{\Omega}$ and $\Delta^g$ is the (negative definite) Laplace-Beltrami operator. Moreover, $\{\alpha_\e\}_{\e\in[0,1]}$ and  $\{\beta_\e\}_{\e\in[0,1]}$ are two families of  Radon measures which are converging in a certain sense as $\e\to 0^+$, see \Cref{def:adm} for the precise assumptions. We point out that \eqref{eq_intr_meas_eq} is equivalent to 
    \begin{equation}
        \begin{bvp}
            -\Delta^g u+\alpha_\e u&=\mu\beta_\e u,&&\text{in }\Omega, \\
            \partial_{\nnu} u&=0, &&\text{on }\partial\Omega,
        \end{bvp}
    \end{equation}
    with $\mu=\lambda-1$ and that $\alpha_\e$ can be taken to be identically zero. In particular, both families of measures can vary all around $\overline{\Omega}$, thus being allowed to concentrate on the boundary (see the next point for an example). 
    
    The idea of quantifying the rate of convergence of $\lambda_{\e,n}-\lambda_{0,n}$ was born from the inspiring paper \cite{GKL}, which establishes the spectral stability in a similar setting in order to prove inequalities concerning Steklov and conformal eigenvalues of a manifold.   
    We refine the continuity result by providing the first term in the asymptotic expansion, as a consequence of \eqref{eq:intr_asy}, covering both the case of simple and multiple limit eigenvalues. In particular, if $\lambda_{0,N}$ is of multiplicity $M$, then
    we have that
    \begin{equation}\label{eq_intr_meas}
        \lambda_{N+i-1,\e}=\lambda_{0,N}+\gamma_{\e,i}+~\textnormal{remainder terms,}\quad\textnormal{as }\e\to 0^+,
    \end{equation}
    for $i=1,\dots,M$, where $\{\gamma_{\e,i}\}_{i=1,\dots,M}$ are the eigenvalues of a certain bilinear form depending on $\beta_\e$ (see \eqref{eq:h_measures}).
    Comparing our assumptions on the families $\{\alpha_\e\}_{\e\in[0,1]}$ and  $\{\beta_\e\}_{\e\in[0,1]}$ (required for the quantitative stability \eqref{eq_intr_meas}) to the ones in \cite{GKL}, we notice the following. In \cite{GKL} the authors have two sets of assumptions: one on the family of measures (see (\textbf{M1}), (\textbf{M2}) and (\textbf{M3})), and one on the eigenfunctions corresponding to the perturbed problem (see (\textbf{EF1}) and (\textbf{EF2})). This second group may be difficult to check in concrete examples, since one needs to analyze quantities which depend on the underlying measures through a variational problem (namely, the eigenfunctions). On the other hand, for our result to hold (see \eqref{eq_intr_meas} and, more precisely, \Cref{theor_eigen_var_measures}) assumptions only on the measures involved are required, thus leading to a neater framework. We refer to \Cref{sec_manifolds_measures} for all the details.
    
    \item \emph{Neumann to Steklov/Robin asymptotics.} As an application of the previous result, we derive the asymptotic behavior, in any dimension and for both simple and multiple eigenvalues, of the spectrum of the following problem:
    \begin{equation}\label{eq:intr_neu_rob}
        \begin{bvp}
            -\Delta^g u&=\lambda p_\e u, &&\text{in }\Omega, \\
            \partial_{\nnu} u&=0, &&\text{on }\partial\Omega,
        \end{bvp}
    \end{equation}
    with $\Omega\subseteq\R^d$ open, bounded and of class $C^{1,1}$,
    \begin{equation}
        p_\e:=\frac{1}{\e}\chi_{\Omega_\e}
    \end{equation}
    and $\Omega_\e:=\{x\in\Omega\colon \mathrm{dist}(x,\partial\Omega)<\e\}$. It is well known that the eigenvalues of \eqref{eq:intr_neu_rob} converge to the Steklov eigenvalues
    \begin{equation}\label{eq_intr_stek}
        \begin{bvp}
            -\Delta u&=0, &&\text{in }\Omega, \\
            \partial_{\nnu} u&=\Lambda u, &&\text{on }\partial\Omega.
        \end{bvp}
    \end{equation}
    The Neumann approximation of the Steklov eigenvalues attracted a lot of attention, since it transposes a boundary eigenvalue problem to a ``bulk'' one which is a \emph{singular} perturbation of it. This then allows to compare the Steklov eigenvalues with some weighted Neumann ones and one can exploit this feature, for instance, in order to treat optimization problems (see e.g. \cite{GKL}). We refer to \cite{Provenzano3,Provenzano1,Provenzano2} for some previous results on the field which cover, in particular, the $2D$ case for simple eigenvalues and the case in which $\Omega$ is a ball (every dimension and multiplicity). We also mention \cite{acampora,aragao,Arrieta,cristoforoni,Golovaty} for other related works on asymptotic spectral analysis in domains with thin layers (see \cite{lobo} for a review). For the sake of completeness, we recall that it is possible to recover the Neumann eigenvalues as the limit of Steklov problems in perforated domains, see \cite{stek-to-neu}.

    Our main result can be obtained by applying \eqref{eq_intr_meas} with $\alpha_\e\equiv 0$ and
    \begin{equation}
        \beta_\e=p_\e \,\mathcal{L}^d\to \mathcal{H}^{d-1}\text{\huge$\llcorner$}\,\partial\Omega\quad\textnormal{as }\e\to 0^+.
    \end{equation}
    Then, we compute the second variation of the map
    \begin{equation}
        \e\mapsto \int_{\Omega\setminus\Omega_\e} u\dx
    \end{equation}
    obtaining that
    \begin{equation}
        \frac{\mathrm{d}^2}{\mathrm{d}\e^2}_{|\e=0} \int_{\Omega\setminus\Omega_\e} u\dx=-\int_{\partial\Omega}(\partial_{\nnu}u+H_{\partial\Omega} u)\,\mathrm{d}\sigma,
    \end{equation}
    where $H_{\partial\Omega}\in L^\infty(\partial\Omega)$ denotes the mean curvature (sum of the principal curvatures), and we employ a fine asymptotic analysis in Fermi coordinates in order to detect the precise asymptotic behavior of $\gamma_{\e,i}$. In particular, we get that, if $\Lambda_{0,N}$ is an eigenvalue of \eqref{eq_intr_stek} of multiplicity $M$,
    \begin{equation}\label{eq_intr_stek_asy}
        \Lambda_{N+i-1,\e}=\Lambda_{0,N}+\e\gamma_i+ o(\e) \quad \text{ as } \e \to 0^+,\quad\textnormal{for }i=1,\dots,M,
    \end{equation} 
    where $\{\gamma_i\}_{i=1,\dots,M}$ are the eigenvalues of the bilinear form
    \begin{equation}
        h_0(\varphi,\psi):=\int_{\partial\Omega}\left[ \frac{2}{3}\Lambda_{0,N}^2+\frac{1}{2}\Lambda_{0,N} H_{\partial\Omega}\right]\varphi\psi\ds,
    \end{equation}
    defined for $\varphi,\psi$ belonging to the eigenspace of $\Lambda_{0,N}$.

    We point out that our estimate \eqref{eq_intr_stek_asy} is valid in every dimension and for every eigenvalue, regardless of the multiplicity, and that it is sharp whenever the eigenbasis $\{\varphi_{N+i-1}\}_{i=1,\dots,M}$ diagonalizing $h_0$ satisfies
    \begin{equation}
        \int_{\partial\Omega}\left[ \frac{2}{3}\Lambda_{0,N}^2+\frac{1}{2}\Lambda_{0,N} H_{\partial\Omega}\right]\varphi_{N+i-1}^2\ds>0.
    \end{equation}
    This is true, for instance, in every convex domain. We refer to \Cref{subsec_steklov} for the precise statements and for the proofs.

    We remark that, by means of the very same techniques, we are able to derive completely analogous asymptotic expansions for the Neumann approximation of Robin eigenvalues, see \Cref{subsec_neum_to_rob}.
    \item \emph{Gluing a small manifold to a fixed one.} The last application of our theoretical results has a geometrical flavor. Given two $m$-dimensional Riemannian manifolds $(M_0,g_0)$ and $(M_1,g_1)$, and given two points $p_0\in M_0$ and $p_1\in M_1$, we consider the topological manifold
    \begin{equation}
        M_\e:=\left(M_0\setminus B_\e^{M_0}(p_0)\right)\cup_{\Phi_\e} \left( M_1\setminus B_1^{M_1}(p_1) \right),
    \end{equation}
    where $\cup_{\Phi_\e}$ stand for the connected sum (see \cite{Ta02} and \Cref{sec_sum_mani} for the details) endowed with the metric 
    \begin{equation}
        g_\e:=\begin{cases}
            g_0, &\text{in }M_0\setminus B_\e^{M_0}(p_0), \\
            \e^2 g_1,&\text{in } M_1\setminus B_1^{M_1}(p_1).
        \end{cases}
    \end{equation}
    Let $\{\lambda_{0,n}\}_n$ and $\{\lambda_{\e,n}\}_n$ denote, respectively, the spectrum of $-\Delta^{g_0}+1$ on $M_0$ and of $-\Delta^{g_\e}+1$ on $M_\e$. Since, by \cite[Theorem 1.1]{Ta06} we know that
    \begin{equation}
        \lambda_{\e,n}\to \lambda_{0,n}\quad\textnormal{as }\e\to 0^+,~\text{for any }n\in\N,
    \end{equation}
    in \Cref{sec_sum_mani} we investigate the rate of such convergence. In particular, this situation falls under the assumptions of our main abstract result \Cref{thm:main}, which can then be applied. We then perform a fine blow-up analysis and we get that
    \begin{equation}\label{eq_intr_sum}
        \lambda_{\e,N+i-1}=\lambda_{0,N}+\gamma_{0,i}\e^m+o(\e^m)\quad\text{as }\e\to 0^+,
    \end{equation}
    for any $i=1,\dots,M_{0,N}$, with $M_{0,N}$ denoting the multiplicity of $\lambda_{0,N}$. In \eqref{eq_intr_sum}, the coefficients $\{\gamma_{0,i}\}_{i=1,\dots,M_{0,N}}$ are characterized as the eigenvalues of the bilinear form
    \begin{equation}
        h_0(\varphi,\psi):=-\mathbb{A}(\nabla^{g_0}\varphi(p_0),\nabla^{g_0}\psi(p_0))+(\lambda_{0,N}-1)\big[|\mathbb{B}_1|-\mathrm{vol}^{g_1}(M_1(1)))\big]\varphi(p_0)\psi(p_0),
    \end{equation}
    defined on the $M_{0,N}$-dimensional eigenspace corresponding to $\lambda_{0,N}$, where $\mathbb{B}_1$ denotes the Euclidean ball of radius $1$, $M_1(1):=M_1\setminus B_1^{M_1}(p_1)$ and $\mathbb{A}\colon \R^m\times\R^m\to \R$ denotes a certain, explicit, symmetric bilinear form (see \Cref{theor_eigen_connected_sum_sharp}).
    We refer to \Cref{sec_sum_mani} for the details and, in particular, to \Cref{theor_eigen_connected_sum_sharp} for the precise statement. 
    
    Similar problems (i.e. the particular case of attaching thin handles) have been treated in various works, see e.g. \cite{colette-colbois,colette-post,colette-takahashi,chavel} and references therein, where the spectral stability is investigated. We also remark that this kind of perturbation has important connections with Hodge-de Rham theory (see e.g. \cite{colette-takahashi,mazzeo} and references therein).
    
\end{enumerate}

\subsection{Assumptions}\label{sec_assumptions}
In this section, we fix the precise assumptions and the main notation for our abstract framework. In particular, we are going to distinguish between two distinct sets of assumptions: 
\begin{enumerate}
    \item hypotheses on the functional framework, which are given for any \emph{fixed} $\e\in[0,1]$;
    \item ``stability'' assumptions, i.e. on how the spaces and the operators behave in the limit as $\e\to 0^+$.
\end{enumerate}
We also observe that, up to a suitable reparametrization of the family of problems under consideration, we may assume without loss of generality that those properties proved in this work for sufficiently small values of $\varepsilon>0$ hold in fact for every $\varepsilon \in [0,1]$. Indeed, what is essential for our purposes is their eventual validity as $\varepsilon$ goes to $0^+$.

\subsubsection{Assumptions on the functional setting}\label{subsec:ass_functional}

For any \emph{fixed} $\e\in[0,1]$, we consider a real separable Hilbert space
\begin{equation}
    \big(H_\e,(\cdot,\cdot)_{H_\e}\big)
\end{equation}
and a real bilinear form 
\begin{equation}
    \mathcal{E}^{(\e)}\colon \mathcal{F}_\e\times\mathcal{F}_\e\to \R
\end{equation}
with domain $\mathcal{F}_\e\subseteq H_\e$ being a linear subspace. We assume, without loss of generality, that $\mathcal{F}_\e$ is dense in $H_\e$, i.e. $\overline{\mathcal{F}_\e}^{\|\cdot\|_{H_\e}}=H_\e$, where
\begin{equation}
    \norm{u}_{H_\e}:=\sqrt{(u,u)_{H_\e}};
\end{equation}
indeed, in case it is not, we simply work in the closure of $\mathcal{F}_\e$ in $H_\e$ as the ambient Hilbert space. We further assume that 
\begin{itemize}
    \item $\mathcal{E}^{(\e)}$ is symmetric;
    \item $\mathcal{E}^{(\e)}$ is bounded from below, i.e. there exists $C_\e\in\R$ such that $\mathcal{E}^{(\e)}(u,u)\geq C_\e\norm{u}_{H_\e}^2$ for all $u\in \mathcal{F}_\e$;
    \item $\mathcal{E}^{(\e)}$ is closed, i.e. $\mathcal{F}_\e$ is complete with respect to the norm induced by the scalar product
    \begin{equation}(u,v)_{C_\e}:=\mathcal{E}^{(\e)}(u,v)+(-C_\e+1)(u,v)_{H_\e}.\end{equation}
    Hence, we assume $\mathcal{F}_\e$ to be endowed with $(\cdot,\cdot)_{C_\e}$, so that $(\mathcal{F}_\e,(\cdot,\cdot)_{C_\e})$ is a real, separable Hilbert space.
\end{itemize}
With a slight abuse of notation, we denote by
\begin{equation}
    \mathcal{E}^{(\e)}(u):=\mathcal{E}^{(\e)}(u,u)
\end{equation}
the quadratic form naturally associated with $\mathcal{E}^{(\e)}$. We now consider a subspace
\begin{equation}
   Z_\e\subseteq \mathcal{F}_\e 
\end{equation}
which we assume to be closed in $\mathcal{F}_\e$ (otherwise $Z_\e$ is not Hilbert and $\overline{\mathcal{E}}^{(\e)}$ is not closed) and endowed with the scalar product $(\cdot,\cdot)_{C_\e}$, and we consider the restriction
\begin{equation}
    \overline{\mathcal{E}}^{(\e)}:=\mathcal{E}^{(\e)}_{|Z_\e \times Z_\e}
\end{equation}
of $\mathcal{E}^{(\e)}$ to $Z_\e\times Z_\e$. In the following, $Z_\e$ will be the space where eigenfunctions live. Finally, we consider
\begin{equation}
    \mathcal{Z}_\e:=\overline{Z_\e}^{\norm{\cdot}_{H_\e}}
\end{equation}
endowed with the scalar product $(\cdot,\cdot)_{H_\e}$, which is then a real separable Hilbert space. In view of the assumptions on $\mathcal{E}^{(\e)}$ and on $Z_\e$, we have that the bilinear form $\overline{\mathcal{E}}^{(\e)}$ with domain $Z_\e\times Z_\e\subseteq \mathcal{Z}_\e\times \mathcal{Z}_\e$ is densely defined in $\mathcal{Z}_\e\times \mathcal{Z}_\e$ and satisfies the following:
\begin{itemize}
    \item $\overline{\mathcal{E}}^{(\e)}$ is symmetric;
    \item $\overline{\mathcal{E}}^{(\e)}$ is bounded from below, i.e. $\overline{\mathcal{E}}^{(\e)}(u,u)\geq C_\e\norm{u}_{H_\e}^2$ for all $u\in Z_\e$;
    \item $\overline{\mathcal{E}}^{(\e)}$ is closed, i.e. $Z_\e$ is complete with respect to the norm induced by the scalar product $(\cdot,\cdot)_{C_\e}$. In particular, $Z_\e$ is a real, separable Hilbert space.
\end{itemize}
Then, since the Hilbert spaces $H_\e$ we are considering are not a priori contained in a common general space, we need a way to link the limit eigenfunctions (which play a crucial role in the rate of convergence) with the perturbed spaces. More precisely, we assume the following condition, which must hold for any $\e\in(0,1]$ (the case $\e=0$ is trivial):
\begin{equation}\label{hp_Le_exist}
    \text{there exists a linear map }L_\e\colon Z_0\to \mathcal{F}_\e.
\end{equation}
Indeed, one can observe that \eqref{hp_Le_exist} serves as an inclusion of $Z_0$ into $\mathcal{F}_\e$ (one can think, in concrete examples of functions on Euclidean domains, of extension-by-zero or restriction operators). {In fact, we anticipate that one can think $L_\e$ as a perturbation of the identity since, in some sense, $\mathcal{F}_\e$ is converging to $\mathcal{F}_0$ and \eqref{hp_Le} will be assumed.

\medskip

It is well known that, under the previous assumptions, there exists a unique self-adjoint operator representing the bilinear form $\overline{\mathcal{E}}^{(\e)}$: more precisely, for any $\e\in[0,1]$ there exists $T_\e\colon D[T_\e]\to \mathcal{Z}_\e$, with $D[T_\e]\subseteq Z_\e$ linear subspace, such that $T_\e$ is self-adjoint and there holds 
\begin{equation}
    \overline{\mathcal{E}}^{(\e)}(u,v)=(T_\e u,v)_{H_\e}=(u,T_\e v)_{H_\e}\quad\text{for all }u,v\in D[T_\e].
\end{equation}
The converse is also true, namely given a densely defined self-adjoint operator $T_\e$ that is bounded from below one can define (through the formula above) a densely defined symmetric, closed, bilinear form $\overline{\mathcal{E}}^{(\e)}$ bounded from below (with the same bound) such that $T_\e$ is the operator associated with $\overline{\mathcal{E}}^{(\e)}$. Moreover (in case $C_\e\geq 0$) the domain of $\overline{\mathcal{E}}^{(\e)}$ is $D[T_\e^{1/2}]$ and there holds
\begin{equation}
    \overline{\mathcal{E}}^{(\e)}(u,v)=(T_\e^{1/2} u,T_\e^{1/2}v)_{H_\e}\quad\text{for all }u,v\in Z_\e,
\end{equation}
where $T_\e^{\frac{1}{2}}$ is the square root of the positive operator $T_\e$ in spectral sense, 
see \cite[Theorem 13.33]{R_functional_analysis_book}. We refer to \cite{helffer} for further details. Therefore, under the assumptions in the present paper, there is a one-to-one correspondence of bilinear forms and linear operators; also in view of present and future applications of our results, which deal with variational formulations, in this paper we prefer to use the language of bilinear forms. We hereafter denote by $T_\e$ the self-adjoint operator uniquely associated with $\overline{\mathcal{E}}^{(\e)}$ and by $R_\e\colon \mathcal{Z}_\e\to Z_\e$ the corresponding resolvent operator, i.e. satisfying
\begin{equation}
    \overline{\mathcal{E}}^{(\e)}(R_\e f,v)=(f,v)_{H_\e}\quad\text{for all }f\in H_\e~\text{and all } v\in Z_\e,
\end{equation}
which is known to be bounded.

\medskip

In order to better visualize the abstract framework set above, for the sake of the clarity of the exposition, we provide a concrete example in which we describe the roles of the various spaces and forms defined so far. 
\begin{example}\label{ex:mixed}
    We consider a prototypical perturbation: the Laplacian on Euclidean bounded domains and mixed Dirichlet-Neumann boundary conditions (this setting was considered e.g. in \cite{FNO1,FNO2,AO}). Let $\Omega\subseteq \R^d$ be an open bounded connected set with smooth boundary and let $K_\e\subseteq \partial\Omega$ be the closure (in $\partial\Omega$) of a relatively open and smooth subset of $\partial\Omega$. We consider the following eigenvalue problem:
    \begin{equation}\label{eq:mixed}
        \left\{\begin{aligned}
            -\Delta \varphi+\varphi&=\lambda\varphi, &&\text{in }\Omega, \\
            \varphi&=0, &&\text{on }K_\e, \\
            \partial_{\nnu} \varphi&=0, &&\text{on }\partial\Omega\setminus K_\e.            
        \end{aligned}\right.
    \end{equation}
    The problem, in this framework, is to understand how the eigenvalues of \eqref{eq:mixed} behave in the cases in which $K_\e$ tends
    to disappear or $K_\e$ tends to cover the whole $\partial\Omega$, as $\e\to 0^+$. However, now we do not describe the asymptotic 
    behavior and we just focus on how to cast this problem (for fixed $\e>0$) in the theoretical setting explained above. 
    We define $H^1_{0,K_\e}(\Omega):=\overline{C_c^\infty(\overline{\Omega}\setminus K_\e)}^{\norm{\cdot}_{H^1(\Omega)}}$ and we let
    \begin{align}
        &H_\e=L^2(\Omega), \qquad \mathcal{F}_\e=H^1(\Omega), \qquad Z_\e=H^1_{0,K_\e}(\Omega), \qquad \mathcal{Z}_\e=L^2(\Omega), \\
        &\mathcal{E}^{(\e)}(u,v)=\int_\Omega(\nabla u\cdot\nabla v+uv)\dx.
    \end{align}
    With these choices, we have $C_\e=1$. Moreover, we have that $T_\e=-\Delta +\mathrm{Id}$ acting on the space
    \begin{equation}
        D[T_\e]=\left\{ u\in H^1(\Omega)\colon \Delta u\in L^2(\Omega),~u=0~\text{on }K_\e,~\partial_{\nnu}u=0~\text{on }\partial\Omega\setminus K_\e\right\}.
    \end{equation}
    Now, for the limit spaces (i.e. for $\varepsilon=0$), we have two possibilities:
    \begin{enumerate}
        \item if $K_\e$ \enquote{disappears} in the limit, i.e. the limit problem has Neumann boundary conditions, then we have
            \begin{align}
        &H_0=L^2(\Omega), \qquad \mathcal{F}_0=H^1(\Omega), \qquad Z_0=H^1(\Omega), \qquad \mathcal{Z}_0=L^2(\Omega), \\
        &\mathcal{E}^{(0)}(u,v)=\int_\Omega(\nabla u\cdot\nabla v+uv)\dx.
    \end{align}
        \item if $K_\e$ \enquote{spreads} in the limit and covers the whole $\partial\Omega$, i.e. the limit problem has Dirichlet boundary conditions, then we have
            \begin{align}
        &H_0=L^2(\Omega), \qquad \mathcal{F}_0=H^1(\Omega), \qquad Z_0=H^1_0(\Omega), \qquad \mathcal{Z}_0=L^2(\Omega), \\
        &\mathcal{E}^{(0)}(u,v)=\int_\Omega(\nabla u\cdot\nabla v+uv)\dx.
    \end{align}
    \end{enumerate}
    Moreover, in both cases, $L_\e\colon H^1(\Omega)\to H^1(\Omega)$ or $L_\e\colon H^1_0(\Omega)\to H^1(\Omega)$ is simply the identity.
\end{example}
Now that the functional setting has been established, we define the main characters of the present paper, that is the eigenvalues of $\overline{\mathcal{E}}^{(\e)}$ in $Z_\e$. We say that $\lambda\in\R$ is an \emph{eigenvalue} if there exists $\varphi\in Z_\e\setminus\{0\}$ (called \emph{eigenvector} or \emph{eigenfunction}) such that
\begin{equation}\label{eq:intr_eigen}
    \overline{\mathcal{E}}^{(\e)}(\varphi,v)=\lambda(\varphi,v)_{H_\e}\quad\text{for all }v\in Z_\e.
\end{equation}
For any eigenvalue $\lambda\in\R$ of $\overline{\mathcal{E}}^{(\e)}$, we denote by
\begin{equation}
    E(\lambda):=\{\varphi\in Z_\e\colon \varphi~\text{satisfies }\eqref{eq:intr_eigen}\}
\end{equation}
the corresponding eigenspace. In particular, in the present paper, we focus on the case in which the spectrum is only composed of eigenvalues of finite multiplicity, i.e. the spectrum is \emph{discrete}. 
In particular, we assume the following:
\begin{equation}\label{hp_compact}
\begin{gathered}
\textnormal{the spectrum of  $T_\e$ is discrete and it consists of a non-decreasing,  diverging}\\
\textnormal{sequence of eigenvalues $\{\la_{\e,n}\}_{n \in \mb{N}\setminus\{0\}}$ with finite multiplicity},
\end{gathered}
\end{equation}
which is known to be equivalent to
\begin{equation}\label{eq:hp_compact'}
    \text{the embedding $i_\e\colon Z_\e\hookrightarrow \mathcal{Z}_\e$ is compact.}
\end{equation}
or to
\begin{equation}\label{eq:hp_compact''}
    \text{the operator $i_\e\circ R_\e\colon\mathcal{Z}_\e\overset{R_\e}{\longrightarrow}Z_\e\overset{i_\e}{\longrightarrow}\mathcal{Z}_\e$ is compact.}
\end{equation}
We assume that each eigenvalue $\lambda_{\e,n}$ is repeated according to its multiplicity. We denote by $M_{\e,n}$ the multiplicity of $\lambda_{\e,n}$, that is
\begin{equation}
    M_{\e,n}:=\dim E(\lambda_{\e,n}).
\end{equation}
We hereafter denote by $\{\varphi_{\e,n,i}\}_{i=1,\dots, M_{\e,n}}\subseteq E(\lambda_{\e,n})$ a family of eigenfunctions which are orthonormal with respect to $(\cdot,\cdot)_{H_\e}$. In particular, the set
\begin{equation}
    \big\{\varphi_{\e,n,i}\colon n\in\N\setminus\{0\},~i=1,\dots,M_{\e,n}\big\}
\end{equation}
forms an orthonormal basis of $\mathcal{Z}_\e$. We also recall that, in case of purely discrete spectrum, the standard min-max characterization holds, that is
\begin{equation}\label{eq:min-max}
    \lambda_{\e,n}=\min_{\substack{U\subseteq Z_\e \\ \dim U=n}}\max_{\substack{u\in U\\ u\neq 0}}\frac{\overline{\mathcal{E}}^{(\e)}(u)}{\norm{u}_{H_\e}^2}.
\end{equation}
Finally, we need to introduce a new character. We define, for any $\e \in [0,1]$, the following bilinear form
\begin{equation}\label{def_qe}
    \begin{aligned}
        &q_{\e,n}:\mc{F}_\e \times \mc{F}_\e \to \R \\
        &q_{\e,n} (u,v):=\mc{E}^{(\e)} (u,v)-\lambda_{0,n} (u,v)_{H_\e}
    \end{aligned}
\end{equation}
and, for any $v \in \mc{F}_\e $, the following functional 
\begin{equation}\label{def_Je}
\begin{aligned}
    &J_{\e,n,v}\colon \mc{F}_\e\to \R \\
    &J_{\e,n,v}(u):= \frac{1}{2}\mc{E}^{(\e)}(u)- q_{\e,n}(u, v).
\end{aligned}
\end{equation}
One can prove that (see \Cref{prop_J_min}) for any $\e \in [0,1]$ and any $v \in \mc{F}_\e$ the following problem 
\begin{equation}\label{intr_prob_J_min}
\inf\{J_{\e,n,v}(u):u \in Z_\e+v\}.
\end{equation}
admits a unique minimizer $V_{\e,n,v} \in  \mc{F}_\e$. Furthermore, $V_{\e,n,v}$ is the unique solution to the equation 
\begin{equation}\label{intr_eq_Ve}
\mc{E}^{(\e)}(V_{\e,n,v},u)=q_{\e,n}(u,v)\quad  \text{ for any } u \in Z_\e
\end{equation}
such that $V_{\e,n,v}-v \in Z_\e$ and the map $v \mapsto V_{\e,n,v}$ is linear. In case $\varphi\in E(\lambda_{0,n})$, we have that $V_{\e,n,L_\e(\varphi)}$ is going to play the role of a first order approximation of perturbed eigenfunctions. In fact, it already appeared in some particular contexts, such as:
\begin{itemize}
    \item in the case we are considering the Dirichlet Laplacian (where $L_\e$ is the extension by zero) in a domain with a small hole, $V_{\e,n,\varphi}$ is the $\varphi$-capacitary potential of the hole, cf \cite{courtois,flucher,AFHL,ABLM};
    \item in the case of Neumann/Robin Laplacian in perforated domains (in which case $L_\e$ is the restriction operator), $V_{\e,n,\varphi}$ is a (boundary) torsion function associated to $\varphi$, cf \cite{FLO_neumann,FLO2,FRS_robin}.
\end{itemize}
For simplicity of notation, in the sequel we may drop the dependence on $n$, being such an index fixed.

\subsubsection{Stability assumptions}\label{subsec:ass_stability}
Now that the framework for any $\e\in[0,1]$ is completely set, we let $\e$ vary and we define the object of our investigation. 
More precisely, being $\{\lambda_{\e,n}\}_{n\in\N}$ the eigenvalues of $\overline{\mathcal{E}}^{(\e)}$ as in \eqref{hp_compact}, we are interested in their behavior in the limit as $\e\to 0^+$. The very first question regards the continuity of the spectrum (frequently called \emph{stability}), that is
\begin{equation}\label{hp_stability}\tag{S1}
    \lambda_{\e,n}\to \lambda_{0,n}\quad\text{as }\e\to 0^+,\quad\text{for all }n\in\N.
\end{equation}
For sufficient conditions for the validity of \eqref{hp_stability} we refer to \cite[Chapter 2, §1.2]{elasticity_book} and \cite[Corollary 3.19]{Post2013}. Moreover, we point out the necessary and sufficient conditions proposed in \cite{KS_stability} in the case in which all the elements of the family $\{H_\e\}_{\e\in[0,1]}$ are framed into a unique fixed Hilbert space $\mathcal{H}$.

\begin{remark}\label{rmk:pos_def}
    First of all, combining \eqref{eq:min-max} and \eqref{hp_stability} we have that there exists $\e_0\in[0,1]$ such that
    \begin{equation}
        \overline{\mathcal{E}}^{(\e)}(u)\geq \lambda_{\e,1}\norm{u}_{H_\e}^2\geq (\lambda_{0,1}-1)\norm{u}_{H_\e}^2\quad\textnormal{for all }u\in Z_\e~\textnormal{and any }\e\in[0,\e_0].
    \end{equation}
    Now, for any $\e\in[0,\e_0]$ and any $u\in\mathcal{F}_\e$ let us consider the form
    \begin{equation}
        \widetilde{\mathcal{E}}^{(\e)}(u,v):=\mathcal{E}^{(\e)}(u,v)-(\lambda_{0,1}-2)(u,v)_{H_\e}.
    \end{equation}
    We have that it satisfies all the assumptions in \Cref{subsec:ass_functional}, its spectrum is $\{\lambda_{\e,n}-\lambda_{0,1}+2\}_{n\in\N\setminus\{0\}}$ and it satisfies 
\begin{equation}
        \widetilde{\mathcal{E}}^{(\e)}(u)\geq \norm{u}_{H_\e}^2\quad\textnormal{for all }u\in Z_\e.
    \end{equation}
    Therefore, up to renaming
    \begin{equation}
        \e_0\longmapsto 1\quad\textnormal{and}\quad \widetilde{\mathcal{E}}^{(\e)}(u)\longmapsto \mathcal{E}^{(\e)}(u)
    \end{equation}
    it is not restrictive to assume
    \begin{equation}\label{eq:pos_def}
        \overline{\mathcal{E}}^{(\e)}(u)\geq \norm{u}_{H_\e}^2\quad\textnormal{for all }u\in Z_\e~\textnormal{and all }\e\in[0,1].
    \end{equation}
    Moreover, we assume $Z_\e$ to be equipped with the norm $\sqrt{\overline{\mathcal{E}}^{(\e)}(\cdot)}$.
\end{remark}

We now introduce the second ``stability'' assumption. In particular, if $L_\e\colon Z_0\to \mathcal{F}_\e$ is as in \eqref{hp_Le_exist}, we assume that for any $n\in\N\setminus\{0\}$ there holds
\begin{equation}\label{hp_Le}\tag{S2}
\lim_{\e \to 0^+}\ps{H_\e}{L_\e(\varphi_{0,n,j})}{L_\e(\varphi_{0,n,i})}=\ps{H_0}{\varphi_{0,n,j}}{\varphi_{0,n,i}}\quad\text{for any  $i,j=1, \dots M_{0,n}$}.
\end{equation}
Please notice that \eqref{hp_Le} encodes a sort of ``convergence'' of both $(\cdot,\cdot)_{H_\e}$ towards $(\cdot,\cdot)_{H_0}$ and of $L_\e$ towards the identity map. We also emphasize that this kind of assumption is not new in the literature, cf. \cite[Chapter 2, §1.2, (1.4)]{elasticity_book}. The last stability assumption concerns the magnitude, as $\e\to 0^+$, of the minimizer $V_{\e,n,L_\e(\varphi)}\in \mathcal{F}_\e$ of \eqref{intr_prob_J_min} (with $v=L_\e(\varphi)$, $L_\e$ as in \eqref{hp_Le_exist} and $\varphi\in E(\lambda_{0,n})$ being a limit eigenfunction). As it will be clear later, this plays a crucial role in quantifying the rate of convergence of perturbed eigenvalues. We define
\begin{equation}
\delta_{\e,n}:=\sup \{\norm{V_{\e,n,L_\e(\varphi)}}_{H_\e}: \varphi\in E(\lambda_{0,n}) \text{ and } \norm{\varphi}_{H_0}=1 \}.\label{def_delta_e} 
\end{equation}
By linearity (see \Cref{prop_J_min}) one can easily observe that
\begin{equation}\label{rem_delta_e}
    \delta_{\e,n}\leq \left(\sum_{i=1}^{M_{0,n}}\norm{V_{\e,n,L_\e(\varphi_{0,n,i})}}_{H_\e}^2\right)^{\frac{1}{2}}.
\end{equation}
We also emphasize that $\delta_{\e,n}$ is a key quantity, since it will act as a remainder term in the asymptotic expansion of the eigenvalue variation. Our assumption is the following:
\begin{equation}\label{hp_delta_e}\tag{S3}
    \delta_{\e,n}\to 0\quad\text{as }\e\to 0^+.
\end{equation}

\subsection{Main result}\label{sec:main_results}

Now that all the assumptions have been fixed, we state here the first main result of the present paper. In order to do so, we first need to define a couple of more elements. More precisely, for any $\e\in[0,1]$ and any $n\in\N\setminus\{0\}$ we let
\begin{equation}\label{eq:def_P_e}
    \begin{aligned}
        &P_\e\colon E(\lambda_{0,n})\to Z_\e \\
        &P_\e(\varphi):=L_\e(\varphi)-V_{\e,n,L_\e(\varphi)},
    \end{aligned}
\end{equation}
where $L_\e$ is as in \eqref{hp_Le_exist} and $V_{\e,n,L_\e(\varphi)}$ is the unique minimizer of \eqref{intr_prob_J_min} with $v=L_\e(\varphi)$. We have that $P_\e$ is linear and, moreover, one can prove (see \Cref{prop_dim_Fe}) that $P_\e$ is injective for $\e\in(0,1]$ sufficiently small, thus implying that $P_\e\colon E(\lambda_{0,n})\to F_{\e,n}$ is a bijection, where
\begin{equation}\label{eq:def_F_e}
    F_{\e,n}:=P_\e(E(\lambda_{0,n})).
\end{equation}
We now consider the following symmetric bilinear form
\begin{equation}
        h_{\e,n}\colon E(\lambda_{0,n})\times E(\lambda_{0,n})\to \R,
\end{equation}
defined as
\begin{equation}\label{def_h_e}
    h_{\e,n}(\varphi,\psi):=\lambda_{0,n}(V_{\e,n,L_\e(\varphi)},L_\e(\psi))_{H_\e}.
\end{equation}
We have that $h_{\e,n}$ is a symmetric bilinear form (see \Cref{rmk:symmetry}) on a $M_{0,n}$-dimensional space and, in particular, $h_{\e,n}$ admits a set of $M_{0,n}$ real eigenvalues, which we denote by $\{\gamma_{\e,n,i}\}_{i=1,\dots,M_{0,n}}$, assuming to be in increasing order. We define
\begin{equation}\label{def_tau_e}
    \tau_{\e,n}^2:=\sup \{|h_{\e,n}(\varphi,\varphi)|: \varphi\in E(\lambda_{0,n}) \text{ and } \norm{\varphi}_{H_0}=1 \}
\end{equation}
and we observe that, by Cauchy-Schwarz, we have
\begin{equation}
    \tau_{\e,n}^2=O(\delta_{\e,n})\quad\text{as }\e\to 0^+,
\end{equation}
see \Cref{lemma:L_e}. We also point out that
\begin{equation}
    \tau_{\e,n}^2=\max_{i=1,\dots,M_{0,n}}|\gamma_{\e,n,i}|.
\end{equation}
Together with $\delta_{\e,n}$, $\tau_{\e,n}$ will play a crucial role, being a remainder term in the expansion of the eigenvalue variation. At this point, we are ready to state our first main result.
\begin{theorem}\label{thm:main}
    For every $N\in\N\setminus\{0\}$, let $M:=M_{0,N}$ be the multiplicity of $\lambda_{0,N}$, being such that
    \begin{equation}
        \lambda_{0,N-1}<\lambda_{0,N}=\cdots=\lambda_{0,N+M-1}<\lambda_{0,N+M}.
    \end{equation}
    Then, we have that
    \begin{equation}
        \lambda_{\e,N+i-1}=\lambda_{0,N}+\gamma_{\e,N,i}+O(\delta_{\e,N}^2)+o(\tau_{\e,N}^2)\quad\text{as }\e\to 0^+,
    \end{equation}
    for any $i=1,\dots,M$, where $\delta_{\e,N}$ and $\tau_{\e,N}$ are as in \eqref{def_delta_e} and \eqref{def_tau_e}, respectively.
\end{theorem}
In view of \eqref{def_h_e}, in the case of simple limit eigenvalues, we have the following result; in particular, it recovers \cite[Theorem 3.1]{BS_eigen}.
\begin{corollary}\label{cor:simple}
    If $\lambda_{0,N}$ is simple and $\varphi_{0,N}\in E(\lambda_{0,N})$ is such that $\norm{\varphi_{0,N}}_{H_0}=1$, then
    \begin{equation}
        \lambda_{\e,N}=\lambda_{0,N}+\lambda_{0,N}(V_{\e,n,L_\e(\varphi_{0,N})},L_\e(\varphi_{0,N}))_{H_\e}+O(\delta_{\e,N}^2)+o(\tau_{\e,N}^2)\quad\text{as }\e\to 0^+.
    \end{equation}
\end{corollary}

\subsubsection{Organization of the paper}

In \Cref{sec_egin_var} we prove our main abstract result \Cref{thm:main} and in the subsequent sections we apply it to three more concrete examples (see the introduction for a quick description). In particular, in \Cref{sec_manifolds_measures} we deal variational eigenvalues depending on varying measures in the $L^2$ part; in \Cref{sec_conc_boundary} we apply the results of \Cref{sec_manifolds_measures} to the Neumann approximation of Steklov and Robin eigenvalues; finally, in \Cref{sec_sum_mani} we investigate the behavior of eigenvalues of manifold which are singularly perturbed by gluing a second small manifold to a fixed point.

\section{A first order expansion of the  eigenvalues variation}\label{sec_egin_var}

In the present section, we are going to detect the rate of convergence of perturbed eigenvalues to the limit one, thus proving our main result \Cref{thm:main}. We hereafter arbitrarily fix an index $N\in\N\setminus\{0\}$ and assume that
\begin{equation}\label{eq:ass_mult}
    \lambda_{0,N-1}<\lambda_{0,N}=\cdots=\lambda_{0,N+M_{0,N}-1}<\lambda_{0,N+M_{0,N}}.
\end{equation}
With respect to the notation introduced in \Cref{subsec:ass_stability} and \Cref{sec:main_results}, \textbf{we are here going to drop the dependence on $N$}, in order to slightly lighten the notation. For instance, we write $q_\e$ in place of $q_{\e,N}$, $J_{\e,v}$ in place of $J_{\e,N,v}$ etc.

As a first step, we state and prove the existence and uniqueness of solutions for an auxiliary problem, which is going to play a crucial role in the asymptotic expansion.
\begin{proposition}\label{prop_J_min}
For any $\e \in [0,1]$ and any $v \in \mc{F}_\e$ there exists a unique $V_{\e,v} \in  \mc{F}_\e$ solving the minimization problem 
\begin{equation}\label{prob_J_min}
\inf\{J_{\e,v}(u):u \in Z_\e+v\}.
\end{equation}
Furthermore, $V_{\e,v}$ is the unique solution to the equation 
\begin{equation}\label{eq_Ve}
\mc{E}^{(\e)}(V_{\e,v},u)=q_\e(u,v)\quad  \text{ for any } u \in Z_\e
\end{equation}
such that $V_{\e,v}-v \in Z_\e$ and the map $v \mapsto V_{\e,v}$ is linear.
Moreover, if $v\neq 0$ and
\begin{equation}\label{hp_Ve_not_0}
\text{either }\quad q_\e(\cdot,v) \not\equiv 0 \text{ in } Z_\e \quad \text{ or } \quad v \not \in Z_\e,
\end{equation}
then $V_{\e,v} \neq 0$. Finally, if $\lambda_{\e,i}\neq \lambda_{0,N}$ for all $i\in\N\setminus\{0\}$, then \eqref{hp_Ve_not_0} holds.
\end{proposition}
\begin{proof}
The proof is rather standard (see e.g. \cite[Proposition 2.5]{BS_eigen}, in which the authors prove everything but linearity in the setting of $L^2$ spaces). Nevertheless, for the sake of completeness and in order to emphasize the role of the assumptions, we provide it here. First of all, we let $w=v-u$, and we observe that
\begin{align}
    J_{\e,v}(u)&=J_{\e,v}(v-w)=\frac{1}{2}\overline{\mathcal{E}}^{(\e)}(v-w)-q_\e(v-w,v) \\
    &=\frac{1}{2}\overline{\mathcal{E}}^{(\e)}(v)+\frac{1}{2}\overline{\mathcal{E}}^{(\e)}(w)-\overline{\mathcal{E}}^{(\e)}(v,w) -\overline{\mathcal{E}}^{(\e)}(v)+\lambda_{0,N}\norm{v}_{H_\e}^2+\overline{\mathcal{E}}^{(\e)}(v,w)-\lambda_{0,N}(v,w)_{H_\e} \\
    &=\frac{1}{2}\overline{\mathcal{E}}^{(\e)}(w)-\lambda_{0,N}(v,w)_{H_\e}-\left(\frac{1}{2}\overline{\mathcal{E}}^{(\e)}(v)-\lambda_{0,N}\norm{v}_{H_\e}^2\right).
\end{align}
In particular,
\begin{equation}
    \inf\{J_{\e,v}(u):u \in Z_\e+v\}=\inf\left\{\frac{1}{2}\overline{\mathcal{E}}^{(\e)}(w)-\lambda_{0,N}(v,w)_{H_\e}\colon w\in Z_\e\right\}-\left(\frac{1}{2}\overline{\mathcal{E}}^{(\e)}(v)-\lambda_{0,N}\norm{v}_{H_\e}^2\right).
\end{equation}
We recall that the form $(\overline{\mathcal{E}}^{(\e)})^{1/2}$ is a norm on $Z_\e$ (see \Cref{rmk:pos_def}) and so we can apply the Lax-Milgram lemma, since the map $w\mapsto (v,w)_{H_\e}$ is linear and continuous (in view of \Cref{rmk:pos_def}). Hence, we derive the existence and uniqueness of a minimizer.
We now prove that the map $v \mapsto V_{\e,v}$ is linear. 
For any $v,w\in\mathcal{F}_\e$ and any $c\in\R$, by \eqref{eq_Ve} we have that
\begin{equation}
\mc{E}^{(\e)}(V_{\e,c v+w},u)=q_\e(u,cv+w)=c q_\e(u,v)+q_\e(u,w)= \mc{E}^{(\e)}(cV_{\e,v}+V_{\e,w},u)\quad\text{for all }u\in Z_\e,
\end{equation}
that is
\begin{equation}\label{eq:linear}
    \mc{E}^{(\e)}(V_{\e,c v+w}-cV_{\e,v}-V_{\e,w},u)=0\quad\text{for all }u\in Z_\e.
\end{equation}
At this point, since by \eqref{prob_J_min}
\begin{equation}
    V_{\e,c v+w}-cV_{\e,v}-V_{\e,w}=V_{\e,c v+w}-(cv+w)-c(V_{\e,v}-v)-(V_{\e,w}-w)\in Z_\e
\end{equation}
then we can let $u=V_{\e,c v+w}-cV_{\e,v}-V_{\e,w}$ in \eqref{eq:linear} and, thanks to \eqref{eq:pos_def}, obtain that
\begin{equation}
    0=\mc{E}^{(\e)}(V_{\e,c v+w}-cV_{\e,v}-V_{\e,w})\geq \norm{V_{\e,c v+w}-cV_{\e,v}-V_{\e,w}}_{H_\e}^2
\end{equation}
which implies that $V_{\e,c v+w}=cV_{\e,v}+V_{\e,w}$. 
Now, we observe that if $v\neq 0$ satisfies $v\not\in Z_\e$, then trivially $V_{\e,v}\neq 0$. On the other hand, if $v\neq 0$ and$q_\e(\cdot,v)\not\equiv 0$, we consider $u\in Z_\e$ such that $q_\e(u,v)>0$ and we observe that
\begin{equation}
    J_{\e,v}(tu)=\frac{t^2}{2}\overline{\mathcal{E}}^{(\e)}(u)-tq_\e(u,v)<0\quad\textnormal{for $t$ sufficiently small},
\end{equation}
so that $J_{\e,v}(V_{\e,v})<0$ which forces $V_{\e,v}\neq 0$.
Finally, we prove that if $\lambda_{\e,i}\neq \lambda_{0,N}$ for all $i\in\N\setminus\{0\}$, then \eqref{hp_Ve_not_0} holds. Let us assume by contradiction that $v\in Z_\e$ and 
\begin{equation}
    q_\e(u,v)=\overline{\mathcal{E}}^{(\e)}(u,v)-\lambda_{0,N}(u,v)_{H_\e}=0\quad  \text{ for any } u \in Z_\e.
\end{equation}
By the spectral theorem, we know that this holds if and only if $\lambda_{0,N}=\lambda_{\e,i}$ for some $i\in\N\setminus\{0\}$, thus reaching a contradiction.
\end{proof}

Second, we prove that our hypothesis \eqref{hp_Le} can be assumed to be uniform on each eigenspace.
\begin{lemma}\label{lemma:L_e}
    We have that \eqref{hp_Le} implies 
    \begin{equation}\label{hp:Le_equiv}\tag{S2'}
       \sup_{\substack{\varphi\in E(\lambda_{0,N}) \\\varphi\neq 0}}\left|\frac{\norm{L_\e(\varphi)}_{H_\e}}{\norm{\varphi}_{H_0}}-1\right|\to 0\quad\text{as }\e\to 0^+.
\end{equation}
\end{lemma}
\begin{proof}
Let $\varphi_\e\in E(\lambda_{0,N})\setminus\{0\}$ be the function achieving the supremum, which we assume to be normalized $\norm{\varphi_\e}_{H_0}=1$, and let us consider its expansion with respect to the eigenbasis, i.e.
\begin{equation}
    \varphi_\e=\sum_{j=1}^{M_{0,N}}\alpha_\e^j\varphi_{0,N,j}.
\end{equation}
In particular, we know that there exists $\psi\in E(\lambda_{0,N})$ with $\norm{\psi}_{H_0}=1$ such that $\varphi_\e\to \psi$ in $H_0$ as $\e\to 0^+$, up to a subsequence.
In view of \eqref{hp_Le}, we have that
\begin{equation}
    \norm{L_\e(\varphi_\e)}_{H_\e}^2=\sum_{j,k=1}^{M_{0,N}}\alpha_\e^j\alpha_\e^k(L_\e(\varphi_{0,N,j}),L_\e(\varphi_{0,N,k}))_{H_\e}\to \sum_{j=1}^{M_{0,N}}(\alpha^j)^2=\norm{\psi}_{H_0}^2=1
\end{equation}
as $\e\to 0^+$. However, since the limit is independent of the chosen subsequence, by the Urysohn principle we conclude that \eqref{hp:Le_equiv} holds.
\end{proof}

A crucial step in the proof of our main result consists in producing a good approximating space, which is $F_\e$ as in \eqref{eq:def_F_e}. In the following, we justify the fact that it is $M_{0,N}$-dimensional.
\begin{proposition}\label{prop_dim_Fe}
Let $P_\e$ be as in \eqref{eq:def_P_e} and $F_\e$ as in \eqref{eq:def_F_e}. The map $P_\e$ is linear and injective. In particular, the linear space $F_\e\subseteq Z_\e$ has dimension $M_{0,N}$ and $\{L_\e(\varphi_{0,N,i})-V_{\e,N,i}\}_{i=1,M_{0,N}}$ is a basis of $F_\e$.
\end{proposition}
\begin{proof}
In view of \Cref{prop_J_min}, $P_\e$ is linear. Moreover, by definition of $V_{\e,v}$ (see \Cref{prop_J_min}) we have that $v-V_{\e,v}\in Z_{\varepsilon}$
thus implying that $F_\e \subseteq Z_\e$. Hence, by definition \eqref{eq:def_F_e}, it is enough to show that $P_\e\colon E(\lambda_{0,N})\to Z_\e$ is injective for $\e$ small enough, that is there exists $\Bar{\e}>0$ such that for any $\e\in[0,\Bar{\e}]$ there holds
\begin{equation}
    P_\e(\varphi)=L_\e(\varphi)-V_{\e,L_\e(\varphi)}=0\implies \varphi=0.
\end{equation}
Indeed, once we have proved that $P_\e$ is bijective for small $\e>0$, it is not restrictive to suppose that, up to a rescaling of the parameter $\e$,   $P_\e$ is bijective on  $[0,1]$ (the case $\e=0$ is trivial).

Suppose by contradiction that there exists sequences $\e_i\to 0^+$ as $i\to\infty$ and $\varphi_i\in E(\lambda_{0,N})$ such that $\norm{\varphi_i}_{H_{0}}=1$ and $L_{\e_i}(\varphi_i)=V_{\e_i,L_{\e_i}(\varphi_i)}$ for all $i$. Since $\mathcal{E}^{(0)}(\varphi_i)=\lambda_{0,N}$ and since $E(\lambda_{0,N})$ is finite dimensional, then there exists $\varphi\in E(\lambda_{0,N})$ such that $\mathcal{E}^{(0)}(\varphi_i-\varphi)\to 0$ and, by \eqref{eq:pos_def}, $\norm{\varphi_i-\varphi}_{H_0}\to 0$, as $i\to\infty$, which implies, in particular, that $\norm{\varphi}_{H_0}=1$. Then, we can expand $\varphi_i$ and $\varphi$ in the eigenbasis as
\begin{equation}
    \varphi_i=\sum_{j=1}^{M_{0,N}}\alpha_i^j\varphi_{0,N,j} \quad\text{and}\quad
    \varphi=\sum_{j=1}^{M_{0,N}}\alpha^j\varphi_{0,N,j}
\end{equation}
and deduce that $\alpha_i^j\to \alpha^j$ as $i\to\infty$ and that 
$$\sum_{j=1}^{M_{0,N}}(\alpha^j)^2=\norm{\varphi}_{H_0}^2=1.$$ 
In view of the previous observations and \eqref{hp_Le} we get that
\begin{equation}
    \norm{L_{\e_i}\varphi_i}_{H_{\e_i}}^2=\sum_{j,k=1}^{M_{0,N}}\alpha_i^j\alpha_i^k(L_{\e_i}(\varphi_{0,N,j}),L_{\e_i}(\varphi_{0,N,k}))_{H_{\e_i}}\to \sum_{j,k=1}^{M_{0,N}}\alpha^j\alpha^k(\varphi_{0,N,j},\varphi_{0,N,k})_{H_0}=\norm{\varphi}_{H_0}^2=1
\end{equation}
as $i\to\infty$. On the other hand, by contradiction assumption and by linearity we have
\begin{equation}
    \norm{L_{\e_i}\varphi_i}_{H_{\e_i}}^2=\|V_{\e_i,L_{\e_i}(\varphi_i)}\|_{H_{\e_i}}^2\leq \delta_{\e_i}^2,
\end{equation}
where $\delta_\e$ is as in \eqref{def_delta_e}, but the right-hand side goes to $0$ as $i\to\infty$, in view of \eqref{hp_delta_e}. This is a contradiction.
\end{proof}

We now recall the following lemma, which is a refinement of a result by Y. Colin de Verdière, see \cite[Proposition 3.1]{dV86}. The proof of this version can be found in \cite{ACM_ramification}.

\begin{lemma}[Lemma on small eigenvalues \cite{dV86}]\label{Lem:Small_Eigenvalues}
		Let $(\mathcal{H},(\cdot,\cdot)_{\mathcal{H}})$ be a real
		Hilbert space and let $\mathcal{D}\subseteq\mathcal{H}$ be a dense
		subspace. Let $q\colon \mathcal{D}\times\mathcal{D}\to\R$ be a
		symmetric bilinear form, such that $q$ is semi-bounded from below,
		$q$ admits an increasing sequence of eigenvalues
		$\{\nu_i\}_{i\in\N}$, 
		and there exists an orthonormal basis $\{g_i\}_{i\in\N}$ of
		$\mathcal H$ such that $g_i\in\mathcal D$ is an eigenvector of $q$
		associated to the eigenvalue $\nu_i$, i.e.  $q(g_i,v)=\nu_i(g_i,v)_{\mathcal{H}}$
		for all $i\geq1$ and $v\in \mathcal D$.
		
		Let $M\in \N\setminus\{0\}$ and $F\subseteq \mathcal{D}$ be
		a $M$-dimensional subspace of $\mathcal D$.
		
		Assume that there exists $N\in\N$ and
		$\gamma>0$ such that
		\begin{itemize}
			\item[\rm (H1)] $\nu_i\leq -\gamma$ for all $i\leq N-1$,
			$|\nu_i|\leq \gamma$ for all $i=N,\dots,N+M-1$, and 
			$\nu_i\geq \gamma$ for all $i\geq N+M$;
			\item[\rm (H2)]
			$0<\delta:=\sup\{|q(u,v)|\colon u\in \mathcal{D},~v\in
			F,~\norm{u}_{\mathcal{H}}=\norm{v}_{\mathcal{H}}=1\}<\gamma/\sqrt{2}$.
		\end{itemize}
		If $\{\xi_i\}_{i=1,\dots,M}$ are the eigenvalues (in
		ascending order) of $q$ restricted to $F$, we have
		\begin{equation}\label{eq:CdV_th1}
			|\nu_{N+i-1}-\xi_i|\leq\frac{4\delta^2}{\gamma}
			\quad\text{for all }i=1,\dots,M.
		\end{equation}
		Moreover, if
		$\Pi\colon\mathcal{D}\to
		\mathrm{span}\,\{g_N,\dots,g_{N+M-1}\}$ denotes the
		orthogonal projection onto the subspace of
		$\mathcal D$  spanned by $\{g_N,\dots,g_{N+M-1}\}$,
		we have
		\begin{equation}
			\frac{\norm{v-\Pi
					v}_{\mathcal{H}}}{\norm{v}_{\mathcal{H}}}
			\leq \frac{\sqrt{2}\delta}{\gamma}\quad\text{for all }v\in F.
		\end{equation}
	\end{lemma}

Essentially, we are going to apply \Cref{Lem:Small_Eigenvalues} with particular choices for the elements appearing in its assumptions, which will then lead to the proof of our main result \Cref{thm:main}. Let $q_\e$ be as in \eqref{def_qe} and $F_\e$ as in \eqref{eq:def_F_e}. Then 
\begin{equation}
  {q_\e}_{|{F_\e \times F_\e}}: F_\e \times F_\e \to \R  
\end{equation}
is a symmetric bilinear form on the $M_{0,N}$-dimensional space $F_\e$.
Hence, it admits exactly $M_{0,N}$ real eigenvalues, counted with multiplicity, which we denote with $\{\xi_{\e,i}\}_{i=1, \dots, M_{0,N}}$ (assumed to be in increasing order).
The first step in the proof of \Cref{thm:main} is the following proposition, which is a consequence of \Cref{Lem:Small_Eigenvalues}.

\begin{proposition}\label{prop_estimate_eigen}
There exist $\e_0 \in [0,1]$ and a positive constant $C_N>0$, not depending on $\e$, such that for any $\e \in (0,\e_0]$ and any $i=1,...,M_{0,N}$ there holds
\begin{equation}\label{ineq_esetimate_eigen}
|\lambda_{\e,i+N-1}-\lambda_{0,N}-\xi_{\e,i}|\leq C_N\delta_\e^2,
\end{equation}  
where $\delta_\e$ is as in \eqref{def_delta_e}.
\end{proposition}
\begin{proof}
We apply \Cref{Lem:Small_Eigenvalues}, for fixed $\e\in[0,1]$, with the following choices
\begin{equation}
        \mathcal{H}=\mathcal{Z}_\e, \quad \mathcal{D}=Z_\e,\quad F=F_\e,\quad q={q_\e}_{| Z_\e \times Z_\e}.
\end{equation}
By assumption on $\mathcal{Z}_\e$, $Z_\e$ and $\mathcal{E}^{(\e)}$ and in view of \Cref{prop_dim_Fe}, these choices satisfy the assumptions of \Cref{Lem:Small_Eigenvalues}.  Indeed, thanks to  \eqref{def_qe} and \eqref{eq:pos_def}, we have that
\begin{equation}
 q_\e (u) =q_\e(u,u)= \mc{E}^{(\e)} (u)-\lambda_{0,N} \norm{u}_{H_\e}^2 \ge (1-\lambda_{0,N}) \norm{u}_{H_\e}^2 \quad \text { for any } u \in  Z_\e.
\end{equation}
Moreover, we have
\begin{equation}
    \nu_n=\lambda_{\e,n}-\lambda_{0,N}\quad\textnormal{for all }n\in\N\setminus\{0\}.
\end{equation}
Furthermore, in view of \eqref{eq:ass_mult}, letting 
\begin{equation}\label{def_gamma}
\gamma:=\frac{1}{2}\min\{\la_{0,N}-\la_{0,N-1},\la_{0,M_{0,N}+N} -\la_{0,N}\}>0,
\end{equation}
thanks to \eqref{hp_stability} we have that there exists $\e_0 \in (0,1]$ such that 
\begin{equation}
\la_{\e,N-1} -\la_{0,N} \le -\frac{1}{2}(\la_{0,N}-\la_{0,N-1})\leq -\gamma \quad \text{ and }  \quad \la_{\e, M_{0,N}+N} -\la_{0,N} \ge \frac{1}{2}(\la_{0,M_{0,N}+N} -\la_{0,N})\geq \gamma
\end{equation}
and 
\begin{equation}
|\la_{\e,n} -\la_{0,N}| \le \frac{1}{2}\min\{\la_{0,N}-\la_{0,N-1},\la_{0,M_{0,N}+N} -\la_{0,N}\}=\gamma
\end{equation}
for any $\e \in (0,\e_0]$, so that (H1) holds.
Furthermore, letting $L_\e(\varphi)- V_{\e, L_\e(\varphi)} \in F_\e$ where $\varphi \in E(\lambda_{0,N})$ and $ w \in Z_\e$, by \eqref{def_qe}  and \eqref{eq_Ve} we have that
\begin{align}
q_\e (L_\e(\varphi)- V_{\e,L_\e(\varphi)}, w)&= q_\e(L_\e(\varphi),w)-\mc{E}^{(\e)}(V_{\e,L_\e(\varphi)}, w)+\la_{0,N}\ps{H_\e}{V_{\e,L_\e(\varphi)}}{w} \\
&=\la_{0,N}\ps{H_\e}{V_{\e,L_\e(\varphi)}}{w}.
\end{align}
Now, thanks to the Cauchy-Schwarz inequality, we derive that
\begin{equation}\label{eq:q_H3}
|q_\e (L_\e(\varphi)- V_{\e,L_\e(\varphi)}, w)| \le  \lambda_{0,N}\norm{V_{\e,L_\e(\varphi)}}_{H_\e}\norm{w}_{H_\e} \leq \lambda_{0,N}\delta_\e \norm{\varphi}_{H_0}\norm{w}_{H_\e},
\end{equation}
where $\delta_\e$ is as in \eqref{def_delta_e}.
At this point, we observe that 
\begin{equation}
    \frac{\norm{\varphi}_{H_0}}{\norm{P_\e\varphi}_{H_\e}}\leq \frac{\norm{\varphi}_{H_0}}{\norm{L_\e(\varphi)}_{H_\e}-\norm{V_{\e,L_\e(\varphi)}}_{H_\e}}\leq \frac{\norm{\varphi}_{H_0}}{\norm{L_\e(\varphi)}_{H_\e}-\delta_\e\norm{\varphi}_{H_0}}= \frac{1}{\frac{\norm{L_\e(\varphi)}_{H_\e}}{\norm{\varphi}_{H_0}}-\delta_\e}\leq 2
\end{equation}
for $\varphi\in E(\lambda_{0,N})$ and $\e$ sufficiently small (independent of $\varphi$), where, in the last step, we used \eqref{hp:Le_equiv} and \eqref{hp_delta_e}. Combining this estimate with \eqref{eq:q_H3} we deduce that
\begin{equation}
    |q_\e (P_\e(\varphi), w)|\leq 2\lambda_{0,N}\delta_\e \norm{P_\e(\varphi)}_{H_\e}\norm{w}_{H_\e}
\end{equation}
for all $\e$ sufficiently small and all $\varphi\in E(\lambda_{0,N})$ and $w\in Z_\e$. By the surjectivity of $P_\e$, we deduce that 
\begin{equation}
    \delta=\sup\{|q(u,v)|\colon u\in \mathcal{D},~v\in
			F,~\norm{u}_{\mathcal{H}}=\norm{v}_{\mathcal{H}}=1\}\leq 2\lambda_{0,N}\delta_\e 
\end{equation}
and so, in view of \eqref{hp_delta_e}, we have that, by taking $\e_0$ sufficiently small, $2\lambda_{0,N}\delta_\e \le \frac{\gamma}{\sqrt{2}}$ for any $\e \in (0,\e_0]$, so that (H2) in \Cref{Lem:Small_Eigenvalues} holds true. Hence, we are in a position to apply \Cref{Lem:Small_Eigenvalues} and conclude that \eqref{ineq_esetimate_eigen} holds.
\end{proof}

We now want to relate the eigenvalues $\{\xi_{\e,n}\}_{n=1, \dots, M_{0,N}}$ of ${q_\e}_{|F_\e \times F_\e}$ with the eigenvalues $\{\gamma_{\e,n}\}_{n=1,\dots,M_{0,N}}$ of $h_\e$ (defined in \eqref{def_h_e}). In order to do this, we pass through the eigenvalues of the symmetric bilinear form
\begin{equation}\label{eq:def_r_e}
    r_\e:=q_\e\circ P_\e:E(\lambda_{0,N}) \times E(\lambda_{0,N}) \to \R
\end{equation}
which, in view of \Cref{prop_dim_Fe}, is known to possess $M_{0,N}$ real eigenvalues $\{\mu_{\e,n}\}_{n=1,\dots,M_{0,N}}$ counted with multiplicity.
The advantage is that the eigenvalues $\{\mu_{\e,n}\}_{n=1,\dots,M_{0,N}}$ (and then $\{\gamma_{\e,n}\}_{n=1,\dots,M_{0,N}}$) are somehow easier to compute (asymptotically), since the bilinear form $r_\e$ can be made explicit in terms of a limit eigenbasis, as we show in the next lemma.

\begin{lemma}\label{lemma:q_circ_P}
Let $r_\e$ be as in \eqref{eq:def_r_e}. For any $\varphi,\psi \in E(\lambda_{0,N})$ we have
\begin{equation}\label{eq_qe_Pe}
 r_\e(\varphi, \psi)=\lambda_{0,N} (V_{\e,L_\e(\varphi)},L_\e(\psi)-V_{\e,L_\e(\psi)})_{H_\e}=\lambda_{0,N} (V_{\e,L_\e(\psi)},L_\e(\varphi)-V_{\e,L_\e(\varphi)})_{H_\e}.
\end{equation}
or
\begin{equation}\label{eq:r_e_2}
    \begin{aligned}
        r_\e(\varphi,\psi)=&\overline{\mathcal{E}}^{(\e)}(L_\e(\varphi),V_{\e,L_\e(\psi)})+\overline{\mathcal{E}}^{(\e)}(L_\e(\psi),V_{\e,L_\e(\varphi)}) \\
        &-\overline{\mathcal{E}}^{(\e)}(V_{\e,L_\e(\varphi)},V_{\e,L_\e(\psi)})-\lambda_{0,N}(V_{\e,L_\e(\varphi)},V_{\e,L_\e(\psi)})_{H_\e}-D_\e(\varphi,\psi),
        \end{aligned}
\end{equation}
where
\begin{equation}
    D_\e(\varphi,\psi):=\overline{\mathcal{E}}^{(\e)}(L_\e(\varphi),L_\e(\psi))-\lambda_{0,N}(L_\e(\varphi),L_\e(\psi))_{H_\e}.
\end{equation}
In particular, if $\norm{\varphi}_{H_0}=\norm{\psi}_{H_0}=1$ then
\begin{equation}\label{eq_r_e_delta}
    r_\e(\varphi, \psi)=h_\e(\varphi,\psi)+O(\delta_\e^2)
\end{equation}
and
\begin{equation}\label{eq_tau_delta}
r_\e(\varphi, \psi)=O(\tau_\e^2)+O(\delta_\e^2),
\end{equation}
as $\e\to 0^+$, where $\delta_\e$ is as in \eqref{def_delta_e}, $\tau_\e$ as in \eqref{def_tau_e} and $h_\e$ as in \eqref{def_h_e}.
\end{lemma}
\begin{proof}
For any $\varphi,\psi \in E(\lambda_{0,N})$  we have
\begin{align}
q_\e (L_\e(\varphi) - V_{\e,L_\e(\varphi)}, L_\e(\psi) - V_{\e,L_\e(\psi)}) \notag =&q_\e(L_\e(\varphi),L_\e(\psi))+q_\e(V_{\e,L_\e(\varphi)}, V_{\e,L_\e(\psi)}) \notag\\
&-q_\e(L_\e(\varphi), V_{\e,L_\e(\psi)})-q_\e(L_\e(\psi),V_{\e,L_\e(\varphi)}).\label{Eq:App1_q}
\end{align}
Observing that, by \eqref{eq_Ve},
\begin{equation}\label{eq:EL_E_e}
    \mc{E}^{(\e)}(V_{\e,L_\e(\varphi)}, V_{\e,L_\e(\psi)}-L_\e(\psi))=q_\e(L_\e(\varphi), V_{\e,L_\e(\psi)}-L_\e(\psi)),
\end{equation}
it follows that
\begin{align}
    \mc{E}^{(\e)}(V_{\e,L_\e(\varphi)}, L_\e(\psi))&=\mc{E}^{(\e)}(V_{\e,L_\e(\varphi)}, V_{\e,L_\e(\psi)}) - q_\e(L_\e(\varphi), V_{\e,L_\e(\psi)}-L_\e(\psi))\\
    &=\mc{E}^{(\e)}(V_{\e,L_\e(\varphi)}, V_{\e,L_\e(\psi)}) - q_\e(L_\e(\varphi), V_{\e,L_\e(\psi)}) + q_\e(L_\e(\varphi), L_\e(\psi))\\
    &=q_\e(V_{\e,L_\e(\varphi)}, V_{\e,L_\e(\psi)}) - q_\e(L_\e(\varphi), V_{\e,L_\e(\psi)}) + q_\e(L_\e(\varphi), L_\e(\psi)) \\
    &\quad+ \lambda_{0,N} (V_{\e,L_\e(\varphi)}, V_{\e,L_\e(\psi)})_{H_\e},
\end{align}
thanks to  \eqref{def_qe}. Hence, by \eqref{def_qe} and \eqref{Eq:App1_q}, we get
\begin{align}
     q_\e (L_\e(\varphi) - V_{\e,L_\e(\varphi)}, L_\e(\psi) - V_{\e,L_\e(\psi)})&=\mc{E}^{(\e)}(L_\e(\psi),V_{\e,L_\e(\varphi)})-\lambda_{0,N}  (V_{\e,L_\e(\varphi)}, V_{\e,L_\e(\psi)})_{H_\e} \\
     &\quad - q_\e(L_\e(\psi), V_{\e,L_\e(\varphi)})\\
     &=\lambda_{0,N} (L_\e(\psi), V_{\e,L_\e(\varphi)})_{H_\e}-\lambda_{0,N}  (V_{\e,L_\e(\varphi)}, V_{\e,L_\e(\psi)})_{H_\e}\\
     &=\lambda_{0,N} (V_{\e,L_\e(\varphi)},L_\e(\psi)-V_{\e,L_\e(\psi)})_{H_\e}.
\end{align}
Hence, we have proved the first equality in \eqref{eq_qe_Pe}. Exchanging the roles of $\varphi$ and $\psi$, we obtain the second one. In order to get \eqref{eq:r_e_2} we just rearrange the terms exploiting \eqref{eq:EL_E_e} and the definition of $q_\e$.
Now, for any $\varphi,\psi \in E(\la_{0,N})$ with $\norm{\varphi}_{H_0}=\norm{\psi}_{H_0}=1$, by \eqref{eq_qe_Pe} and Cauchy-Schwarz inequality we get
\begin{equation}\label{eq:r_e}
  |r_\e(\varphi,\psi)-\lambda_{0,N}(V_{\e,L_\e(\varphi)},L_\e(\psi))_{H_\e}|\leq \lambda_{0,N}\norm{V_{\e,L_\e(\psi)}}_{H_\e}\norm{V_{\e,L_\e(\varphi)}}_{H_\e}=O(\delta_\e^2), \quad  \text{ as } \e \to 0^+,
\end{equation}
thus proving \eqref{eq_r_e_delta}. Finally, we observe that, by simple algebra and by definition of $\tau_\e^2$, we have
\begin{align}
   \la_{0,N}|(V_{\e,L_\e(\varphi)},L_\e(\psi))_{H_\e}|&=\frac{\la_{0,N}}{4}\left| (V_{\e,L_\e(\varphi+\psi)},L_\e(\varphi+\psi))_{H_\e} -(V_{\e,L_\e(\varphi-\psi)},L_\e(\varphi-\psi))_{H_\e}\right| \\
    &\leq \frac{\la_{0,N}}{4}\left( |(V_{\e,L_\e(\varphi+\psi)},L_\e(\varphi+\psi))_{H_\e}| 
    +|(V_{\e,L_\e(\varphi-\psi)},L_\e(\varphi-\psi))_{H_\e}|\right) \\
    &\leq \frac{\tau_\e^2}{4}\left( \norm{\varphi+\psi}_{H_0}^2+\norm{\varphi-\psi}_{H_0}^2\right)\leq 2\tau_\e^2.
\end{align}
Combining this fact with \eqref{eq:r_e} we get \eqref{eq_tau_delta}, thus concluding the proof.
\end{proof}

\begin{remark}\label{rmk:symmetry}
    We observe that, from \eqref{eq_qe_Pe}, we can deduce that the bilinear form
    \begin{equation}
        (\varphi,\psi)\mapsto (V_{\e,L_\e(\psi)},L_\e(\varphi))_{H_\e}
    \end{equation}
    is symmetric, since 
    \begin{equation}
        (\varphi,\psi)\mapsto r_\e(\varphi,\psi)\quad\text{and}\quad(\varphi,\psi)\mapsto (V_{\e,L_\e(\psi)},V_{\e,L_\e(\varphi)})_{H_\e}
    \end{equation}
    are so.
\end{remark}

We are now able to asymptotically relate the eigenvalues $\{\xi_{\e,n}\}_{n=1,\dots,M_{0,N}}$ and $\{\gamma_{\e,n}\}_{n=1,\dots,M_{0,N}}$. Even if the proof is rather standard (see e.g. \cite[Appendix C]{ACM_ramification}), since it is an important point in our argument, we provide the details.

\begin{lemma}\label{lemma:xi_mu}
For any $n=1,\dots,M_{0,N}$
\begin{equation}\label{eq_xi_mu}
\xi_{\e,n}=\gamma_{\e,n}+O(\delta_\e^2)+o(\tau_\e^2) \quad\text{as }\e\to 0^+,
 \end{equation}
 where $\delta_\e$ is as in \eqref{def_delta_e} and  $\tau_\e$ as in \eqref{def_tau_e}.
\end{lemma}
\begin{proof}
Let $A_\e$ be the matrix associated with ${q_\e}_{|F_\e \times F_\e}$ with respect to the basis $\{f_i\}_{i=1, \dots M_{0,N}}$, i.e.
\begin{equation}
    (A_\e)_{ij}:=q_\e(f_i,f_j),
\end{equation}
where
\begin{equation}
    f_i:=P_\e(\varphi_{0,N,i})=L_\e(\varphi_{0,N,i})-V_{\e,L_\e(\varphi_{0,N,i})}
\end{equation}
and $\{\varphi_{0,N,i}\}_{i=1,\dots,M_{0,N}}$ is an $H_0$-orthonormal basis of $E(\lambda_{0,N})$. In other words, $A_\e$ is the matrix associated with the bilinear form $r_\e$ with respect to the basis $\{\varphi_{0,N,i}\}_{i=1,\dots,M_{0,N}}$. Moreover, let $G_\e$ be the Gram matrix associated with $\{f_i\}_{i=1, \dots M_{0,N}}$, i.e.
\begin{equation}
    (G_\e)_{ij}:=(f_i,f_j)_{H_\e}.
\end{equation}
Now let $v\in F_\e$ be an eigenfunction of ${q_\e}_{|F_\e\times F_\e}$ corresponding to some eigenvalue $\xi_\e\in\R$ and let $w\in F_\e$ be arbitrary. Let us write them in the basis $\{f_i\}_{i=1, \dots M_{0,N}}$, that is
\begin{equation}
    v=\sum_{i=1}^{M_{0,N}}a_i f_i\quad\textnormal{and}\quad w=\sum_{i=1}^{M_{0,N}}b_i f_i.
\end{equation}
The equation
\begin{equation}
    q_{\e}(v,w)=\xi_\e(v,w)_{H_\e}
\end{equation}
then becomes
\begin{equation}
    \sum_{i,j=1}^{M_{0,N}}a_i b_j q_\e(f_i,f_j)=\xi_\e \sum_{i,j=1}^{M_{0,N}}a_i b_j(f_i,f_j)_{H_\e}
\end{equation}
that is
\begin{equation}\label{eq:eigen_alpha_beta}
    \bm{b}^T A_\e \bm{a}=\xi_\e \bm{b}^T G_\e\bm{a},
\end{equation}
where $\bm{a}=(a_1,\dots,a_{M_{0,N}})$ and $\bm{b}=(b_1,\dots,b_{M_{0,N}})$. Since $\bm{b}$ is arbitrary, \eqref{eq:eigen_alpha_beta} holds if and only if
\begin{equation}
    G_\e^{-1}A_\e \bm{a}=\xi_\e \bm{a}.
\end{equation}
Hence, we proved that $\{\xi_{\e,n}\}_{n=1,\dots,M_{0,N}}$ coincide with the family of eigenvalues of $B_\e:=G_\e^{-1}A_\e$. Analogously, one can prove that $A_\e$ has eigenvalues $\{\mu_{\e,n}\}_{n=1, \dots M_{0,N}}$. Furthermore, in view of \eqref{eq_tau_delta}, we have that
\begin{equation}\label{eq:A_asy}
(A_\e)_{ij}=O(\tau_\e^2)+O(\delta_\e^2)\quad \text{ as } \e \to 0^+
\end{equation}
while, in view of \eqref{hp_Le}, we get that
\begin{equation}
    (G_\e)_{ij}=\delta_{ij}+o(1)\quad\textnormal{as }\e\to 0^+,
\end{equation}
so that, by basic linear algebra, there holds
\begin{equation}\label{eq:C-1_asy}
    (G_\e^{-1})_{ij}=\delta_{ij}+o(1)\quad\textnormal{as }\e\to 0^+.
\end{equation}
Then, by \eqref{eq:A_asy} and \eqref{eq:C-1_asy} we get
\begin{equation}
    (B_\e)_{ij}=(G_\e^{-1})_{ii}(A_\e)_{ij}+\sum_{k\neq i}(G_\e^{-1})_{ik}(A_\e)_{kj}=(A_\e)_{ij}+o(\tau_\e^2)+o(\delta_\e^2),\quad\textnormal{as }\e\to 0^+.
\end{equation}
Then, in view of \eqref{eq_r_e_delta}, the thesis follows.
\end{proof}

\begin{proof}[Proof of \Cref{thm:main} and \Cref{cor:simple}]
 Combining \Cref{prop_estimate_eigen} and \Cref{lemma:xi_mu}, we complete the proof.   
\end{proof}

To conclude the section, we check that the asymptotic expansion of \Cref{thm:main} and of \Cref{cor:simple} does not depend on the shift of the bilinear form $\mathcal{E}^{(\e)}$, even if the bilinear form $h_{\e,N}$ as in \eqref{def_h_e} does.

\begin{lemma}
    For any $\kappa\in\R$, let
    \begin{equation}\label{def_Eek}
        \mathcal{E}^{(\e)}_\kappa(u,v):=\mathcal{E}^{(\e)}(u,v)+\kappa(u,v)_{H_\e}\quad\textnormal{and}\quad \overline{\mathcal{E}}^{(\e)}_\kappa:=\mathcal{E}^{(\e)}_{\kappa\,|Z_\e}.
    \end{equation}
    For $\e\in[0,1]$ and $v\in\mathcal{F}_\e$, we let $V_{\e,v}^\kappa\in Z_\e+v$ be the unique minimizer of
    \begin{equation}
        \inf\left\{\frac{1}{2}\mathcal{E}^{(\e)}_\kappa(u)-q_\e(v,u)\colon u\in Z_\e+v\right\}.
    \end{equation}
    Let $C_\e\in\R$ be the lower bound of $\mathcal{E}^{(\e)}$ as in \Cref{subsec:ass_functional}. Then, for any $-C_\e<\kappa_1\leq\kappa_2$, we have that
    \begin{equation}\label{shiftsplit}
        (\lambda_{0,N}+\kappa_1)(V_{\e,v}^{\kappa_1},v)_{H_\e}-(\lambda_{0,N}+\kappa_2)(V_{\e,v}^{\kappa_2},v)_{H_\e}=(\kappa_1-\kappa_2)(V_{\e,v}^{\kappa_1},V_{\e,v}^{\kappa_2})_{H_\e}
    \end{equation}
    and there holds
    \begin{equation}\label{shifteng}
        \norm{V_{\e,v}^{\kappa_2}}_{H_\e}\leq \norm{V_{\e,v}^{\kappa_1}}_{H_\e}.
    \end{equation}
    In particular,
    \begin{equation}\label{shiftegn}
        \max_{\substack{\varphi\in E(\lambda_{0,N}) \\ \norm{\varphi}_{H_0}=1}}\left| (\lambda_{0,N}+\kappa_1)(V_{\e,L_\e(\varphi)}^{\kappa_1},L_\e(\varphi))_{H_\e}-(\lambda_{0,N}+\kappa_2)(V_{\e,L_\e(\varphi)}^{\kappa_2},L_\e(\varphi))_{H_\e}\right|\leq |\kappa_2-\kappa_1|(\delta_{\e}^{\kappa_1})^2,
    \end{equation}
    where $\delta_\e^\kappa$ is defined as in \eqref{def_delta_e} by replacing $V_{\e,n,L_\e(\varphi)}$ with $V^\kappa_{\e,N,L_\e(\varphi)}$.
\end{lemma}
\begin{proof}
First we notice that $V_{\e,v}^{\kappa_1}$ and $V_{\e,v}^{\kappa_2}$ satisfy the following equations
\begin{align}
&\mathcal{E}^{(\e)}(V_{\e,v}^{\kappa_1},u)+\kappa_1(V_{\e,v}^{\kappa_1},u)_{H_\e}=q_\e(v,u)\quad\textnormal{for all }u\in Z_\e,\label{eqv:shift1}\\
&\mathcal{E}^{(\e)}(V_{\e,v}^{\kappa_2},u)+\kappa_2(V_{\e,v}^{\kappa_2},u)_{H_\e}=q_\e(v,u)\quad\textnormal{for all }u\in Z_\e.\label{eqv:shift2}
\end{align}
Testing  \eqref{eqv:shift1} and   \eqref{eqv:shift2} with $u=V_{\e,v}^{\kappa_2}-v$ and $u=V_{\e,v}^{\kappa_1}-v$ respectively, we can deduce \eqref{shiftsplit} subtracting \eqref{eqv:shift2} from \eqref{eqv:shift1} and  taking  \eqref{def_qe} into account. By minimality, we also have
\begin{align}
&\frac{1}{2}\mathcal{E}^{(\e)}_{\kappa_1}(V_{\e,v}^{\kappa_1})-q_\e(v,V_{\e,v}^{\kappa_1})\leq \frac{1}{2}\mathcal{E}^{(\e)}_{\kappa_1}(V_{\e,v}^{\kappa_2})-q_\e(v,V_{\e,v}^{\kappa_2}),\\
&\frac{1}{2}\mathcal{E}^{(\e)}_{\kappa_2}(V_{\e,v}^{\kappa_2})-q_\e(v,V_{\e,v}^{\kappa_2})\leq \frac{1}{2}\mathcal{E}^{(\e)}_{\kappa_2}(V_{\e,v}^{\kappa_1})-q_\e(v,V_{\e,v}^{\kappa_1}).
\end{align}
Summing up the two inequalities,
\begin{equation}
\mathcal{E}^{(\e)}_{\kappa_1}(V_{\e,v}^{\kappa_1})+\mathcal{E}^{(\e)}_{\kappa_2}(V_{\e,v}^{\kappa_2}) 
\le \mathcal{E}^{(\e)}_{\kappa_1}(V_{\e,v}^{\kappa_2})+\mathcal{E}^{(\e)}_{\kappa_2}(V_{\e,v}^{\kappa_1}),
\end{equation}
thus we obtain \eqref{shifteng} by \eqref{def_Eek}. Finally,   \eqref{shiftegn} follows from  \eqref{shiftsplit} and the definition of $\delta_\e^2$, see \eqref{def_delta_e}.
\end{proof}

\section{Weighted eigenvalue problems with varying measures}\label{sec_manifolds_measures}
In the present section, we deal with a family of eigenvalue problems for the Laplace-Beltrami operator on Riemannian manifolds, in which the $L^2$ terms may be weighted with Radon measures. This setting allows to take into account a wide range of singular perturbations, among which we find the Neumann approximation of the Steklov problem, which we are going to treat in \Cref{subsec_steklov}.

Let us now describe the problems we here consider. Let $(\overline{\Omega},g)$ be a compact, connected Riemannian manifold with Lipschitz boundary $\partial \Omega$ and denote by $\Omega$ the interior of $\overline{\Omega}$, i.e. $\Omega=\overline{\Omega}\setminus \partial \Omega$.
We consider the corresponding Lebesgue space $L^2(\Omega)$ and the Sobolev space $H^1(\Omega)$, which can be defined as the completion of $C^\infty(\overline{\Omega})$ with respect to the norm $\norm{\cdot}_{H^1(\Omega)}$ induced by the scalar product
\begin{equation}
    (u,v)_{H^1(\Omega)}:=\int_\Omega \left[g(\nabla^g u,\nabla^g v)+uv\right] \dvol^g,\quad\textnormal{defined for }u,v\in C^\infty(\overline{\Omega}).
\end{equation}
For the sake of simplicity of notation, in the following we may write
\begin{equation}
    |\nabla^g u|^2\quad\textnormal{in place of}\quad g(\nabla^g u,\nabla^g u).
\end{equation}
Let us now take two families of positive Radon measures $\bm{\alpha}=\{\alpha_\e\}_{\e \in [0,1]}$ and $\bm{\beta}=\{\beta_\e\}_{\e \in [0,1]}$ on $\overline{\Omega}$ and, for any $\e\in[0,1]$, let us consider the following eigenvalue problem
\begin{equation}\label{eq:eigen_measure}
    -\Delta^g \varphi+(\alpha_\e+\beta_\e) \varphi=\lambda\beta_\e\varphi\quad\textnormal{in }\overline{\Omega},
\end{equation}
which, we remark, is the same as 
\begin{equation}
    -\Delta^g \varphi+\alpha_\e \varphi=(\lambda-1)\beta_\e\varphi\quad\textnormal{in }\overline{\Omega}.
\end{equation}
Our aim is to study the asymptotic behavior of the spectrum of \eqref{eq:eigen_measure}, as $\e\to 0^+$. Our path consists of, first, giving sense to \eqref{eq:eigen_measure} by encoding it in a suitable functional framework and, second, proving that we are allowed to apply our main abstract result \Cref{thm:main}. In order to pursue this approach, we first restrict the admissible families of measures. 
\begin{definition}\label{def:adm}
    We say that a family of positive Radon measures $\bm{\nu}=\{\nu_\e\}_{\e\in[0,1]}$ on $\overline{\Omega}$  is \emph{admissible}, and we write $\bm{\nu}\in \mathcal{A}(\overline{\Omega})$ if it satisfies the following:
    \begin{enumerate}[
    label=\textbf{($A$.\arabic*)}, 
    ref=$A$.\arabic*, 
    leftmargin=*]
    \item for any $\delta>0$ there exists $C(\delta)>0$ for which
                \begin{equation}
                \int_{\overline \Omega}|u|^2 \,\mathrm{d} \nu_\varepsilon \leq \delta\int_{\Omega} |\nabla^g u|^2 \dvol^g + C(\delta)\int_{\Omega}|u|^2 \dvol^g, 
                \end{equation}
                for any $u \in H^1(\Omega)$ and for any $\e \in [0,1]$; \label{A1}
    \item for any $u\in C^\infty(\overline{\Omega})$ there holds
                \begin{equation}
                     \int_{\overline \Omega} u \, \mathrm{d} \nu_\e \to \int_{\overline \Omega} u \, \mathrm{d} \nu_0 \quad\textnormal{as }\e\to 0^+.
                \end{equation}\label{A2}
    \end{enumerate}
\end{definition}
We assume that $\bm{\alpha}=\{\alpha_\e\}_{\e\in[0,1]}$ and $\bm{\beta}=\{\beta_\e\}_{\e\in[0,1]}$ are admissible, i.e.
\begin{equation}\label{eq:ass_alpha_beta}
    \bm{\alpha},\bm{\beta}\in \mathcal{A}(\overline{\Omega}).
\end{equation}
In particular, we have that $\bm{\alpha}+\bm{\beta}\in \mathcal{A}(\overline{\Omega})$. Moreover, we also assume that 
\begin{equation}\label{beta3'}
    \beta_0(\overline{\Omega})>0.
\end{equation}
This is not restrictive since, if $\beta_0(\overline{\Omega})=0$ then $H_0=L^2(\overline{\Omega},\beta_0)=\{0\}$ and the whole limit problem becomes trivial. Moreover, \eqref{beta3'} implies that
\begin{equation}\label{beta3}
    \beta_\e(\overline{\Omega})>0\quad\text{for any }\e\in[0,1].
\end{equation}
Indeed, if there exists a sequence $\e_n\to 0^+$ such that $\beta_{\e_n}(\overline{\Omega})\to 0$ as $n\to\infty$, then \ref{A2} yields $\beta_0(\overline{\Omega})=0$, contradicting \eqref{beta3'}. Then, by renaming $\e$, we have that \eqref{beta3} holds.

We notice that we can choose $\alpha_\e\equiv 0$, while \eqref{beta3} forces $\beta_\e$ to be nontrivial. We also emphasize that, in view of \ref{A1}, we are essentially  assuming that the families $\{\alpha_\e\}_{\e\in[0,1]}$ and $\{\beta_\e\}_{\e\in[0,1]}$ support a (compact) embedding-type inequality \emph{uniformly} in $\e\in[0,1]$. 

For any positive nontrivial Radon measure $\nu$ on $\overline{\Omega}$, we let $H_\nu^1(\Omega)$ be the completion of $C^\infty(\overline{\Omega})$ with respect to the norm $\norm{\cdot}^2_{H_\nu^1(\Omega)}$ induced by the scalar product
\begin{equation}\label{def_sobolev_space_meausure}
(u,v)_{H^1_\nu(\Omega)}:= \int_{\Omega} g(\nabla^g u,\nabla^g v) \dvol^g+\int_{\overline{\Omega}} uv \diff \nu.
\end{equation}
Now, coherently with the notation of \Cref{sec_assumptions}, we define
\begin{equation}\label{def_H_F_Z}
H_\e= L^2(\overline{\Omega}, \beta_\e),  \quad  Z_\e=\mc{F}_\e=H^1_{\alpha_\e+\beta_\e}(\Omega),
\end{equation}
which trivially gives $Z_\e=\mathcal{F}_\e\subseteq H_\e$.
Moreover, combining the fact that $C^\infty(\overline{\Omega})\subseteq C(\overline{\Omega})$ is dense in the supremum norm and that $C(\overline{\Omega})\subseteq L^2(\overline{\Omega},\beta_\e)$ is dense in the $L^2$ norm (see e.g. \cite[Proposition 7.9]{folland}), and since $C^\infty(\overline{\Omega})\subseteq H^1_{\alpha_\e+\beta_\e}(\Omega)$ we deduce that $Z_\e\subset H_\e$ is dense; hence
\begin{equation}\label{def_Z_e_measures}
    \mc{Z}_\e=H_\e=L^2(\overline{\Omega},\beta_\e).
\end{equation}
We now define, for $u,v\in H^1_{\alpha_\e+\beta_\e}(\Omega)$, the symmetric bilinear form
\begin{equation}\label{def_Ee_measures_manifolds}
    \mathcal{E}^{(\e)}(u,v):=(u,v)_{H^1_{\alpha_\e+\beta_\e}(\Omega)}=\int_\Omega g(\nabla^gu,\nabla^gv)\dvol^g+\int_{\overline{\Omega}}uv\,\mathrm{d}(\alpha_\e+\beta_\e).
\end{equation}
We say that $\lambda\in\R$ is an \emph{eigenvalue} of \eqref{eq:eigen_measure} if there exists $\varphi\in H^1_{\alpha_\e+\beta_\e}(\Omega)\setminus\{0\}$, called \emph{eigenfunction}, such that
\begin{equation}\label{eq_eigenfunction_measures}
    \int_\Omega g(\nabla^g\varphi,\nabla^g v)\dvol^g+\int_{\overline{\Omega}}\varphi v\,\mathrm{d}(\alpha_\e+\beta_\e)=\lambda\int_{\overline{\Omega}}\varphi v\,\mathrm{d}\beta_\e\quad\textnormal{for all }v\in H^1_{\alpha_\e+\beta_\e}(\Omega).
\end{equation}
By definition, $\mathcal{E}^{(\e)}$ is a norm on $Z_\e=H^1_{\alpha_\e+\beta_\e}(\Omega)$, hence it is continuous and positive definite on $Z_\e$. Moreover, combining \Cref{prop:equiv_norms} and \Cref{lemma_compactness_equiv} we obtain that the embedding $Z_\e\hookrightarrow H_\e$ is compact. Hence, the spectrum of $(\mc{E}^{(\e)}, H^1_{\alpha_\e+\beta_\e}(\Omega))$ in  $L^2(\overline{\Omega}, \beta_\e)$ is discrete and it consists of a non-decreasing diverging sequence of non-negative eigenvalues, denoted by $\{\la_{\e,n}\}_{n \in \mb{N}\setminus\{0\}}$, where each eigenvalue is repeated according to its multiplicity. For any $\e \in [0,1]$, we consider a basis of associated eigenfunctions $\{\varphi_{\e,n,i}\colon i=1,\dots,M_{0,n}, ~n\in\N\setminus\{0\}\}$ as in \Cref{sec_assumptions}, where $M_{0,n}=\dim E(\lambda_{0,n})$ with $E(\lambda_{\e,n})$ denoting the eigenspace corresponding to $\lambda_{\e,n}$, for any $\e\in[0,1]$ and any $n\in\N\setminus\{0\}$.

\subsection{Discussion on the functional setting}\label{subsec_mesure_assum}

In the present section, we discuss the properties of the admissible families of measures (as in \Cref{def:adm}), and we prove that assuming \eqref{eq:ass_alpha_beta} implies that the assumptions in \Cref{subsec:ass_functional} hold true. 

The first step is to prove that \ref{A1} is equivalent to the compactness of the corresponding embeddings $H^1(\Omega) \hookrightarrow L^2(\overline{\Omega},\nu_\e)$.

\begin{lemma}\label{lemma_compactness_equiv}
Let $\nu$ be a positive Radon measure on $\overline{\Omega}$. The compactness of the embedding  $H^1(\Omega) \hookrightarrow L^2(\overline{\Omega},\nu)$ is equivalent to assuming that  for any $\delta>0$ there exists a constant $C(\nu,\delta)>0$ such that
\begin{align}\label{eqn-sobolev-trace-improve}
\int_{\overline \Omega}|u|^2 d \nu \leq \delta\int_{\Omega} |\nabla^g u|^2 \dvol^g + C(\nu,\delta)\int_{\Omega}|u|^2 \dvol^g
\end{align}
for any $u \in H^1(\Omega)$.
\end{lemma}
\begin{proof}
Suppose that the embedding  $H^1(\Omega) \hookrightarrow L^2(\overline{\Omega},\nu)$ is compact but \eqref{eqn-sobolev-trace-improve} does not hold true. By contradiction, there exists $\eta>0$ such that for any $n\in\N\setminus\{0\}$, there exist $u_n\in H^1(\Omega)$, with $||u_n||_{L^2(\overline{\Omega}, \nu)}=1$, such that 
\begin{align}
1 >\eta\int_{\Omega} |\nabla^g u_n|^2 \dvol^g + n \int_{\Omega}|u_n|^2 \dvol^g.
\end{align}
Hence, there exists $C>0$ such that 
\begin{equation}\label{eq:comp_equiv_1}
||\nabla^g u_n||_{L^2(\Omega)}^2\leq C \quad \text{ and }\quad ||u_n||_{L^2(\Omega)}^2\leq \frac{1}{n}\quad\textnormal{for every }n \in \mb{N}\setminus\{0\}.
\end{equation}
In particular $\{u_n\}_{n \in \mathbb{N}\setminus\{0\}}$ is bounded in $H^1(\Omega)$.
Therefore, there exist $u\in H^1(\Omega)$, such that,  up to a subsequence, $u_n\rightharpoonup u$ weakly in $H^1(\Omega)$ as $n\to \infty$. Moreover, by \eqref{eq:comp_equiv_1} and weak lower semicontinuity we have that $\norm{u}_{L^2(\Omega)}=0$ and so $u_n\rightharpoonup 0$ weakly in $H^1(\Omega)$ as $n\to\infty$ and, by compactness of $H^1(\Omega) \hookrightarrow L^2(\overline{\Omega},\nu)$, we have that $u_n\to 0$ strongly in $L^2(\overline{\Omega},\nu)$: this is a contradiction with the fact that $\norm{u_n}_{L^2(\overline\Omega, \nu)}=1$.

On the other hand, suppose that \eqref{eqn-sobolev-trace-improve} holds and that $\{u_n\}_{n \in \mathbb{N}}$ is bounded in $H^1(\Omega)$. Fix $\delta>0$. Then, by  \eqref{eqn-sobolev-trace-improve},
\begin{align}
\int_{\overline \Omega} |u_n - u_m|^2 d \nu \le \delta\int_{\Omega} |\nabla^g( u_n-u_m)|^2 \dvol^g+ C(\nu,\delta)\int_{\Omega}|u_n-u_m|^2  \dvol^g,
\end{align}
for any $n,m \in \mathbb{N}$. Being the embedding in $H^1(\Omega)\hookrightarrow L^2(\Omega)$ compact, there exists a subsequence  $\{u_{n_k}\}_{k \in \mathbb{N}}$ such that 
\begin{equation}
C(\nu,\delta)\int_{\Omega}|u_{n_k}-u_{n_h}|^2  \dvol^g \le \delta
\end{equation}
for any $k, h \ge k_0$ for some $k_0>0$. Since $\delta$ is arbitrary, we conclude that $\{u_{n_k}\}_{k \in \mathbb{N}}$ is Cauchy in $L^2(\overline \Omega, \nu)$ and so $H^1(\Omega) \hookrightarrow L^2(\overline{\Omega},\nu)$ is compact.
\end{proof}

As a consequence, we get a uniform compactness property.

\begin{corollary}\label{cor:compact_nu}
    Let $\bm{\nu}=\{\nu_\e\}_{\e\in[0,1]}$ be a family of positive Radon measures on $\overline{\Omega}$ satisfying \ref{A1}. Then the embedding $i_\e\colon H^1(\Omega)\hookrightarrow L^2(\overline{\Omega},\nu_\e)$ is compact for any $\e\in[0,1]$ and there exists a constant $c>0$ such that $\norm{u}_{L^2(\overline{\Omega},\nu_\e)}\leq c\norm{u}_{H^1(\Omega)}$ for any $u\in H^1(\Omega)$ and any $\e\in[0,1]$, that is $i_\e$ is bounded.
\end{corollary}
\begin{proof}
    This is a straightforward consequence of \Cref{lemma_compactness_equiv} and \ref{A1}.
\end{proof}

We now observe that a measure that admits an embedding-type inequality cannot be supported on sets of zero capacity.
\begin{lemma}\label{lemma:capacity}
    Let $\nu$ be a measure on $\overline{\Omega}\subseteq M$ and assume there exists $C=C(\nu)>0$ such that
    \begin{equation}\label{eq:capacity_poinc}
        \int_{\overline{\Omega}}u^2\diff\nu\leq C\left(\int_\Omega|\nabla^g u|^2\dvol^g+\int_\Omega u^2\dvol^g\right)\quad\textnormal{for all }u\in H^1(\Omega).
    \end{equation}
    For any compact $K\subseteq M$, let 
    \begin{equation}
        \mathrm{Cap}(K):=\inf\left\{\int_M (|\nabla^g u|^2+u^2)\dvol^g\colon u\in H^1(M),~u\equiv 1~\textnormal{in an open neighborhood of }K\right\}
    \end{equation}
    denote its Sobolev capacity. Then, for any compact $K\subseteq \overline{\Omega}$ such that $\mathrm{Cap}(K)=0$ we have $\nu(K)=0$.
\end{lemma}
\begin{proof}
    Let $\{u_n\}_n$ be a minimizing sequence for $\mathrm{Cap}(K)$, i.e. $u_n\equiv 1 $ in an open neighborhood of $K$ and 
    \begin{equation}
        \int_M(|\nabla^g u_n|^2+u_n^2)\dvol^g\to 0\quad\textnormal{as }n\to\infty.
    \end{equation}
    If we test \eqref{eq:capacity_poinc} with a the restriction of $u_n$ to $\overline{\Omega}$ (still denoted by $u_n$) we get
    \begin{equation}
        \nu(K)\leq \int_{\overline{\Omega}}u_n^2\diff\nu\leq C\left(\int_\Omega|\nabla^g u_n|^2\dvol^g+\int_\Omega u_n^2\dvol^g\right)\to 0,
    \end{equation}
    as $n\to\infty$, thus completing the proof.
\end{proof}

We recall from \cite[Section 3]{GKL} (cf. \cite[Proposition 3.12]{GKL}) the following.
\begin{theorem}\label{theor_prel_GKL}
Let $\nu$ be a Radon measure on  $\overline{\Omega}$ not supported on a single point and such that 
\begin{equation}\label{ineq_poinc}
\int_{\overline{\Omega}}|u-c_{u,\nu}|^2 \diff \nu \leq K_\nu \int_{\Omega} |\nabla^g u|^2 \dvol^g\quad\textnormal{for all }u\in H^1_\nu(\Omega)
\end{equation}
holds for some $K_\nu>0$, where
\begin{equation}
c_{u,\nu}:= \frac{1}{\nu(\overline{\Omega})}\int_{\overline{\Omega}} u\diff\nu.
\end{equation}
Then there holds
\begin{equation}\label{ineq_equiv_norm}
C_{1,\nu}\norm{u}_{H_\nu^1(\Omega)}\le \norm{u}_{H^1(\Omega)} \le C_{2,\nu}\norm{u}_{H_\nu^1(\Omega)},\quad\textnormal{for all }u\in C^\infty(\overline\Omega),
\end{equation}
with $C_{1,\nu}>0$ and 
\begin{equation}
C_{2,\nu}:=(1+K_\nu)\left(1+\frac{{\rm vol}^g(\overline{\Omega})^{3/2}}{\nu(\overline{\Omega})^{1/2}}\right)(1+C_{1,\nu}),
\end{equation}
In particular, $H_\nu^1(\Omega)$ and $H^1(\Omega)$ are uniformly isomorphic with equivalent norms. Furthermore, the embedding $H^1(\Omega) \hookrightarrow L^2(\Omega,\nu)$ is compact if and only if $H_\nu^1(\Omega) \hookrightarrow L^2(\Omega,\nu)$  is compact and  \eqref{ineq_poinc} holds.
\end{theorem}

We want to show that we can apply this result to a family of measures satisfying \ref{A1} and get a uniform equivalence of norms. In particular, we show that \ref{A1} guarantees that the Poincaré-type inequality \eqref{ineq_poinc} holds uniformly in $\e\in [0,1]$, as explained in the following result.

\begin{proposition}\label{prop_poinc}
If a family of measures $\bm{\nu}=\{\nu_\e\}_{\e \in [0,1]}$ satisfies \ref{A1},
then there exists a constant $K>0$, that does not depend on $\e$, such that 
\begin{equation}
\int_{\overline{\Omega}}|u-c_{u,\nu_\e}|^2 \diff\nu_\e \leq K \int_{\Omega} |\nabla^g u|^2 \dvol^g
\end{equation}
for any $u \in H^1(\Omega)$ and any $\e\in[0,1]$.
\end{proposition}
\begin{proof}
Assume by contradiction that there exists a sequence $\e_n \to 0^+$ as $n \to \infty$  and a sequence $\{u_n \}_n\subset H^1(\Omega)$ such that 
\begin{equation}
\int_{\overline{\Omega}}|u_n-c_{u_n,\nu_{\e_n}}|^2 \diff\nu_{\e_n} > n  \int_{\Omega} |\nabla^g u_n|^2 \dvol^g.
\end{equation}
Now, let
\begin{equation}
    v_n:=\frac{u_n-c_{u_n,\nu_{\e_n}}}{\norm{u_n-c_{u_n,\nu_{\e_n}}}_{L^2(\overline{\Omega},\nu_{\e_n})}},
\end{equation}
so that we have
\begin{align}
    &\int_{\overline{\Omega}}v_n^2\diff\nu_{\e_n}=1, \label{eq:poinc1}\\
    &\int_{\overline{\Omega}}v_n\diff\nu_{\e_n}=0, \label{eq:poinc2}\\
    &\int_\Omega|\nabla^g v_n|^2\dvol^g<\frac{1}{n}\label{eq:poinc3}
\end{align}
for all $n\in\N\setminus\{0\}$. We now let
\begin{equation}
    a_n:=\frac{1}{\mathrm{vol}^g(\Omega)}\int_\Omega v_n\dvol^g,
\end{equation}
and we use the Poincaré inequality to get that
\begin{equation}
    \int_\Omega|v_n-a_n|^2\dvol^g\leq C\int_\Omega|\nabla^g v_n|^2\dvol^g<\frac{C}{n}\quad\textnormal{for all }n\in\N,
\end{equation}
for some $C>0$ independent from $n$. Combining this with \ref{A1} and \eqref{eq:poinc3} we get that, for some fixed $\delta>0$,
\begin{equation}
    \int_{\overline{\Omega}}|v_n-a_n|^2\diff\nu_{\e_n}\leq \delta\int_\Omega|\nabla^g v_n|^2\dvol^g+C(\delta)\int_\Omega |v_n-a_n|^2\dvol^g<\frac{C'}{n},
\end{equation}
for some $C'>0$ independent from $n$. Now, by expanding the left-hand side, and using \eqref{eq:poinc1} and \eqref{eq:poinc2} we get that
\begin{equation}
    1\leq \int_{\overline{\Omega}}|v_n|^2\diff\nu_{\e_n}+a_n^2\nu_{\e_n}(\overline{\Omega})-2a_n\int_{\overline{\Omega}}v_n\diff\nu_{\e_n}=\int_{\overline{\Omega}}|v_n-a_n|^2\diff\nu_{\e_n}\leq \frac{C'}{n},
\end{equation}
which is a contradiction for $n$ large.
\end{proof}

As a consequence, we get the uniform equivalence of norms.
\begin{proposition}\label{prop:equiv_norms}
    Let $\bm{\nu}=\{\nu_\e\}_{\e\in[0,1]}$ satisfy \ref{A1} and suppose there exists a positive constant $c>0$ such that $\nu_\e(\overline{\Omega})\geq c$ for every $\e \in [0,1]$. Then, there exists $C_{\textup{E}}>0$ (depending on $c$) such that
    \begin{equation}
        C_{\textup{E}}\norm{u}_{H^1_{\nu_\e}(\Omega)}\leq \norm{u}_{H^1(\Omega)}\leq C_{\textup{E}}^{-1}\norm{u}_{H^1_{\nu_\e}(\Omega)}\quad\text{for all }u\in C^\infty(\overline\Omega)
    \end{equation}
    and for all $\e\in[0,1]$. In particular, the completions $H^1_{\nu_\e}(\Omega)$ and $H^1(\Omega)$ coincide as vector spaces and are uniformly isomorphic as metric spaces, for any $\e\in[0,1]$. 
 \end{proposition}
\begin{proof}
    From \Cref{lemma:capacity} we know that $\nu_\e$ is not supported on a single point (since points have zero Sobolev capacity), while from \Cref{prop_poinc} we know that \eqref{ineq_poinc} holds. Therefore, we can apply \Cref{theor_prel_GKL} and obtain the thesis.
\end{proof}

In view of the previous discussion, we have that the functional setting, for fixed $\e\in[0,1]$, satisfies the assumptions as in \Cref{subsec:ass_functional}. More precisely, we have the following.

\begin{lemma}\label{lemma:ass-measures}
    Let $\e\in[0,1]$ be fixed. Let $H_\e,\mathcal{F}_\e,Z_\e$ be as in \eqref{def_H_F_Z} and $\mathcal{E}^{(\e)}$ be as in \eqref{def_Ee_measures_manifolds}. Then, $\mathcal{F}_\e\subseteq H_\e$ and the real bilinear form $\mathcal{E}^{(\e)}\colon \mathcal{F}_\e\times\mathcal{F}_\e\to \R$ is 
    \begin{itemize}
        \item symmetric;
        \item bounded from below and, in particular, satisfies $\mathcal{E}^{(\e)}(u)\geq \norm{u}_{H_\e}^2$ for all $u\in \mathcal{F}_\e$;
        \item closed, i.e. $\mathcal{F}_\e$ is complete with respect to $\sqrt{\mathcal{E}^{(\e)}}$.
    \end{itemize}
    Moreover, $Z_\e\hookrightarrow \mathcal{Z}_\e$ compactly, where $\mathcal{Z}_\e$ is as in \eqref{def_Z_e_measures}. In particular, all the assumptions in \Cref{subsec:ass_functional} are satisfied, with $L_\e$ being the identity.
\end{lemma}
\begin{proof}
    The result is a trivial consequence of the definitions, apart from the compactness $Z_\e\hookrightarrow \mathcal{Z}_\e$ which can be easily obtained by combining \Cref{lemma_compactness_equiv}, \Cref{prop:equiv_norms} with \ref{A1} and \eqref{eq:ass_alpha_beta}.

\end{proof}

\subsection{Discussion on the stability} 

In this section, we prove that the stability assumptions in \Cref{subsec:ass_stability} are satisfied within the functional framework we fixed. Then, we apply \Cref{thm:main} to the present setting and provide the quantitative spectral stability in the case of weighted eigenvalues.

We start by giving a characterization of \Cref{def:adm}.
\begin{proposition}\label{prop_limit_alpha_e}
    We have that the following are equivalent:
    \begin{enumerate}
        \item $\bm{\nu}=\{\nu_\e\}_{\e \in [0,1]}$ is admissible, in the sense of \Cref{def:adm}; \label{it:lim_alpha_e_th1}
        \item let $\e_n\to 0^+$ as $n\to\infty$ and let $\{u_n\}_n,\{v_n\}_n\subseteq H^1(\Omega)$ and $u,v\in H^1(\Omega)$ be such that $u_n\rightharpoonup u$ and $v_n\rightharpoonup v$ weakly in $H^1(\Omega)$ as $n\to\infty$. Then
        \begin{equation}
          \int_{\overline{\Omega}} u_n v_n \diff\nu_{\e_n}\to  \int_{\overline{\Omega}} uv \diff \nu_0\quad\textnormal{as }n\to\infty.
        \end{equation}\label{it:lim_alpha_e_th2}
    \end{enumerate}
\end{proposition}
\begin{proof}
    \eqref{it:lim_alpha_e_th1} $\implies$ \eqref{it:lim_alpha_e_th2}. First of all, we observe that in \eqref{it:lim_alpha_e_th2} we can take $u_n=v_n$ and $u=v$, in view of the identity 
    \begin{equation}
        -2 \int_{\overline{\Omega}} u_n v_n \diff \nu_{\e_n}=  \int_{\overline{\Omega}} (u_n-v_n)^2 \diff \nu_{\e_n}-\int_{\overline{\Omega}} u^2_n \diff \nu_{\e_n}-\int_{\overline{\Omega}} v^2_n \diff \nu_{\e_n}.
    \end{equation}
    Furthermore, we show that if $\e_n\to 0^+$ as $n\to\infty$ and 
    \begin{equation}\label{eq:alpha_equiv1}
        \int_{\overline{\Omega}} u^2 \diff\nu_{\e_n}\to  \int_{\overline{\Omega}} u^2 \diff \nu_0\quad\textnormal{as }n\to\infty
    \end{equation}
    for any $u\in H^1(\Omega)$, then 
    \begin{equation}\label{eq:alpha-equiv4}
        \int_{\overline{\Omega}} u_n^2 \diff\nu_{\e_n}\to  \int_{\overline{\Omega}} u^2 \diff \nu_0\quad\textnormal{as }n\to\infty
    \end{equation}
    for any $\{u_n\}_n\subset H^1(\Omega)$ such that $u_n\rightharpoonup u$ weakly in $H^1(\Omega)$ as $n\to\infty$ for some $u\in H^1(\Omega)$. First of all, we estimate 
    \begin{equation}\label{eq:alpha-equiv2}
        \left|\int_{\overline{\Omega}} u^2  \diff \nu_0-\int_{\overline{\Omega}} u_n^2  \diff \nu_{\e_n}\right|
        \le\left|\int_{\overline{\Omega}} u^2 \diff \nu_0-\int_{\overline{\Omega}} u^2 \diff \nu_{\e_n}\right|+\left|\int_{\overline \Omega}( u^2-u_n^2)  \diff \nu_{\e_n}\right|.
    \end{equation}
    The first term vanishes in view of \eqref{eq:alpha_equiv1}, while for the second, thanks to \ref{A1}, we have that for any $\delta>0$ there holds
    \begin{multline}
        \left|\int_{\overline \Omega} (u-u_n)(u+u_n)  \diff \nu_{\e_n}\right| \le \left(\int_{\overline \Omega} (u+u_n)^2  \diff \nu_{\e_n}\right)^{\frac{1}{2}}\left(\int_{\overline \Omega}  (u-u_n)^2  \diff \nu_{\e_n}\right)^{\frac{1}{2}}\\
        \le \sqrt{\max\{C(1),1\}} \norm{u+u_n}_{H^1(\Omega)} \left[\delta \int_{\Omega}  |\nabla^g (u-u_n)|^2  \dvol^g + C(\delta)\int_{\Omega}  (u-u_n)^2 \dvol^g \right]^\frac{1}{2}.
    \end{multline}
    Hence, since the embedding of $H^1(\Omega)$ into $L^2(\Omega)$ is compact, 
    we can pass to the limit as $n\to\infty$ thus obtaining, since  $\{u_n\}_n$ is bounded in $H^1(\Omega)$,
    \begin{equation}\label{eq:alpha-equiv3}
        \limsup_{n\to\infty}\left|\int_{\overline \Omega} (u-u_n)(u+u_n) \diff \nu_{\e_n}\right|
        \le C  \delta^{\frac{1}{2}}
    \end{equation}
    for some constant $C>0$ that does not depend on $\delta$. By combining \eqref{eq:alpha-equiv2} with \eqref{eq:alpha_equiv1} and \eqref{eq:alpha-equiv3}, we conclude that \eqref{eq:alpha-equiv4} holds since $\delta>0$ is arbitrary.
    Hence, we are only left to prove \eqref{eq:alpha_equiv1}.
    Let $\{u_k\}_k \subset C^\infty(\overline{\Omega})$ such that $u_k \to u$ strongly in $H^1(\Omega)$ as $k \to \infty$. In view of Cauchy-Schwarz inequality and \ref{A1} we have
    \begin{multline}
        \left|\int_{\overline{\Omega}} u^2 \diff \nu_{\e_n}- \int_{\overline{\Omega}} u^2 \diff \nu_0 \right| 
        \le \left |\int_{\overline{\Omega}} (u^2-u_k^2) \diff \nu_{\e_n}\right|+\left|\int_{\overline{\Omega}} (u^2-u_k^2) \diff \nu_{0}\right| \\
        +\left|\int_{\overline{\Omega}} u_k^2 \diff \nu_{\e_n}-\int_{\overline{\Omega}} u_k^2 \diff \nu_{0}\right|\le C \norm{u-u_k}_{H^1(\Omega)}\norm{u+u_k}_{H^1(\Omega)}
        +\left|\int_{\overline{\Omega}} u_k^2 \diff \nu_{\e_n}-\int_{\overline{\Omega}} u_k^2 \diff \nu_{0}\right|,
    \end{multline}
    for some constant $C>0$ that does not depend on $n$ and $k$. Hence, passing to the superior limit as $n\to\infty$, using \ref{A2} and then to the limit as $k\to\infty$, we conclude the proof of \eqref{eq:alpha_equiv1}, which also completes the proof of \eqref{it:lim_alpha_e_th2}.

    \noindent \eqref{it:lim_alpha_e_th2} $\implies$ \eqref{it:lim_alpha_e_th1}. We immediately observe that \eqref{it:lim_alpha_e_th2} implies \ref{A2}; hence, let us prove \ref{A1}.  Let us assume by contradiction that there exists $\delta>0$ such that for any $n\in\N$ there exists $u_n\in H^1(\Omega)$ and $\{\e_n\}_n$ such that $\e_n\to 0^+$ as $n\to\infty$ and
    \begin{equation}\label{eq:alpha-equiv-7}
        \int_{\overline \Omega}|u_{n}|^2 \diff \nu_{\e_n} > \delta\int_{\Omega} |\nabla^g u_{n}|^2 \dvol^g + n\int_{\Omega}|u_{n}|^2 \dvol^g\quad\textnormal{for any }n\in\N\setminus\{0\}.
    \end{equation}
    If $\nu_{\e_n}(\overline{\Omega})=0$ for some $n\in\N$, then this trivially contradicts \eqref{eq:alpha-equiv-7}. Hence, let us assume that $\nu_{\e_n}(\overline{\Omega})>0$ for any $n\in\N$.
    If we define
    \begin{equation}
        v_n:=\frac{u_n}{\norm{u_n}_{L^2(\overline{\Omega},\nu_{\e_n})}},
    \end{equation}
    this translates into
    \begin{align}
        &\norm{v_n}_{L^2(\overline{\Omega},\nu_{\e_n})}=1 \label{eq:alpha-equiv-5}\\
        &\delta\int_{\Omega} |\nabla^g v_{n}|^2 \dvol^g + n\int_{\Omega}|v_{n}|^2 \dvol^g<1\label{eq:alpha-equiv-6}
    \end{align}
    for any $n\in\N$. In particular, $\{v_n\}_n$ is bounded in $H^1(\Omega)$, hence there exists $v\in H^1(\Omega)$ such that $v_n\rightharpoonup v$ weakly in $H^1(\Omega)$ as $n\to\infty$. By compact embedding of $H^1(\Omega)$ into $L^2(\Omega)$, \eqref{eq:alpha-equiv-6} and Fatou lemma we deduce that $v=0$, while from \eqref{eq:alpha-equiv-5} and \eqref{it:lim_alpha_e_th2} we get that $\norm{v}_{L^2(\overline{\Omega},\nu_0)}=1$, thus reaching a contradiction. The proof is thereby complete.
\end{proof}

We are now in a position to prove the validity of \eqref{hp_stability}.
\begin{proposition}\label{prop_stab_eigen_measures}
Let $\bm{\alpha}=\{\alpha_\e\}_{\e\in[0,1]}$ and $\bm{\beta}=\{\beta_\e\}_{\e\in[0,1]}$ be a family of admissible positive Radon measures, in the sense of \Cref{def:adm}, and let $\{\lambda_{\e,n}\}_n$ be the eigenvalues of \eqref{eq:eigen_measure}. For any $n \in \mb{N}$
\begin{equation}\label{eq_stab_eigen_measures}
\lim_{\e \to 0^+} \la_{\e,n} =\la_{0,n}.
\end{equation}
\end{proposition}

\begin{proof}
Let $i_\e\colon H^1(\Omega)\to L^2(\overline{\Omega},\beta_\e)$ be the compact, uniformly bounded embedding (see \Cref{cor:compact_nu}) and let $j_\e\colon H^1(\Omega) \to H^1_{\alpha_\e+\beta_\e}(\Omega)$ be the identity. Let  $T_\e$ be the self-adjoint operator associated to the bilinear form $\mc{E}^{(\e)}$ (defined in \eqref{def_Ee_measures_manifolds}), for any $\e \in [0,1]$, and let $R_\e\colon L^2(\overline{\Omega},\beta_\e)\to H^1_{\alpha_\e+\beta_\e}(\Omega)$ be its resolvent, which is known to be well-defined, linear, self-adjoint and continuous (see the discussion in \Cref{subsec:ass_functional}). Hence, the operator
\begin{equation}
Q_\e: H^1(\Omega)\overset{i_\e}{\longrightarrow} L^2(\overline{\Omega},\beta_\e)\overset{R_\e}{\longrightarrow} H^1_{\alpha_\e+\beta_\e}(\Omega)\overset{j_\e^{-1}}{\longrightarrow}H^1(\Omega)
\end{equation}
is compact and its spectrum coincide with $\{\lambda_{\e,n}^{-1}\}_n$.  Moreover, we have that
\begin{equation}\label{eq:resolvent}
    \mathcal{E}^{(\e)}(Q_\e v,u)=\int_{\overline{\Omega}}vu\diff\beta_\e\quad\textnormal{for all }u,v\in H^1(\Omega)\textnormal{ and for all }\e\in[0,1].
\end{equation}
Furthermore, the operators $Q_\e$ are equibounded. More precisely, by \Cref{prop:equiv_norms}, Cauchy-Schwarz inequality and \ref{A1}
\begin{multline}\label{ineq_norm_Re_measures}
\norm{Q_\e v}_{H^1(\Omega)}^2\le C \mc{E}^{(\e)}(Q_\e v,Q_\e v) = C\int_{\overline{\Omega}} Q_\e vv \diff \beta_\e \\
\le C \norm{Q_\e v}_{L^2(\overline{\Omega},\beta_\e)}\norm{v}_{L^2(\overline{\Omega},\beta_\e)}\leq C \norm{Q_\e v}_{H^1(\Omega)}\norm{v}_{H^1(\Omega)},
\end{multline} 
for any $v\in H^1(\Omega)$, where the constant $C>0$ does not depend on $v$ and $\e$ and changes from time to time.
Hence, for any $v \in H^1(\Omega)$ the family $\{Q_\e v\}_{\e\in[0,1]}$ is bounded in $H^1(\Omega)$ and so there exists a subsequence $\e_n\to 0^+$ such that $Q_{\e_n}v \rightharpoonup w$ weakly in $H^1(\Omega)$ as $n\to \infty$,  for some $w \in H^1(\Omega)$.
Furthermore, in view of \Cref{prop_limit_alpha_e} we have that for any $u \in  H^1(\Omega)$ there holds
\begin{equation}
   \lim_{n \to \infty}\int_{\overline{\Omega}} v u \diff \alpha_{\e_n}= \int_{\overline{\Omega}}v u \diff \alpha_0\quad\textnormal{and}\quad \lim_{n \to \infty}\int_{\overline{\Omega}} v u \diff \beta_{\e_n}=\int_{\overline{\Omega}}v u \diff \beta_0
\end{equation}
while, by the weak convergence, we have
\begin{equation}
    \lim_{n\to\infty}\int_\Omega g(\nabla^g Q_{\e_n}v,\nabla^g u)\dvol^g=\int_\Omega g(\nabla^g w,\nabla^g u)\dvol^g.
\end{equation}
As a consequence of these two facts and \eqref{eq:resolvent}, we have that
\begin{equation}
\int_{\overline{\Omega}}v u \diff \beta_0=\lim_{n \to \infty}\int_{\overline{\Omega}} v u \diff \beta_{\e_n}= \lim_{n \to \infty}\mc{E}^{(\e_n)}( Q_{\e_n}v, u)= \mc{E}^{(0)}(w, u)
\end{equation}
and so  $w=Q_0v$ and, by the Urysohn subsequence principle,  $Q_\e v \rightharpoonup Q_0 v$ weakly in $H^1(\Omega)$ as $\e \to 0^+$. 
We are going to prove that  
\begin{equation}\label{eq_norm_Re}
\lim_{\e \to 0^+}\norm{Q_\e-Q_0}_{\mc{L}(H^1(\Omega), H^1(\Omega))}=0.
\end{equation}
Then,  \eqref{eq_stab_eigen_measures} will follow from \cite[Chapter XI 9.5, Lemma 5, Page~1091]{DSJ_book}.

Since $Q_\e$ is compact  for any $\e \in [0,1]$ there exists $u_\e \in H^1(\Omega)$ with $\norm{u_\e}_{H^1(\Omega)}=1$ such that 
\begin{equation}
\norm{Q_\e-Q_0}_{\mc{L}(H^1(\Omega), H^1(\Omega))}= \norm{Q_\e u_\e-Q_0 u_\e}_{H^1(\Omega)}.
\end{equation}
Furthermore, there exists $u_0 \in H^1(\Omega)$ and a sequence $\e_n\to 0^+$ such that $u_{\e_n} \rightharpoonup u_0$  weakly in $H^1(\Omega)$ as $n \to \infty$.
Since $Q_0$ is compact, $Q_0 u_{\e_n} \to Q_0 u_0$ strongly in  $H^1(\Omega)$ as $n \to \infty$.

By \eqref{ineq_norm_Re_measures}, up to passing to a further subsequence, there exists $w \in H^1(\Omega)$ such that  $Q_{\e_n} u_{\e_n}  \rightharpoonup w$  weakly in $H^1(\Omega)$ as $n \to \infty$. For any $v \in H^1(\Omega)$, by \Cref{prop_limit_alpha_e} and \eqref{eq:resolvent}
\begin{equation}
\int_{\overline{\Omega}} wv \diff \beta_0=\lim_{n \to \infty}\int_{\overline{\Omega}}Q_{\e_n} u_{\e_n} v \diff  \beta_{\e_n} =\lim_{n \to \infty}\int_{\overline{\Omega}} u_{\e_n} Q_{\e_n}v \diff  \beta_{\e_n}=\int_{\overline{\Omega}} u_0 Q_0 v \diff \beta_0=\int_{\overline{\Omega}}Q_0 u_0 \, v \diff \beta_0.
\end{equation}
It follows that $w=Q_0 u_0$ in $L^2(\overline{\Omega},\beta_0)$. Finally, again in view of  \eqref{eq:resolvent} and \Cref{prop_limit_alpha_e} we have
\begin{align}
\lim_{n\to\infty}\mc{E}^{(\e_n)}(Q_{\e_n} u_{\e_n},Q_{\e_n} u_{\e_n})&=\lim_{n\to\infty}\int_{\overline{\Omega}} u_{\e_n} Q_{\e_n}u_{\e_n} \diff  \beta_{\e_n}\\ &= \int_{\overline{\Omega}}  u_0 w \diff  \beta_0\\ &=\int_{\overline{\Omega}}  u_0  Q_0u_0 \diff  \beta_0\\
&=\mc{E}^{(0)}(Q_0 u_0,Q_0 u_0).
\end{align}
We now prove that
\begin{equation}\label{eq:stab_1}
    \lim_{n\to\infty}\mathcal{E}^{(0)}(Q_{\e_n}u_{\e_n},Q_{\e_n}u_{\e_n})=\mc{E}^{(0)}(Q_0 u_0,Q_0 u_0).
\end{equation}
Indeed, from \eqref{eq:resolvent} we have that
\begin{equation}
    \int_\Omega |\nabla^g Q_{\e_n}u_{\e_n}|^2\dvol^g=\int_{\overline{\Omega}}u_{\e_n}Q_{\e_n}u_{\e_n}\diff\beta_{\e_n}-\int_{\overline{\Omega}}|Q_{\e_n}u_{\e_n}|^2\diff(\alpha_{\e_n}+\beta_{\e_n}),
\end{equation}
from which we deduce that
\begin{equation}
    \mathcal{E}^{(0)}(Q_{\e_n}u_{\e_n},Q_{\e_n}u_{\e_n})=\int_{\overline{\Omega}}u_{\e_n}Q_{\e_n}u_{\e_n}\diff\beta_{\e_n}-\int_{\overline{\Omega}}|Q_{\e_n}u_{\e_n}|^2\diff(\alpha_{\e_n}+\beta_{\e_n})+\int_{\overline{\Omega}}|Q_{\e_n}u_{\e_n}|^2\diff(\alpha_0+\beta_0).
\end{equation}
Then, by \eqref{eq:ass_alpha_beta}, \ref{A1}, \Cref{lemma_compactness_equiv} and \Cref{prop_limit_alpha_e} we obtain that
\begin{equation}
    \lim_{n\to\infty}\mathcal{E}^{(0)}(Q_{\e_n}u_{\e_n},Q_{\e_n}u_{\e_n})=\int_{\overline{\Omega}}  u_0 w \diff  \beta_0=\int_{\overline{\Omega}}u_0Q_0u_0\diff\beta_0
\end{equation}
which, in view of \eqref{eq:resolvent}, implies \eqref{eq:stab_1}. 
Hence, since weak  convergence in $H^1(\Omega)$ implies weak convergence in $H_{\beta_0}^1(\Omega)$ (see \Cref{prop:equiv_norms}), taking into account the definition of $\mathcal{E}^{(0)}$ and \eqref{eq:stab_1} we have that
\begin{align}
\lim_{n\to \infty} \norm{Q_{\e_n} u_{\e_n}-Q_0 u_0}_{H_{\beta_0}^1(\Omega)}^2 =&\lim_{n\to \infty} \mc{E}^{(0)}(Q_{\e_n} u_{\e_n}-Q_0 u_0,Q_{\e_n} u_{\e_n}-Q_0 u_0)\\
=&\lim_{n\to \infty} \left[\mc{E}^{(0)}(Q_{\e_n} u_{\e_n},Q_{\e_n} u_{\e_n})-2 \mc{E}^{(0)}(Q_{\e_n} u_{\e_n},Q_0 u_0)\right]\\
\qquad &+\mc{E}^{(0)}(Q_0 u_0,Q_0 u_0)=0
\end{align}
and so $Q_{\e_n} u_{\e_n}  \to Q_0 u_0$  strongly in $H_{\beta_0}^1(\Omega)$ and so, in view of \Cref{prop:equiv_norms}, in $H^1(\Omega)$ as $n \to \infty$.
We conclude that 
\begin{equation}
\lim_{n  \to \infty}\norm{Q_{\e_n}-Q_0}_{\mc{L}(H^1(\Omega), H^1(\Omega))}=\lim_{n  \to \infty}\norm{Q_{\e_n} u_{\e_n}-Q_0 u_{\e_n}}_{H^1(\Omega)}=0.
\end{equation}
By the Urysohn subsequence principle, we have proved \eqref{eq_norm_Re}.
\end{proof}

With the same notations of \Cref{sec_assumptions}, we let $L_\e$ be the identity operator in $H^1(\Omega)$. Then 
\begin{equation}
\lim_{\e \to 0^+}\int_{\overline{\Omega}}\varphi_{0,n,i} \varphi_{0,n,j} \diff \beta_\e 
=\int_{\overline{\Omega}}\varphi_{0,n,i} \varphi_{0,n,j} \diff \beta_0,
\end{equation}
by \Cref{prop_limit_alpha_e}. Hence, \eqref{hp_Le} trivially holds. In order to apply our main result \Cref{thm:main}, we are only left to prove that \eqref{hp_delta_e} holds.

Let $N \in \mathbb{N} \setminus\{0\}$ be fixed and let $\la_{0,N}$ the corresponding eigenvalue of $\mc{E}^{(0)}$ on $L^2(\overline\Omega, \beta_0)$.
If $E(\la_{0,N})$ is the associated eigenspace, let $V_{\e,i}:=V_{\e,\varphi_{0,N,i}}$ be as in \Cref{prop_J_min} for any $i=1, \dots M_{0,N}$, where $M_{0,N}$ is the dimension of $E(\la_{0,N})$.
The bilinear form $q_\e$ defined in \eqref{def_qe} is given by
\begin{equation}
q_\e(u,v)=\int_{\Omega}g(\nabla^g u, \nabla^g v) \dvol^g + \int_{\overline{\Omega}} uv \diff (\alpha_\e+\beta_\e)-\la_{0,N}\int_{\overline{\Omega}} uv \diff \beta_\e,
\end{equation}
for any $u,v \in H^1(\Omega)$. Then,  by \eqref{eq_eigenfunction_measures}, \eqref{eq_Ve} becomes
\begin{multline}\label{eq_Ve_meausures}
\int_{\Omega} g( \nabla^g V_{\e,i}, \nabla^g v) \dvol^g+  \int_{\overline{\Omega}} V_{\e,i} v \diff  (\alpha_\e+\beta_\e)
 \\
 = \int_{\overline{\Omega}}   \varphi_{0,N,i} v \diff  (\alpha_\e-\alpha_0+\beta_\e-\beta_0)-\la_{0,N}  \int_{\overline{\Omega}}   \varphi_{0,N,i} v \diff  (\beta_\e-\beta_0) ,
\end{multline}
for any $v \in H^1(\Omega)$. We now prove the following, which implies \eqref{hp_delta_e}.

\begin{lemma}
    We have that
    \begin{equation}
        \delta_\e\leq \left(\sum_{i=1}^{M_{0,N}}\norm{V_{\e,i}}_{L^2(\overline{\Omega},\beta_\e)}^2\right)^{\frac{1}{2}}\to 0\quad\textnormal{as }\e\to 0^+.
    \end{equation}
\end{lemma}
\begin{proof}
    In view on the assumptions on $\alpha_\e$ and $\beta_\e$ (see \eqref{eq:ass_alpha_beta} and \Cref{def:adm}) and thanks to \Cref{prop_limit_alpha_e}, we can pass to the limit as $\e\to 0^+$ in \eqref{eq_Ve_meausures} and obtain that
    \begin{equation}
        \int_{\Omega} g( \nabla^g V_{\e,i}, \nabla^g v) \dvol^g+  \int_{\overline{\Omega}} V_{\e,i} v \diff (\alpha_\e+\beta_\e)\to 0\quad\textnormal{for all }v\in H^1(\Omega).
    \end{equation}
Hence, we have proved that $V_{\e,i} \rightharpoonup 0$ weakly in $H^1(\Omega)$ as $\e \to0^+$. Then, thanks to \Cref{prop_limit_alpha_e}, we can test with $v=V_{\e,i}$ and conclude the proof.
\end{proof}

We now recall that the symmetric bilinear form $h_\e$, defined in \eqref{def_h_e}, has the following expression
\begin{equation}\label{eq:h_measures}
    h_\e(\varphi,\psi)=\lambda_{0,N}\int_{\overline{\Omega}}\varphi V_{\e,\psi}\,d\beta_\e,\qquad \varphi,\psi\in E(\lambda_{0,N})
\end{equation}
and eigenvalues $\{\gamma_{\e,i}\}_{i=1,\dots,M_{0,N}}$. We finally recall the definition of the following quantity (see \eqref{def_tau_e})
\begin{equation}
    \tau_\e^2=\max_{i=1,\dots,M_{0,N}}|\gamma_{\e,i}|.
\end{equation}

In conclusion, we have checked all the assumptions of \Cref{thm:main} and thus by \Cref{thm:main} we have proved the following result.
\begin{theorem}\label{theor_eigen_var_measures}
Let $\bm{\alpha}=\{\alpha_\e\}_{\e\in[0,1]}$ and $\bm{\beta}=\{\beta_\e\}_{\e\in[0,1]}$ be a family of admissible positive Radon measures, in the sense of \Cref{def:adm}, and let $\{\lambda_{\e,n}\}_n$ be the eigenvalues of \eqref{eq:eigen_measure}. If $\lambda_{0,N}$ is of multiplicity $M_{0,N}$ and
\begin{equation}
    \lambda_{0,N-1}<\lambda_{0,N}=\cdots=\lambda_{0,N+M_{0,N}-1}<\lambda_{0,N+M_{0,N}},
\end{equation}
then there holds
\begin{equation}
\lambda_{\e,N+i-1}-\lambda_{0,N}=\gamma_{\e,i}+ O(\delta_\e^2)+o(\tau_\e^2) \quad \text{ as } \e \to 0^+
\end{equation} 
for any $i=1, \dots, M_{0,N}$, where $\{\gamma_{\e,i}\}_{i=1,\dots,M_{0,N}}$ are the eigenvalues of \eqref{eq:h_measures}.
\end{theorem}

\section{Measures concentrating at the boundary}\label{sec_conc_boundary}
In this section, we apply the results obtained in \Cref{sec_manifolds_measures} to a couple of concrete examples, in particular we consider families of measures concentrating to the boundary of an open set in the Euclidean space.

\subsection{Approximating Steklov eigenvalues}\label{subsec_steklov}
We let $M=\R^d$ and $\Omega\subseteq\R^d$ be an open, bounded, connected set with $C^{1,1}$ boundary. Moreover, we consider the family of measures
\begin{equation}\label{eq:def_steklov_beta}
    \alpha_\e=0\quad\textnormal{and}\quad\beta_\varepsilon=p_\varepsilon \mathcal{L}^d,
\end{equation}
where
\begin{equation}\label{eq_p_e_Omega_e}
    p_\varepsilon:=\frac{1}{\e}\chi_{\Omega\setminus\Omega_\e},\qquad \Omega_\e:=\{x\in\Omega\colon \mathrm{dist}\,(x,\partial\Omega)>\e\}.
\end{equation}
Finally, we let
\begin{equation}
    \alpha_0=0\quad\textnormal{and}\quad\beta_0=\mathcal{H}^{d-1}\text{\huge$\llcorner$}\,\partial\Omega.
\end{equation}
Within this framework, we have that the limit problem is given by the Steklov eigenvalue problem
\begin{equation}\label{eq:steklov}
    \begin{cases}
        -\Delta\varphi =0, &\text{in }\Omega, \\
        \partial_{\nnu} \varphi +\varphi=\lambda \varphi, &\text{on }\partial\Omega,
    \end{cases}
\end{equation}
whose weak formulation is $\varphi\in H^1(\Omega)$ and
\begin{equation}\label{eq_ste_eigen}
    \int_\Omega\nabla\varphi\cdot\nabla v\dx+\int_{\partial\Omega}\varphi v\ds=\lambda\int_{\partial\Omega}\varphi v\ds\quad\textnormal{for all }v\in H^1(\Omega).
\end{equation}
On the other hand, the approximating problems are weighted, Neumann-type, eigenvalue problems
\begin{equation}\label{eq_neu_to_stek}
    \begin{cases}
        -\Delta \varphi + p_\e \varphi =\lambda p_\e \varphi , &\text{in }\Omega, \\
        \partial_{\nnu} \varphi =0, &\text{on }\partial\Omega,
    \end{cases}
\end{equation}
with the weight $p_\e$ that is concentrating on the boundary $\partial\Omega$, whose formulation is $\varphi\in H^1(\Omega)$ and
\begin{equation}
   \int_\Omega\nabla\varphi\cdot\nabla v\dx+\frac{1}{\e} \int_{\Omega\setminus\Omega_\e}\varphi v\dx=\frac{\lambda}{\e}\int_{\Omega\setminus\Omega_\e}\varphi v\dx\quad\textnormal{for all }v\in H^1(\Omega).
\end{equation} 
We remark that this kind of approximation has already been considered in the literature, see e.g. \cite{Arrieta,Provenzano1,Provenzano2}. From classical regularity theory it is known that if $\varphi$ satisfies \eqref{eq:steklov}, then $\varphi\in W^{2,p}(\Omega)$ for any $p>1$ and
\begin{equation}\label{eq:estimate_steklov}
    \norm{\varphi}_{W^{2,p}(\Omega)}\leq C(d,\Omega,p)\norm{\varphi}_{H^1(\Omega)},
\end{equation}
This can be seen e.g. by bootstrapping \cite[Theorem 2.4.2.6]{grisvard} and using \cite[Theorem 2.3.3.3]{grisvard}. In particular, this implies that $\varphi \in C^{1,\alpha}(\overline\Omega)$  for any $\alpha \in (0,1)$.

We point out that 
\begin{equation}\label{eq:kappa_shift}
    \lambda_{\e,n}=\Lambda_{\e,n}+1\quad\textnormal{for any }n\in\N~\textnormal{and any }\e\in[0,1],
\end{equation}
where $\Lambda_{0,n}$ is an eigenvalue of the ``classical'' Steklov problem
\begin{equation}
    \begin{bvp}
        -\Delta\varphi&=0, &&\textnormal{in }\Omega, \\
        \partial_{\nnu}\varphi&=\Lambda\varphi,&&\textnormal{on }\partial\Omega
    \end{bvp}
\end{equation}
and, for $\e>0$, $\Lambda_{\e,n}$ is an eigenvalue of
\begin{equation}
    \begin{bvp}
        -\Delta \varphi&=\Lambda p_\e\varphi, &&\textnormal{in }\Omega, \\
        \partial_{\nnu}\varphi &=0, &&\textnormal{on }\partial\Omega.
    \end{bvp}
\end{equation}
In particular, we have that $\lambda_{\e,n}-\lambda_{0,n}=\Lambda_{\e,n}-\Lambda_{0,n}$, for any $n\in\N$.  

Now, consistently with the previous sections, we consider a Steklov (limit) eigenvalue $\lambda_{0,N}$ of problem \eqref{eq:steklov} with multiplicity $M_{0,N}\geq 1$ such that
\begin{equation}
        \lambda_{0,N-1}<\lambda_{0,N}=\cdots=\lambda_{0,N+M_{0,N}-1}<\lambda_{0,N+M_{0,N}},
\end{equation}
and a family of $M_{0,N}$ (approximating) eigenvalues $\lambda_{\e,N+i-1}$, with $i=1,\dots,M_{0,N}$, which are converging to $\lambda_{0,N}$. Since we want to apply \Cref{theor_eigen_var_measures} in this setting, we need to verify all the assumptions. 
We start by adapting to our framework a preliminary basic result on trace inequalities. More precisely, in \cite[Theorem 18.1, (iii)]{leoni} a weighted trace inequality is proved and we push it forward to let it hold uniformly on hypersurfaces parallel (and sufficiently close) to $\partial\Omega$.

\begin{lemma}[Uniform trace inequality]\label{lemma:trace_leoni}
    There exists $C,\rho_0>0$ depending on $d$ and $\Omega$ such that
    \begin{equation}\label{ineq_traces_Ste}
        \int_{\partial\Omega_\rho}v^2\ds\leq C\int_{\Omega_\rho} (\tau|\nabla v|^2+\tau^{-1}v^2)\dx
    \end{equation}
    for all $v\in H^1(\Omega_\rho)$, all $\tau\in(0,1]$ and all $\rho\in(0,\rho_0)$, where $\Omega_\rho$ is as in \eqref{eq_p_e_Omega_e}.
\end{lemma}
\begin{proof}
    In the whole proof, we denote by $C>0$ a positive constant which depends only on $d$ and $\Omega$ and which may change from time to time. 
    By density, we can assume $v\in C^\infty(\overline{\Omega}_\rho)$. 
    Since $\Omega$ is of class $C^{1,1}$, then the distance function
    \begin{equation}
        \delta(x):=\mathrm{dist}\,(x,\partial\Omega),\quad x\in\overline{\Omega},
    \end{equation}
    is of class $C^{1,1}$ in $\overline{\Omega}\setminus\Omega_{\rho_1}$, for some $\rho_1>0$ depending only on $d$ and $\Omega$. 
    Let $\eta\in C^\infty_c([0,\rho_1/2))$ be a standard cut-off function such that $\eta \equiv1$ in $[0, \rho_1/4]$ and $\eta$ is decreasing in $(\rho_1/4,\rho_1/2)$. Moreover, we choose $\eta$ such that
    \begin{equation}
        |\eta'|\leq \frac{1}{\rho_1}<\frac{1}{\rho}.
    \end{equation}
    Then, for any $x\in\overline{\Omega}$ and any $\rho\in(0,\rho_1)$, we define
    \begin{equation}
        T_\rho(x):=x+\rho\eta(\delta(x))\nabla \delta(x).
    \end{equation}
    One can easily see that $T_\rho$ is a bi-Lipschitz homeomorphism from $\overline{\Omega}$ to $\overline{\Omega}_\rho$ and that
    \begin{equation}\label{eq:trace_1}
    \norm{D T_\rho-{\mathrm{Id}}_d}_{L^\infty(\Omega)}\leq \frac{1}{2},
    \end{equation}
    for $\rho \in (0,\rho_2)$, for some  $\rho_2 \in (0,\rho_1/2)$ small enough (depending on $d$ and $\Omega$). In particular, $v\circ T_\rho\in C^{0,1}(\overline{\Omega})$. At this point we use the change of variable $x\mapsto T_\rho(x)$,  \cite[Corollary 8.11]{maggi} (used in $\Omega$ for $v\circ T_\rho$) and \eqref{eq:trace_1} to obtain
    \begin{align}
        \int_{\partial\Omega_\rho}v^2\ds &\le C\int_{\partial\Omega}v^2(T_\rho(x))\ds(x) \\
        &\leq C\int_\Omega (\tau |\nabla (v\circ T_\rho)(x)|^2+\tau^{-1}v^2(T_\rho(x)))\dx \\
        &=C\int_\Omega (\tau |\nabla v(T_\rho(x))DT_\rho(x)|^2+\tau^{-1}v^2(T_\rho(x)))\dx \\
        &\leq C \int_\Omega \left(\frac{9}{4}\tau |\nabla v(T_\rho(x)|^2+\tau^{-1}v^2(T_\rho(x))\right)\dx\leq C\int_{\Omega_\rho}\left(\tau |\nabla v|^2+\tau^{-1}v^2\right)\dx,
    \end{align}
    which yields \eqref{ineq_traces_Ste}.
\end{proof}

We are now able to obtain a uniform Poincaré inequality, whose proof is based on  \Cref{lemma:trace_leoni} and the coarea formula, which implies \ref{A1}.
\begin{lemma}[Uniform Poincaré]\label{lemma:unif_poinc}
There exists $C,\varepsilon_0>0$, depending on $d$ and $\Omega$, such that
\begin{equation}\label{ineq_poinc_steklov}
\int_{\Omega}v^2\,d\beta_\e\leq C\int_\Omega (\tau|\nabla v|^2+\tau^{-1}v^2)\dx
\end{equation}
for all $v\in H^1(\Omega)$ and all $\tau\in(0,1]$ and all $\varepsilon\in(0,\varepsilon_0)$.
\end{lemma}
\begin{proof}
By the coarea formula and \Cref{lemma:trace_leoni}
\begin{equation}
\int_{\Omega}v^2\,d\beta_\e=\frac{1}{\e}\int_{\Omega\setminus \Omega_\e}v^2\dx= \frac{1}{\e}\int_{0}^\e \int_{\partial \Omega_\rho}v^ 2\ds d \rho
\le  C\int_{\Omega} (\tau|\nabla v|^2+\tau^{-1}v^2)\dx.
\end{equation}
Hence, \eqref{ineq_poinc_steklov} is proved. 
\end{proof}

As a consequence, we obtain that the families $\bm{\alpha}=\{\alpha_\e\}_{\e\in[0,1]}$ and $\bm{\beta}=\{\beta_\e\}_{\e\in[0,1]}$ are admissible in the sense of \Cref{def:adm}.
\begin{lemma}[Admissibility]\label{lemma_add_ste}
    The families $\bm{\alpha}=\{\alpha_\e\}_{\e\in[0,1]}$ and $\bm{\beta}=\{\beta_\e\}_{\e\in[0,1]}$ defined in \eqref{eq:def_steklov_beta} are admissible, i.e. they satisfy \ref{A1} and \ref{A2}. 
\end{lemma}
\begin{proof}
    We immediately observe that since $\alpha_\e=0$, it is enough to prove the result for $\beta_\e$. First, we observe that \ref{A1} follows from \Cref{lemma:unif_poinc}. Furthermore, by performing the change of variable $y=T_\rho(x)$ as in the proof of \Cref{lemma:trace_leoni}, it is easy to see that 
\begin{equation}
\int_{\Omega}v \,d\beta_\e=\frac{1}{\e}\int_0^\e \int_{\partial\Omega_\rho}v  \ds\,d\rho\to \int_{\partial\Omega}v \ds \quad\text{as }\e\to 0^+,
\end{equation}
for any $v \in C^\infty(\overline{\Omega})$. This yields \ref{A2} and concludes the proof.
\end{proof}

Hence, the assumptions of \Cref{theor_eigen_var_measures} are satisfied. Now, we move to investigate the asymptotic behavior (as $\e\to 0^+$) of the family $\{\gamma_{\e,i}\}_{i=1,\dots,M_{0,N}}$ as in \Cref{theor_eigen_var_measures}, which in turn yields the asymptotic expansion of the eigenvalue variation. To begin with, let us recall some facts. For any $\varphi\in E(\lambda_{0,N})$, we let $V_{\e,\varphi}\in H^1(\Omega)$ be the unique solution of 
\begin{equation}\label{eq_Ve_ste}
\int_\Omega\nabla V_{\e,\varphi}\cdot\nabla u\dx+\frac{1}{\e}\int_{\Omega\setminus\Omega_\e}V_{\e,\varphi} u\dx=(\lambda_{0,N}-1)\left(\int_{\partial\Omega}\varphi u\ds-\frac{1}{\e}\int_{\Omega\setminus\Omega_\e}\varphi u\dx\right),
\end{equation}
for any $u\in H^1(\Omega)$, whose corresponding strong formulation is the following
\begin{equation}\label{prob_Ve_steklov}
\begin{bvp}
-\Delta V_{\e,\varphi}&=0, &&\text{in }\Omega_\e, \\
 -\Delta V_{\e,\varphi} +\frac{1}{\e}V_{\e,\varphi}&=-\frac{\lambda_{0,N}-1}{\e}\varphi, &&\text{in }\Omega\setminus\Omega_\e, \\
\partial_{\nnu} V_{\e,\varphi}&=(\lambda_{0,N}-1)\varphi, &&\text{on }\partial\Omega.
\end{bvp}
\end{equation}
This corresponds to the unique solution of \eqref{intr_prob_J_min} (and \eqref{eq_Ve_meausures}, once translated in the setting of measures). In particular, its energy plays the role of a remainder term in the eigenvalue variation.

We also emphasize that, in view of the known regularity for $\varphi\in E(\lambda_{0,N})$, by the standard elliptic theory, there holds $V_{\e,\varphi} \in W^{2,p}(\Omega)$ for every $p>1$, hence $V_{\e,\varphi}\in C^{1,\alpha}(\overline\Omega)$ for any $\alpha\in(0,1)$ (see e.g. \cite[Theorem 3.17 (ii)]{troianiello}). We now recall that $\gamma_{\e,i}$ is the $i$-th eigenvalue of the bilinear form
\begin{equation}
    h_\e(\varphi,\psi):=\frac{\lambda_{0,N}}{\e}\int_{\Omega\setminus\Omega_\e}\varphi V_{\e,\psi}\dx,\quad\textnormal{defined for }\varphi,\psi\in E(\lambda_{0,N}),
\end{equation}
see \eqref{def_h_e} and \eqref{eq:h_measures}. Hence, the study of the asymptotic behavior of $\gamma_{\e,i}$ passes through a fine analysis of $V_{\e,\varphi}$ in the limit $\e\to 0^+$. To pursue this, we first establish the following standard result (recall that $\Omega$ is assumed to have a $C^{1,1}$ boundary).

\begin{lemma}\label{lemma:I''}
Let $\Omega \subset \mathbb{R}^N$ be a bounded domain of class $C^{1,1}$ and let $u\in W^{2,1}(\Omega)$. Define
\begin{equation}
I_u(\varepsilon):=\int_{\Omega\setminus\Omega_\varepsilon}u\dx,
\end{equation}
where $\Omega_\e$ is as in \eqref{eq_p_e_Omega_e}.
Then there exists an $\varepsilon_0>0$ such that $I_u \in C^2([0, \varepsilon_0])$ and there holds
\begin{equation}\label{eq_I_derivatives}
I_u'(\varepsilon)=\int_{\partial\Omega_\varepsilon}u\ds\quad\text{and}\quad I_u''(\varepsilon)=-\int_{\partial\Omega_\varepsilon}(\partial_{\bm{\nu}_{\Omega_\varepsilon}} u + H_{\partial\Omega_\varepsilon}u)\ds,
\end{equation}
where $\bm{\nu}_{\Omega_\varepsilon}$ denotes the outer unit normal to $\partial\Omega_\varepsilon$ and $H_{\partial\Omega_\varepsilon} = \operatorname{div}(\bm{\nu}_{\Omega_\varepsilon})$ denotes the mean curvature (sum of principal curvatures) of $\partial\Omega_\varepsilon$. 
In particular, there holds
\begin{equation}\label{eq:I_taylor}
\int_{\Omega\setminus\Omega_\varepsilon}u\dx =\varepsilon\int_{\partial\Omega}u\ds
- \frac{\varepsilon^2}{2}\int_{\partial\Omega}(\partial_{\bm{\nu}} u + H_{\partial\Omega}u)\ds 
+o(\varepsilon^2) \quad \text{as } \varepsilon \to0^+,
\end{equation}
where $\nnu$ denotes the outer unit normal to $\partial\Omega$ and $H_{\partial \Omega}=\operatorname{div}(\nnu)$ denotes the mean curvature of $\partial\Omega$. Finally, let $v \in H^1(\Omega)\cap L^{\infty}(\Omega)$ and $ u \in H^1(\Omega)$. Then there exists a constant $C>0$, depending only on $\Omega$, such that for any $\e<\e_0$
\begin{multline}\label{ineq_uv_trace_ste}
\left|\int_{\partial \Omega} uv \ds - \frac{1}{\e}\int_{\Omega\setminus \Omega_\e} uv \dx \right| \\
\le C(\norm{v}_{L^\infty(\Omega)}+\norm{v}_{H^1(\Omega)})\e^\frac12 \left(\int_\Omega |\nabla u|^2 \dx +\frac{1}{\e}\int_{\Omega\setminus \Omega_\e} u^2 \dx \right)^\frac12. 
\end{multline}
\end{lemma}

\begin{proof}
Let $\varepsilon_0 > 0$ be the reach of $\partial\Omega$, so that the distance function $\delta(x) = \operatorname{dist}(x, \partial\Omega)$ is $C^{1,1}$ on $\Omega \setminus \Omega_{\varepsilon_0}$. By the coarea formula, since $|\nabla \delta| = 1$ almost everywhere in this tubular neighborhood, we have
\begin{equation}
I_u(\varepsilon) = \int_0^\varepsilon \int_{\partial\Omega_\rho} u\ds\,d\rho.
\end{equation}
Differentiating with respect to $\varepsilon$, we trivially obtain that $I_u$ is absolutely continuous and
\begin{equation}
I_u'(\varepsilon)=\int_{\partial\Omega_\varepsilon}u\ds, \quad \text{for a.e. } \varepsilon \in [0, \varepsilon_0).
\end{equation}
For $\varepsilon < \varepsilon_0$, the distance function satisfies
\begin{equation}\label{eq:delta_nu}
\nabla \delta(x)=-\bm{\nu}_{\Omega_\varepsilon}(x)\quad\textnormal{for every }x\in\partial\Omega_\varepsilon.
\end{equation}
Taking the divergence, since $H_{\partial\Omega_\varepsilon}$ is defined as the sum of the principal curvatures with respect to the outward normal $\bm{\nu}_{\Omega_\varepsilon}$, we have
\begin{equation}
\Delta\delta(x)= \operatorname{div}(-\bm{\nu}_{\Omega_\varepsilon}(x)) = -H_{\partial\Omega_\varepsilon}(x) \quad \text{a.e. on } \partial\Omega_\varepsilon.
\end{equation}
Let $0 \le \varepsilon_1 < \varepsilon_2 < \varepsilon_0$. We apply the divergence theorem to the vector field $u\nabla \delta \in W^{1,1}(\Omega \setminus \Omega_{\varepsilon_0}; \mathbb{R}^N)$ over the annular region $A = \Omega_{\varepsilon_1}\setminus\Omega_{\varepsilon_2}$. The boundary of $A$ consists of $\partial\Omega_{\varepsilon_1}$ (with outward normal $\bm{\nu}_{\Omega_{\varepsilon_1}}$) and $\partial\Omega_{\varepsilon_2}$ (with outward normal $-\bm{\nu}_{\Omega_{\varepsilon_2}}$). Therefore, thanks to \eqref{eq:delta_nu}, we have
\begin{align}
\int_{\Omega_{\varepsilon_1}\setminus\Omega_{\varepsilon_2}}\operatorname{div}(u\nabla \delta)\dx 
&= -\int_{\partial\Omega_{\varepsilon_2}} u \nabla \delta \cdot \bm{\nu}_{\Omega_{\varepsilon_2}}\ds + \int_{\partial\Omega_{\varepsilon_1}} u \nabla \delta \cdot\bm{\nu}_{\Omega_{\varepsilon_1}}\ds \\
&=- \int_{\partial\Omega_{\varepsilon_2}} u (-\bm{\nu}_{\Omega_{\varepsilon_2}}) \cdot \bm{\nu}_{\Omega_{\varepsilon_2}}\ds + \int_{\partial\Omega_{\varepsilon_1}} u (-\bm{\nu}_{\Omega_{\varepsilon_1}}) \cdot \bm{\nu}_{\Omega_{\varepsilon_1}}\ds \\
&= \int_{\partial\Omega_{\varepsilon_2}} u\ds -\int_{\partial\Omega_{\varepsilon_1}} u\ds = I_u'(\varepsilon_2) -I_u'(\varepsilon_1).
\end{align}
Rearranging this identity, we find
\begin{equation}
I_u'(\varepsilon_2)-I_u'(\varepsilon_1)=\int_{\Omega_{\varepsilon_1}\setminus\Omega_{\varepsilon_2}}\operatorname{div}(u\nabla \delta)\dx =\int_{\Omega_{\varepsilon_1}\setminus\Omega_{\varepsilon_2}}(\nabla u\cdot\nabla \delta+u\Delta\delta)\dx.
\end{equation}
Then, by the coarea formula and using that  $\nabla \delta = -\bm{\nu}_{\Omega_\varepsilon}$ and $\Delta \delta = -H_{\partial\Omega_\varepsilon}$, we have
\begin{align}
I_u'(\varepsilon_2)-I_u'(\varepsilon_1) &= \int_{\varepsilon_1}^{\varepsilon_2} \int_{\partial\Omega_\rho}(\nabla u\cdot\nabla \delta+u\Delta\delta)\ds \,d\rho \\
&= -\int_{\varepsilon_1}^{\varepsilon_2} \int_{\partial\Omega_\rho}(\partial_{\nnu_{\Omega_\rho}} u+H_{\partial\Omega_\rho}u)\ds \,d\rho 
\end{align}
which implies that $I_u'(\varepsilon)$ is absolutely continuous. Differentiating with respect to $\varepsilon$, we obtain that 
\begin{equation}
I_u''(\varepsilon)=-\int_{\partial\Omega_\varepsilon}(\partial_{\bm{\nu}_{\Omega_\varepsilon}} u + H_{\partial\Omega_\varepsilon}u)\ds,\quad\textnormal{for a.e. }\e\in[0,\e_0),
\end{equation}
which proves \eqref{eq_I_derivatives}. We finally prove that the map $\e\mapsto I''_u(\e)$ is continuous in $[0,\e_0]$. Since $\partial\Omega \in C^{1,1}$, the map $\Phi: \partial\Omega \times [0, \varepsilon_0] \to \overline{\Omega \setminus \Omega_{\varepsilon_0}}$ given by $\Phi(x, \varepsilon) = x - \varepsilon \bm{\nu}(x)$ is a bi-Lipschitz homeomorphism. This change of variables has been used e.g. in \cite{Arrieta} and the new coordinates are also known as Fermi coordinates. We denote $\Phi_\varepsilon(x) = \Phi(x, \varepsilon)$ and we let $\kappa_1(x), \dots, \kappa_{d-1}(x)$ be the principal curvatures of $\partial\Omega$ at $x$, which belong to $L^\infty(\partial\Omega)$. It is a standard procedure to derive the following identities, which hold for almost every $x \in \partial\Omega$, 
\begin{equation}\label{eq_DPhi}
\det D\Phi_\e(x) = \prod_{i=1}^{d-1} \big(1 - \varepsilon \kappa_i(x)\big),\quad H_{\partial\Omega_\varepsilon}(\Phi_\varepsilon(x)) = \sum_{i=1}^{d-1} \frac{\kappa_i(x)}{1 - \varepsilon \kappa_i(x)},\quad\textnormal{and}\quad \bm{\nu}_{\Omega_\varepsilon}(\Phi_\varepsilon(x)) = \bm{\nu}(x),
\end{equation}
from which, up to reducing $\e_0$, we deduce that
\begin{equation}\label{eq:det_DPhi}
    \frac{1}{2}\leq \det D\Phi_\e(x) \leq \frac{3}{2}\quad\text{for $\sigma$-a.e. }x\in\partial\Omega~\text{and all }\e\in[0,\e_0].
\end{equation}
and that
\begin{equation}\label{eq:det_DPhi_bis}
    \det D\Phi_\e(x) =1-\e H_{\partial \Omega}(x)+O(\e^2), \quad \text{as } \e \to 0^+.
\end{equation}
with $O(\e^2)$ uniform in $x\in\partial \Omega$.
We now claim that there exists $C>0$ such that
\begin{equation}\label{eq:claim_C2_1}
    \norm{w\circ\Phi_{\e_2}-w\circ\Phi_{\e_1}}_{L^1(\partial\Omega)}\leq C\int_{\Omega_{\e_1}\setminus\Omega_{\e_2}}|\nabla w|\dx
\end{equation}
for all $w\in W^{1,1}(\Omega\setminus\overline{\Omega_{\e_0}})$ and all $\e\in[0,\e_0]$.  Let us first assume that $w\in C^1(\overline{\Omega\setminus\Omega_{\e_0}})$. We first observe that for $x\in\partial\Omega$
\begin{equation}
    w(\Phi_{\e_2}(x))-w(\Phi_{\e_1}(x))=-\int_{\e_1}^{\e_2}\nabla w(\Phi_t(x))\cdot\nnu(x) \dt
\end{equation}
and so, by the Fubini theorem, we have
\begin{equation}
    \norm{w\circ\Phi_{\e_2}-w\circ\Phi_{\e_1}}_{L^1(\partial\Omega)}\leq\int_{\e_1}^{\e_2}\int_{\partial\Omega}|\nabla w(\Phi_t(x))|\ds\dt.
\end{equation}
We now perform the change of variables $y=\Phi_t(x)$, noticing that
\begin{equation}
    D\Phi_t(x)\nnu(x)=\nnu(x)
\end{equation}
and so, by symmetry of $D\Phi_t$, also
\begin{equation}
    D\Phi_t(x)^{-T}\nnu(x)=\nnu(x).
\end{equation}
Then, in view also of \eqref{eq:det_DPhi} and the coarea formula we get that
\begin{equation}
    \norm{w\circ\Phi_{\e_2}-w\circ\Phi_{\e_1}}_{L^1(\partial\Omega)}\leq C\int_{\e_1}^{\e_2}\int_{\partial\Omega_t}|\nabla w|\ds\dt\leq \int_{\Omega_{\e_1}\setminus\Omega_{\e_2}}|\nabla w|\dx,
\end{equation}
thus proving \eqref{eq:claim_C2_1} for $w\in C^1(\overline{\Omega\setminus\Omega_{\e_0}})$. Now, by density of $C^1(\overline{\Omega\setminus\Omega_{\e_0}})$ into $W^{1,1}(\Omega\setminus\overline{\Omega_{\e_0}})$ and continuity of the trace embedding $W^{1,1}(\Omega\setminus\overline{\Omega_{\e_0}})\hookrightarrow L^1(\partial\Omega_\e)$ for any $\e\in[0,\e_0]$, we deduce \eqref{eq:claim_C2_1}. In particular, the maps
\begin{equation}\label{eq:claim_C2_2}
    \e\mapsto u\circ \Phi_\e\quad\text{and}\quad \e\mapsto \nabla u\circ\Phi_\e\quad\text{are continuous from $[0,\e_0]$ to $L^1(\partial\Omega)$.}
\end{equation}
 Now, again by the change of variables $y=\Phi_\e(x)$, we have that
\begin{equation}
I_u''(\varepsilon) = -\int_{\partial\Omega} \Big[ \nabla u(\Phi_\varepsilon(x)) \cdot \bm{\nu}(x) + u(\Phi_\varepsilon(x)) H_{\partial\Omega_\varepsilon}(\Phi_\varepsilon(x)) \Big] \det D\Phi_\varepsilon(x) \ds(x).
\end{equation}
Therefore, using \eqref{eq:claim_C2_2} and the fact that the maps
\begin{equation}
    \e\mapsto\det D\Phi_\e(x)\quad\text{and}\quad \e\mapsto H_{\partial\Omega_\e}(\Phi_\e(x))
\end{equation}
are continuous from $[0,\e_0]$ to $L^\infty(\partial\Omega)$, we deduce that $I''_u$ is continuous, thus concluding the proof of \eqref{eq_I_derivatives}. The proof of \eqref{eq:I_taylor} is by standard Taylor expansion. 

We now turn to the proof of \eqref{ineq_uv_trace_ste}. We can write, with the change of variables $t=\e r$,
\begin{multline}\label{proof_lemma_lemma:I''_1}
\int_{\partial \Omega} uv \ds - \frac{1}{\e}\int_{\Omega\setminus \Omega_\e} uv \dx\\
=\int_{\partial\Omega}  v(x)u(x)\ds-\frac{1}{\e}\int_{\partial\Omega}\int_0^\e v(\Phi(x,t ))u(\Phi(x,t)) \det D\Phi_{t}(x) \, dt  \ds\\
=\int_{\partial\Omega} \int_0^1 [v(x)u(x)-v(\Phi(x,\e r ))u(\Phi(x,\e r)) \det D\Phi_{\e r}(x)] \, dr  \ds\\
=\int_{\partial\Omega} \int_0^1 v(x)[u(x)-u(\Phi(x,\e r))]+u(\Phi(x,\e r))[v(x)-v(\Phi(x,\e r))]  \, dr  \ds\\
-\int_{\partial\Omega}\int_0^1 v(\Phi(x,\e r )u(\Phi(x,\e r))[\det D\Phi_{\e r}(x)-1] \, dr  \ds.
\end{multline}
Now we estimate each term in \eqref{proof_lemma_lemma:I''_1}. We observe that by the Cauchy-Schwarz inequality 
\begin{multline}
|u(x)-u(\Phi(x,\e r))|\le \int_0^{\e r} |\nabla u(\Phi(x,s))\cdot \bm{\nu}(x)| \, ds  \\
\le (\e r)^\frac{1}{2}\left(\int_0^{\e r} |\nabla u(\Phi(x,s))|^2 \, ds\right)^{\frac12} 
\le \e^\frac{1}{2}\left(\int_0^{\e } |\nabla u(\Phi(x,s))|^2 \, ds\right)^{\frac12}.
\end{multline}
Integrating over $\partial \Omega \times(0,1)$ we obtain, by \eqref{eq:det_DPhi},
\begin{multline}\label{proof_lemma_lemma:I''_2}
\left|\int_{\partial\Omega} \int_0^1 v(x)[u(x)-u(\Phi(x,\e r))]  \, dr  \ds\right|\\
\le \norm{v}_{L^\infty(\Omega)} \e^\frac12\left(\int_{\partial\Omega}\int_0^{\e } |\nabla u(\Phi(x,s))|^2 \, ds\right)^{\frac12}\\
\le \sqrt{2} \norm{v}_{L^\infty(\Omega)} \e^\frac12
\left(\int_{\partial\Omega}\int_0^{\e}|\nabla u(\Phi(x,s))|^2\det D\Phi_s(x)\, ds\right)^{\frac12}\\
\le\sqrt{2} \norm{v}_{L^\infty(\Omega)} \e^\frac12\left(\int_{\Omega\setminus\Omega_\e}|\nabla u|^2 \dx\right)^{\frac12}.
\end{multline}
Similarly, by  \eqref{eq:det_DPhi}, the Cauchy-Schwarz inequality,   and a change of variables 
\begin{multline}\label{proof_lemma_lemma:I''_3}
\left|\int_{\partial\Omega} \int_0^1 u(\Phi(x,\e r))[v(x)-v(\Phi(x,\e r))]  \, dr\ds\right| \\
\le \int_{\partial\Omega}\left(\int_0^{\e } |\nabla v(\Phi(x,s))| \, ds\right)  \left(\int_0^1 |u(\Phi(x,\e r))|  \, dr\right)\ds \\
\le\left(\int_{\partial\Omega}\left(\int_0^{\e } |\nabla v(\Phi(x,s))| \, ds\right)^2 \ds \right)^\frac{1}{2} 
\left(\int_{\partial\Omega}\left(\int_0^1 |u(\Phi(x,\e r))|  \, dr\right)^2\ds \right)^\frac{1}{2} \\
\le \sqrt{2}\e^\frac{1}{2}\left(\int_{\Omega\setminus\Omega_\e}|\nabla v|^2 \dx\right)^\frac12
\left(\int_{\partial\Omega} \int_0^1 |u(\Phi(x,\e r))|^2  \, dr\ds\right)^\frac12\\
\le 2 \norm{v}_{H^1(\Omega)}\e^\frac12\left(\frac{1}{\e}\int_{\Omega\setminus\Omega_\e}u^2  \dx\right)^\frac12.
\end{multline}
Finally, \eqref{eq:det_DPhi_bis}, a change of variables and the Cauchy-Schwarz inequality yield, for some constant $C>0$ depending on $\Omega$, 
\begin{multline}\label{proof_lemma_lemma:I''_4}
\left|\int_{\partial\Omega}\int_0^1 v(\Phi(x,\e r ))u(\Phi(x,\e r))[\det D\Phi_{\e r}(x)-1] \, dr  \ds\right|\\
\le C \norm{v}_{L^\infty(\Omega)} \e \int_{\partial\Omega} \int_0^1 |u(\Phi(x,\e r))|  \, dr\ds
\le C\norm{v}_{L^\infty(\Omega)}\e\left(\frac{1}{\e}\int_{\Omega\setminus\Omega_\e}u^2  \dx\right)^\frac12,
\end{multline}
where we recall that $H_{\partial \Omega}$ is bounded since $\Omega$ is a  $C^{1,1}$ domain.
In conclusion, we can prove \eqref{ineq_uv_trace_ste} putting together \eqref{proof_lemma_lemma:I''_1}, 
\eqref{proof_lemma_lemma:I''_2}, \eqref{proof_lemma_lemma:I''_3} and \eqref{proof_lemma_lemma:I''_4}.   
\end{proof}

\begin{lemma}\label{lemma_Ve_stek}
For any $\varphi \in E(\lambda_{0,N})$,
\begin{equation}\label{limit_Ve_stek}
\frac{\la_{0,N}}{\e}\int_{\Omega\setminus\Omega_\e}\varphi V_{\e,\varphi}\dx
=\e\int_{\partial\Omega}\left[\frac{2}{3}(\la_{0,N}-1)^2+\frac{1}{2}(\la_{0,N}-1)H_{\partial \Omega}\right] \varphi^2 \ds
+o(\e),\quad\text{as }\e\to 0^+
\end{equation}
or, more precisely, as $\e \to 0^+$,
\begin{equation}\label{estimate_Ve_stek}
\sup_{\substack{\varphi\in E(\lambda_{0,N}) \\ \norm{\varphi}_{L^2(\partial\Omega)}=1}}
\left|\frac{\lambda_{0,N}}{\e^2}\int_{\Omega\setminus\Omega_\e}\varphi V_{\e,\varphi}\dx
- \int_{\partial \Omega}\left[\frac{2}{3}(\la_{0,N}-1)^2+\frac{1}{2}(\la_{0,N}-1)H_{\partial \Omega}\right] \varphi^2 \ds \right|\to 0^+.
\end{equation}
Finally, 
\begin{equation}\label{eq_Ve_stek_L2_smallo}
\sup_{\substack{\varphi\in E(\lambda_{0,N})\\ \norm{\varphi}_{L^2(\partial\Omega)}=1}}
\left(\frac{1}{\e}\int_{\Omega\setminus\Omega_\e}|V_{\e,\varphi}|^2\dx\right)=o(\e),\quad\text{as }\e\to 0^+.
\end{equation}
\end{lemma}

\begin{proof}
Just as in \Cref{lemma:I''}, there exists $\e_0>0$ such that for any $\e \in [0,\e_0]$  the map 
$\Phi: \partial\Omega \times [0, \e_0] \to \overline{\Omega \setminus \Omega_{\e_0}}$ 
given by $\Phi_\e(x)=\Phi(x, \e) = x - \e \bm{\nu}(x)$  is a bi-Lipschitz homeomorphism. 
Let also denote its inverse $\Phi^{-1}: \overline{\Omega \setminus \Omega_{\e_0}} \to \partial\Omega \times [0, \e_0]$ with  
$\Phi^{-1}=(\Phi^{-1}_1, \Phi^{-1}_2)$ where $\Phi^{-1}_1:\overline{\Omega \setminus \Omega_{\e_0}} \to \partial \Omega$ and $\Phi^{-1}_2: \overline{\Omega \setminus \Omega_{\e_0}} \to [0, \e_0]$.
It is easy to see that, denoting with $\pi(x) \in \partial \Omega$ the point of minimal distance from $x$ to $\partial \Omega$,
\begin{equation}\label{proof_lemma_Ve_stek_0}
\Phi^{-1}_1(x)=\pi(x) \quad \text{ and } \quad\Phi^{-1}_2(x)=\mathop{\rm dist}(x,\partial \Omega)=:\delta(x).
\end{equation}
Let us define for any $r \in (0,1)$, $x \in \partial \Omega$ and $\e \in (0,\e_0)$
\begin{equation}
W_\e(x,r):=\frac{1}{\e}V_{\e,\varphi}(\Phi(x,\e r))=\frac{1}{\e}V_{\e,\varphi}( x - \e r \bm{\nu}(x)).
\end{equation}
Since $\Phi$ is a bi-Lipschitz map,  $W_\e \in H^1(\partial \Omega \times (0,1))$. Furthermore,
\begin{equation}\label{proof_lemma_Ve_stek_1}
\pd{W_\e}{r}(x,r)=-\nabla V_{\e,\varphi}( x - \e r \bm{\nu}(x))\cdot\bm{\nu}(x).
\end{equation}
For the sake of clarity we divide the rest of the proof in two steps.

\noindent
\textbf{Step 1.} We claim that 
\begin{equation}\label{proof_lemma_Ve_stek_2}
\pd{W_\e}{r}(x,r) \rightharpoonup (\la_{0,N}-1)(r-1)\varphi(x) \quad \text{weakly in } L^2(\partial \Omega \times (0,1)) \text{ as } \e \to 0^+.
\end{equation}
To prove \eqref{proof_lemma_Ve_stek_2} we need some compactness. 
By testing \eqref{eq_Ve_ste} with $u=V_{\e,\varphi}$ we get that
\begin{equation}\label{proof_lemma_Ve_stek_2_5}
\int_\Omega|\nabla V_{\e,\varphi}|^2\dx+\frac{1}{\e}\int_{\Omega\setminus\Omega_\e}|V_{\e,\varphi}|^2\dx 
=(\lambda_{0,N}-1)\left[\int_{\partial\Omega}\varphi V_{\e,\varphi}\ds 
-\frac{1}{\e}\int_{\Omega\setminus\Omega_\e}\varphi V_{\e,\varphi}\dx\right]
\end{equation}
which implies, by \eqref{ineq_uv_trace_ste} with $v=\varphi$,
\begin{equation}
\int_\Omega|\nabla V_{\e,\varphi}|^2\dx+\frac{1}{\e}\int_{\Omega\setminus\Omega_\e}|V_{\e,\varphi}|^2\dx  \le C \e
\end{equation}
for some constant $C>0$ depending only on $\Omega$, $\norm{\varphi}_{L^\infty(\Omega)}$ and  $\norm{\varphi}_{H^1(\Omega)}$. In particular,
\begin{equation}\label{proof_lemma_Ve_stek_3}
\int_{\partial\Omega} \int_0^1 \left|\pd{W_\varepsilon}{r}\right|^2 \, dr \ds \le C \quad \text{ and }\quad 
\int_{\partial\Omega} \int_0^1 |W_\varepsilon|^2 \, dr \ds \le \frac{C}{\varepsilon}.
\end{equation}
Let $\zeta \in C^\infty(\partial \Omega)$ and let $\eta \in C^\infty([0,1])$ with $\eta(1)=0$.
Let us define, taking \eqref{proof_lemma_Ve_stek_0} into account,
\begin{equation}
u_\e(x):=
\begin{cases}
\zeta(\pi(x))\eta(\delta(x)/\e), &\text{ if } x \in \Omega\setminus \Omega_\e,\\
0, &\text{ if } x \in \Omega_\e.
\end{cases}
\end{equation}
We have $u_\e \in H^1(\Omega)$ since  $\eta \in C^\infty([0,1])$ with $\eta(1)=0$. Furthermore, 
\begin{multline}
\nabla u_\e(x)=\nabla (\zeta\circ \pi)(x)\eta(\delta(x)/\e)+
\frac{1}{\e}\zeta(\pi(x))\eta'(\delta(x)/\e) \nabla \delta(x)\\
=\nabla (\zeta\circ \pi)(x)\eta(\delta(x)/\e)
-\frac{1}{\e}\zeta(\pi(x))\eta'(\delta(x)/\e) \nu(\pi(x)).
\end{multline}
Testing \eqref{eq_Ve_ste} with $u_\e$ we have 
\begin{equation}
\int_\Omega\nabla V_{\e,\varphi}\cdot\nabla u_\e\dx+\frac{1}{\e}\int_{\Omega\setminus\Omega_\e}V_{\e,\varphi} u_\e\dx=(\lambda_{0,N}-1)\left(\int_{\partial\Omega}\varphi u_\e\ds-\frac{1}{\e}\int_{\Omega\setminus\Omega_\e}\varphi u_\e\ds\right),
\end{equation}
or equivalently with the change of variables $t=r\e$  by \eqref{proof_lemma_Ve_stek_1}
\begin{multline}\label{proof_lemma_Ve_stek_4}
\int_{\partial \Omega}\int_0^1 \pd{W_\e}{r}(x,r)\zeta(x)  \eta'(r) \det D\Phi_{\e r}(x) \,dr \ds \\
+\int_{ \Omega\setminus \Omega_\e} \nabla V_{\e,\varphi}\cdot \nabla (\zeta\circ \pi)(x)\eta(\delta(x)/\e)\dx 
+\e\int_{\partial \Omega}\int_0^1 W_\e(x,r)  \zeta(x)  \eta(r) \det D\Phi_{\e r}(x)\,dr \ds \\
=(\lambda_{0,N}-1)\left(\int_{\partial \Omega} \varphi(x)\zeta(x) \eta(0)\ds-\int_{\partial \Omega}\int_0^1 \varphi(\Phi(x,\e r)) \zeta(x)  \eta(r) \det D\Phi_{\e r}(x)\,dr \ds\right).
\end{multline}
By the Cauchy-Schwarz inequality,
\begin{equation}\label{proof_lemma_Ve_stek_5}
\left|\int_{ \Omega\setminus \Omega_\e} \nabla V_{\e,\varphi}\cdot 
\nabla (\zeta\circ \pi)(x)\eta(\delta(x)/\e)\dx \right|
\le C \e^\frac12 \norm{\nabla V_{\e,\varphi}}_{L^2(\Omega)},
\end{equation}
for some constant $C>0$ depending on $\Omega$, $\eta$ and  $\zeta$. Furthermore, by the Cauchy-Schwarz inequality, \eqref{eq:det_DPhi} and \eqref{proof_lemma_Ve_stek_3}
\begin{equation}\label{proof_lemma_Ve_stek_6}
\left|\int_{\partial \Omega}\int_0^1 W_\e(x,r)  \zeta(x)  \eta(r) \det D\Phi_{\e r}(x)\,dr \ds\right| 
\le C\left(\int_{\partial\Omega} \int_0^1 |W_\varepsilon|^2 \, dr \ds\right)^\frac12 \le C\varepsilon^{-\frac12},
\end{equation}
for some constant $C>0$ depending on $\Omega$, $\varphi$, $\eta$ and  $\zeta$.
Passing to the limit as $\e \to 0^+$ in \eqref{proof_lemma_Ve_stek_4}, and taking \eqref{eq:det_DPhi_bis}, \eqref{proof_lemma_Ve_stek_5} and \eqref{proof_lemma_Ve_stek_6} into account, we get that
\begin{multline}\label{proof_lemma_Ve_stek_7}
\lim_{\e \to 0^+}\int_{\partial \Omega}\int_0^1 \pd{W_\e}{r}(x,r)\zeta(x)  \eta'(r) \,dr \ds\\
=(\lambda_{0,N}-1)\left(\int_{\partial \Omega} \varphi(x)\zeta(x) \eta(0)\ds-\int_{\partial \Omega}\varphi(x) \zeta(x) \int_0^1   \eta(r) \,dr \ds\right)\\
=(\lambda_{0,N}-1)\left(\int_{\partial\Omega}\int_0^1\varphi(x) \zeta(x) (r-1)\eta'(r)\,dr \ds\right),
\end{multline}
since  $\eta(1)=0$. In view of \eqref{proof_lemma_Ve_stek_3}, 
\begin{equation}\label{proof_lemma_Ve_stek_7_5}
\pd{W_{\e_n}}{r}(x,r) \rightharpoonup w_0 \quad \text{weakly in } L^2(\partial \Omega \times (0,1)) \text{ as } n \to \infty,
\end{equation}
for some $w_0 \in L^2(\partial \Omega \times (0,1))$ and some sequence $\e_n \to 0^+$. We now let $\xi\in C^\infty_c(0,1)$ and choose
\begin{equation}
    \eta(t)=\int_0^t \xi(s)\,ds-\int_0^1 \xi(s)\,ds
\end{equation}
so that, since $\eta \in C^\infty([0,1])$ with $\eta(1)=0$, by \eqref{proof_lemma_Ve_stek_7} and \eqref{proof_lemma_Ve_stek_7_5}
we conclude that
\begin{equation}
    \int_{\partial\Omega}\int_0^1 w_0(x,r)\zeta(x)\xi(r)\ds\,dr=(\lambda_{0,N}-1)\left(\int_{\partial\Omega}\int_0^1\varphi(x) \zeta(x) (r-1)\xi(r)\,dr \ds\right),
\end{equation}
for all $\zeta\in C^\infty(\partial\Omega)$ and all $\xi\in C_0^\infty(0,1)$. It follows that 
$w_0(x,r)=(\lambda_{0,N}-1)(r-1)\varphi(x)$ thus by the Urysonh subsequence principle, we have proved \eqref{proof_lemma_Ve_stek_2}.

\noindent
\textbf{Step 2.} We now prove \eqref{limit_Ve_stek}. Testing \eqref{eq_Ve_ste} with $\varphi$ we obtain 
\begin{multline}\label{proof_lemma_Ve_stek_8}
\frac{\la_{0,N}}{\e}\int_{\Omega\setminus\Omega_\e}V_{\e,\varphi}\varphi\dx 
=(\lambda_{0,N}-1)\left[\int_{\partial\Omega}\varphi^2\ds 
-\frac{1}{\e}\int_{\Omega\setminus\Omega_\e}\varphi^2\dx\right]-\int_\Omega\nabla V_{\e,\varphi}\cdot \nabla \varphi\dx \\
+\frac{\la_{0,N}-1}{\e}\int_{\Omega\setminus\Omega_\e}V_{\e,\varphi}\varphi\dx 
=(\lambda_{0,N}-1)\left[\int_{\partial\Omega}\varphi^2\ds 
-\frac{1}{\e}\int_{\Omega\setminus\Omega_\e}\varphi^2\dx\right]\\
-(\la_{0,N}-1)\left[\int_{\partial\Omega}\varphi V_{\e,\varphi}\ds
-\frac{1}{\e}\int_{\Omega\setminus\Omega_\e}V_{\e,\varphi}\varphi\dx\right],
\end{multline} 
by \eqref{eq_ste_eigen} tested with $V_{\e,\varphi}$. Let us compute the asymptotic behavior of the right hand side of \eqref{proof_lemma_Ve_stek_8}. By \eqref{eq:I_taylor}
\begin{multline}\label{proof_lemma_Ve_stek_9}
\int_{\partial\Omega}\varphi^2\ds 
-\frac{1}{\e}\int_{\Omega\setminus\Omega_\e}\varphi^2\dx
=\frac{\varepsilon}{2}\int_{\partial\Omega}(\partial_{\bm{\nu}} (\varphi^2) 
+ H_{\partial\Omega}\varphi^2)\ds 
+o(\varepsilon)\\
=\frac{\varepsilon}{2}\int_{\partial\Omega} (2(\la_{0,N}-1) + H_{\partial\Omega})\varphi^2\ds 
+o(\varepsilon) \quad \text{as } \varepsilon \to0^+.
\end{multline}
Since $\varphi$ is Lipschitz, by \eqref{proof_lemma_Ve_stek_3}
\begin{equation}\label{proof_lemma_Ve_stek_9_5}
\left|\int_{\partial \Omega} \int_0^1 [\varphi(x)-\varphi(\Phi(x,\e r))] W_\e(x,r) \, dr \ds\right| 
\le C \e \int_{\partial \Omega} \int_0^1 |W_\e(x,r)| \, dr \ds \le C \e^\frac12. 
\end{equation}
Furthermore, the change of variables $t=\e r$, \eqref{eq:det_DPhi_bis},  \eqref{proof_lemma_Ve_stek_2} and \eqref{proof_lemma_Ve_stek_9_5} yield 
\begin{multline}\label{proof_lemma_Ve_stek_10}
\int_{\partial\Omega}\varphi V_{\e,\varphi}\ds
-\frac{1}{\e}\int_{\Omega\setminus\Omega_\e}V_{\e,\varphi}\varphi\dx
=\e\int_{\partial\Omega} \varphi(x) W_\e(x,0)\ds \\
-\e\int_{\partial \Omega} \int_0^1 \varphi(\Phi(x,\e r)) W_\e(x,r)(1-\e r H_{\partial \Omega}+O(\e^2)) \, dr\ds\\
=\e\int_{\partial\Omega} \varphi(x) \left[W_\e(x,0)-\int_0^1W_\e(x,r)\, dr\right]\ds +o(\e)\\
=\e\int_{\partial\Omega} \varphi(x) \int_0^1(r-1) \pd{W_\e}{r}(x,r)\, dr \ds +o(\e)
=\e(\la_{0,N}-1) \int_0^1(r-1)^2  \, dr \int_{\partial\Omega} \varphi^2 \ds +o(\e)\\
=\frac{1}{3}(\la_{0,N}-1)\left(\int_{\partial\Omega} \varphi^2 \ds\right) \e  +o(\e) \quad \text{ as } \e \to 0^+.
\end{multline}
Hence, \eqref{limit_Ve_stek} follows from \eqref{proof_lemma_Ve_stek_8}, \eqref{proof_lemma_Ve_stek_9} and \eqref{proof_lemma_Ve_stek_10}.
Finally \eqref{estimate_Ve_stek} follows from \eqref{limit_Ve_stek} since $E(\la_{0,N})$ is a finite dimensional vector space.

\noindent
\textbf{Step 3.} We now prove \eqref{eq_Ve_stek_L2_smallo}. Let $\varphi \in E(\la_{0,N})$
with $\norm{\varphi}_{L^2(\partial\Omega)}=1$. By \eqref{proof_lemma_Ve_stek_2_5} and \eqref{proof_lemma_Ve_stek_10}
\begin{equation}\label{proof_lemma_Ve_stek_11}
\limsup_{\e\to 0^+}\frac{1}{\e}\int_\Omega|\nabla V_{\e,\varphi}|^2\dx
\le \frac{1}{3}(\la_{0,N}-1)^2\int_{\partial\Omega} \varphi^2 \ds.
\end{equation}
Furthermore, by a change of variables and \eqref{eq:det_DPhi_bis}
\begin{multline}
\int_\Omega|\nabla V_{\e,\varphi}|^2\dx \ge \int_{\Omega\setminus \Omega_\e}|\nabla V_{\e,\varphi}\cdot \nu(\pi(x))|^2\dx
=\e\int_{\partial \Omega} \int_0^1\left|\pd{W_\e}{r}(x,r)\right|^2 \det D \Phi_{\e r} \, dr \ds\\
=\e\int_{\partial \Omega} \int_0^1\left|\pd{W_\e}{r}(x,r)\right|^2 \, dr \ds +o(\e), \quad  \text{ as } \e \to0^+.
\end{multline}
It follows that, by \eqref{proof_lemma_Ve_stek_2} and the lower semicontinuity of norms,
\begin{multline}\label{proof_lemma_Ve_stek_12}
\liminf_{\e\to 0^+}\frac{1}{\e}\int_\Omega|\nabla V_{\e,\varphi}|^2\dx\
\ge \liminf_{\e\to 0^+}\int_0^1\left|\pd{W_\e}{r}(x,r)\right|^2 \, dr \ds\\
\ge (\la_{0,N}-1)^2\int_0^1(r-1)^2 \, dr  \int_{\partial \Omega}\varphi^2\ds
=\frac{1}{3}(\la_{0,N}-1)^2\int_{\partial\Omega} \varphi^2 \ds.
\end{multline}
Putting together \eqref{proof_lemma_Ve_stek_11} and \eqref{proof_lemma_Ve_stek_12}, we have shown that 
\begin{equation}
\lim_{\e\to 0^+}\frac{1}{\e}\int_\Omega|\nabla V_{\e,\varphi}|^2\dx=\frac{1}{3}(\la_{0,N}-1)^2\int_{\partial\Omega} \varphi^2 \ds,
\end{equation}
while by  \eqref{proof_lemma_Ve_stek_2_5} and \eqref{proof_lemma_Ve_stek_10}
\begin{equation}
\lim_{\e\to 0^+}\frac{1}{\e}\int_\Omega|\nabla V_{\e,\varphi}|^2\dx +\frac{1}{\e^2}\int_{\Omega\setminus \Omega_\e}|V_{\e,\varphi}|^2\dx
=\frac{1}{3}(\la_{0,N}-1)^2\int_{\partial\Omega} \varphi^2 \ds.
\end{equation}
It follows that for any $\varphi \in E(\la_{0,N})$ with $\norm{\varphi}_{L^2(\partial\Omega)}=1$
\begin{equation}
\frac{1}{\e}\int_{\Omega\setminus \Omega_\e}|V_{\e,\varphi}|^2\dx=o(\e), \quad \text{ as } \e \to 0^+.
\end{equation}
Hence, we have proved  \eqref{eq_Ve_stek_L2_smallo} since $E(\la_{0,N})$ is a finite dimensional vector space.

\end{proof}

We are now finally able to prove the main result of the present section, quantifying the convergence of eigenvalues in the Neumann-to-Steklov approximation.
\begin{theorem}\label{theor_eigen_var_measures_ste}
Let $\{\lambda_{\e,n}\}_n$ be the eigenvalues of \eqref{eq_neu_to_stek} and $\{\lambda_{0,n}\}_n$ be the eigenvalues of \eqref{eq:steklov}. Let $\lambda_{0,N}$ be of multiplicity $M_{0,N}$ and such that
\begin{equation}
    \lambda_{0,N-1}<\lambda_{0,N}=\cdots=\lambda_{0,N+M_{0,N}-1}<\lambda_{0,N+M_{0,N}}.
\end{equation}
Let $\{\gamma_i\}_{i=1,\dots,M_{0,N}}$ be the eigenvalues (in increasing order) of the bilinear form
\begin{equation}
    h_0(\varphi,\psi):=\int_{\partial\Omega}\left[ \frac{2}{3}(\lambda_{0,N}-1)^2+\frac{1}{2}(\lambda_{0,N}-1)H_{\partial\Omega}\right]\varphi\psi\ds,
\end{equation}
defined for $\varphi,\psi\in E(\lambda_{0,N})$. Then there holds
\begin{equation}\label{eq_eigen_first_order_measures}
\lambda_{\e,N+i-1}=\lambda_{0,N}+\e\gamma_i+ o(\e) \quad \text{ as } \e \to 0^+.
\end{equation} 
In particular, if $\lambda_{0,N}$ is simple and $\varphi\in H^1(\Omega)$ is a corresponding $L^2(\partial\Omega)$-normalized eigenfunction, then
\begin{equation}
    \lambda_{\e,N}=\lambda_{0,N}+\e \int_{\partial\Omega}\left[ \frac{2}{3}(\lambda_{0,N}-1)^2+\frac{1}{2}(\lambda_{0,N}-1)H_{\partial\Omega}\right]\varphi^2\ds+o(\e)\quad\textnormal{as }\e\to 0^+.
\end{equation}
\end{theorem}
\begin{proof}
    We apply \Cref{theor_eigen_var_measures} with $\alpha_\e=0$ and $\beta_\e$ as in \eqref{eq:def_steklov_beta}. Let $h_\e$ be the bilinear form as in \eqref{eq:h_measures}, i.e.
    \begin{equation}
        h_\e(\varphi,\psi)=\frac{\lambda_{0,N}}{\e}\int_{\Omega\setminus\Omega_\e}\varphi V_{\e,\psi}\dx.
    \end{equation}
    From \Cref{lemma_Ve_stek} and the identity
    \begin{equation}
        h_\e(\varphi,\psi)=\frac{1}{2}\left(h_\e(\varphi,\varphi)+h_\e(\psi,\psi)-h_\e(\varphi-\psi,\varphi-\psi)\right)
    \end{equation}
    we derive that
    \begin{equation}
        \sup\left\{\frac{1}{\e}|h_\e(\varphi,\psi)-\e h_0(\varphi,\psi)|\colon \varphi,\psi\in E(\lambda_{0,N}),~\norm{\varphi}_{L^2(\partial\Omega)} =\norm{\psi}_{L^2(\partial\Omega)}=1\right\}\to 0\quad\textnormal{as }\e\to 0^+.
    \end{equation}
    If $\{\gamma_{\e,i}\}_{i=1,\dots,M_{0,N}}$ are the eigenvalues of $h_\e$ on $E(\lambda_{0,N})$ as in \Cref{theor_eigen_var_measures}, then we have that
    \begin{equation}\label{eq_gamma_e_0}
        \gamma_{\e,i}=\e\gamma_i+o(\e)\quad\textnormal{as }\e\to 0^+.
    \end{equation}
    We now turn to the analysis of $\delta_\e$ and $\tau_\e$ as in \Cref{theor_eigen_var_measures}. By definition and from \eqref{eq_gamma_e_0}, we have that
    \begin{equation}
        \tau_\e^2=\max_{i=1,\dots,M_{0,N}}|\gamma_{\e,i}|=O(\e)\quad\textnormal{as }\e\to 0^+.
    \end{equation}
    Moreover, by definition and \eqref{eq_Ve_stek_L2_smallo}
    \begin{equation}
        \delta_\e^2=\sup_{\substack{\varphi\in E(\lambda_{0,N})\\ \norm{\varphi}_{L^2(\partial\Omega)}=1}}
\left(\frac{1}{\e}\int_{\Omega\setminus\Omega_\e}|V_{\e,\varphi}|^2\dx\right)=o(\e),\quad\text{as }\e\to 0^+,
    \end{equation}
    which concludes the proof in view of \Cref{theor_eigen_var_measures}.
\end{proof}

\subsection{A Neumann to Robin approximation}\label{subsec_neum_to_rob}
In this subsection, we analyze the approximation of a Robin eigenvalue problem with weighted Neumann ones in the same spirit of \Cref{subsec_steklov}. Since the path and all the arguments are completely analogous to the ones in \Cref{subsec_steklov}, we omit all the proofs in the present subsection.

We let $M=\R^d$ and $\Omega\subseteq\R^d$ an open, bounded, connected set with $C^{1,1}$ boundary. Moreover, we consider the family of measures
\begin{equation}\label{eq:def_Rob_beta}
    \alpha_\e=ap_\varepsilon \mathcal{L}^d\quad\textnormal{and}\quad\beta_\varepsilon=\mathcal{L}^d,
\end{equation}
where $p_\varepsilon$ is as in \eqref{eq_p_e_Omega_e} and $a>0$. Finally, we let
\begin{equation}
\alpha_0=a\mathcal{H}^{d-1}\text{\huge$\llcorner$}\,\partial\Omega.\quad\textnormal{and}\quad\beta_0=\mathcal{L}^d.
\end{equation}
Within this framework,  the limit problem is a Robin eigenvalue problem
\begin{equation}\label{eq:Rob}
\begin{cases}
-\Delta\varphi +\varphi=\lambda \varphi, &\text{in }\Omega, \\
\partial_{\nnu} \varphi +a\varphi=0, &\text{on }\partial\Omega,
\end{cases}
\end{equation}
whose weak formulation is $\varphi\in H^1(\Omega)$ and
\begin{equation}\label{eq_Rob_eigen}
\int_\Omega \nabla\varphi\cdot\nabla v \dx+a\int_{\partial \Omega}\varphi v\ds 
=(\lambda-1)\int_\Omega\varphi v \dx \quad\textnormal{for all }v\in H^1(\Omega).
\end{equation}
On the other hand, the approximating problems are the Neumann eigenvalue problems
\begin{equation}\label{eq_neu_to_Rob}
\begin{cases}
-\Delta \varphi + \varphi+ ap_\e \varphi =\lambda \varphi , &\text{in }\Omega, \\
\partial_{\nnu} \varphi =0, &\text{on }\partial\Omega,
\end{cases}
\end{equation}
 whose formulation is $\varphi\in H^1(\Omega)$ and
\begin{equation}
\int_\Omega\nabla\varphi\cdot\nabla v\dx
+\frac{a}{\e}\int_{\Omega\setminus\Omega_\e}\varphi v\dx=(\la-1)\int_\Omega\varphi v\dx\quad\textnormal{for all }v\in H^1(\Omega).
\end{equation} 
Now we consider a Robin (limit) eigenvalue $\lambda_{0,N}$ of problem \eqref{eq:Rob} with multiplicity $M_{0,N}\geq 1$ such that
\begin{equation}
    \lambda_{0,N-1}<\lambda_{0,N}=\cdots=\lambda_{0,N+M_{0,N}-1}<\lambda_{0,N+M_{0,N}},
\end{equation}
and a family of $M_{0,N}$ (approximating) eigenvalues $\lambda_{\e,N+i-1}$ of \eqref{eq:Rob}, with $i=1,\dots,M_{0,N}$, which are converging to $\lambda_{0,N}$. 

Since we want to apply \Cref{theor_eigen_var_measures} in this setting, we need to verify all the assumptions: this boils down to check  
the admissibility of the families  $\bm{\alpha}=\{\alpha_\e\}_{\e\in[0,1]}$ and $\bm{\beta}=\{\beta_\e\}_{\e\in[0,1]}$ in the sense of \Cref{def:adm}, just as in \Cref{subsec_steklov}. However, 
the admissibility of $\alpha_\e$ was already proved in \Cref{lemma_add_ste} while the admissibility of $\beta_\e$ is trivial. 
Hence, the assumptions of \Cref{theor_eigen_var_measures} are satisfied. Next, we study the asymptotic behavior (as $\e\to 0^+$) of the family of eigenvalues $\{\gamma_{\e,i}\}_{i=1,\dots,M_{0,N}}$ as in \Cref{theor_eigen_var_measures}, which in turn yields the asymptotic expansion of the eigenvalue variation.

To begin with, let us recall some facts. For any $\varphi\in E(\lambda_{0,N})$, we let $V_{\e,\varphi}\in H^1(\Omega)$ be the unique solution of 
\begin{equation}\label{eq_Ve_Rob}
\int_\Omega[\nabla V_{\e,\varphi}\cdot\nabla u+ V_{\e,\varphi} u]\dx 
+\frac{a}{\e}\int_{\Omega\setminus\Omega_\e}V_{\e,\varphi} u\dx=\frac{a}{\e}\int_{\Omega\setminus\Omega_\e}\varphi u\dx-a\int_{\partial \Omega} \varphi u \ds,
\end{equation}
for any $u\in H^1(\Omega)$ whose corresponding strong formulation is the following
\begin{equation}\label{prob_Ve_Rob}
\begin{bvp}
-\Delta V_{\e,\varphi}+V_{\e,\varphi}&=0 , &&\text{in }\Omega_\e, \\
 -\Delta V_{\e,\varphi}+V_{\e,\varphi} +\frac{a}{\e}V_{\e,\varphi}&=\frac{a}{\e}\varphi, &&\text{in }\Omega\setminus\Omega_\e, \\
\partial_{\nnu} V_{\e,\varphi}&=-a\varphi, &&\text{on }\partial\Omega.
\end{bvp}
\end{equation}
By the known regularity for $\varphi\in E(\lambda_{0,N})$ and  by standard elliptic regularity theory, there holds $V_{\e,\varphi} \in W^{2,p}(\Omega)$ for every $p>1$, hence $V_{\e,\varphi}\in C^{1,\alpha}(\overline\Omega)$ for any $\alpha\in(0,1)$ (see e.g. \cite[Theorem 3.17 (ii)]{troianiello}). We now recall that $\gamma_{\e,i}$ is the $i$-th eigenvalue of the bilinear form
\begin{equation}
h_\e(\varphi,\psi):=\lambda_{0,N}\int_\Omega\varphi V_{\e,\psi}\dx,
\quad\textnormal{defined for }\varphi,\psi\in E(\lambda_{0,N}),
\end{equation}
we refer to \eqref{def_h_e} and \eqref{eq:h_measures}.

Now, arguing as in \Cref{lemma_Ve_stek}, we can then prove the following.

\begin{lemma}\label{lemma_Ve_rob}
For any $\varphi \in E(\lambda_{0,N})$
\begin{equation}\label{limit_Ve_Rob}
\la_{0,N}\int_{\Omega}\varphi V_{\e,\varphi}\dx
=\e a\int_{\partial\Omega}\left[\frac{2}{3}{a}-\frac{1}{2}H_{\partial \Omega}\right] \varphi^2\ds
+o(\e),\quad\text{as }\e\to 0^+
\end{equation}
or, more precisely, as $\e \to 0^+$,
\begin{equation}\label{estimate_Ve_Rob}
\sup_{\substack{\varphi\in E(\lambda_{0,N}) \\ \norm{\varphi}_{L^2(\partial \Omega)}=1}}
\left|\frac{\la_{0,N}}{\e}\int_{\Omega}\varphi V_{\e,\varphi}\dx
- a\int_{\partial\Omega}\left[\frac{2}{3}{a}-\frac{1}{2}H_{\partial \Omega}\right] \varphi^2 \ds \right|\to 0^+.
\end{equation}
Finally, 
\begin{equation}
\sup_{\substack{\varphi\in E(\lambda_{0,N})\\ \norm{\varphi}_{L^2(\partial\Omega)}=1}}
\int_{\Omega}|V_{\e,\varphi}|^2\dx=o(\e),\quad\text{as }\e\to 0^+.
\end{equation}
\end{lemma}

The main result of the present section is the following theorem, quantifying the convergence of eigenvalues in the Neumann-to-Robin approximation. It can be proved as \Cref{theor_eigen_var_measures_ste}.
\begin{theorem}\label{theor_eigen_var_measures_Rob}
Let $\{\lambda_{\e,n}\}_n$ be the eigenvalues of \eqref{eq_neu_to_Rob} and $\{\lambda_{0,n}\}_n$ be the eigenvalues of \eqref{eq:Rob}. Let $\lambda_{0,N}$ be of multiplicity $M_{0,N}$ and such that
\begin{equation}
    \lambda_{0,N-1}<\lambda_{0,N}=\cdots=\lambda_{0,N+M_{0,N}-1}<\lambda_{0,N+M_{0,N}}.
\end{equation}
Let $\{\gamma_i\}_{i=1,\dots,M_{0,N}}$ be the eigenvalues (in increasing order) of the bilinear form
\begin{equation}
    h_0(\varphi,\psi):=a\int_{\partial\Omega}\left[\frac{2}{3}{a}-\frac{1}{2}H_{\partial \Omega}\right]\varphi\psi\ds,
\end{equation}
defined for $\varphi,\psi\in E(\lambda_{0,N})$. Then there holds
\begin{equation}
\lambda_{\e,N+i-1}=\lambda_{0,N}+\e\gamma_i+ o(\e) \quad \text{ as } \e \to 0^+.
\end{equation} 
In particular, if $\lambda_{0,N}$ is simple and $\varphi\in H^1(\Omega)$ is a corresponding $L^2(\Omega)$-normalized eigenfunction, then
\begin{equation}
\lambda_{\e,N}=\lambda_{0,N}+\e a\int_{\partial\Omega}\left[\frac{2}{3}a
-\frac{1}{2}H_{\partial\Omega}\right]\varphi^2+o(\e)\quad\textnormal{as }\e\to 0^+.
\end{equation}
\end{theorem}

\section{Connected sum of manifolds}\label{sec_sum_mani}
As stated in \Cref{sec_intro}, in this section given two Riemannian manifolds $(M_0,g_0)$ and $(M_1,g_1)$, and  two points $p_0\in M_0$ and $p_1\in M_1$, we consider the topological manifold
\begin{equation}
 M_\e:=\left(M_0\setminus B_\e^{M_0}(p_0)\right)\cup_{\Phi_\e} \left( M_1\setminus B_1^{M_1}(p_1) \right),
\end{equation}
endowed with the metric 
\begin{equation}
g_\e:=\begin{cases}
g_0, &\text{in }M_0\setminus B_\e^{M_0}(p_0), \\
\e^2 g_1,&\text{in } M_1\setminus B_1^{M_1}(p_1),
\end{cases}
\end{equation}
By \cite[Theorem 1.1]{Ta06} we know that, under suitable assumptions on the metrics $g_0$ and $g_1$,  the eigenvalues (of the Laplace-Beltrami operator) of $(M_\e,g_\e)$ converge, as $\e\to 0^+$ to the ones of $(M_0,g_0)$. In this section, we explicitly quantify the rate of convergence. 

\subsection{Setting of the problem}
Let $(M_0,g_0)$ and $(M_1,g_1)$ be two closed, connected, oriented, $m$-dimensional Riemannian manifolds with $m \ge 3$. 
Up to rescaling the metrics $g_i$, we can assume that the injectivity radius of each $M_i$ is greater than $1$.

Fix $p_0\in M_0$ and $p_1\in M_1$ and assume there exists $\{\Phi_\e\}_\e$ a one-parameter family of diffeomorphisms $\Phi_\e:\overline{B_\e^{M_0}(p_0)}\to \overline{B_1^{M_1}(p_1)}$ that restricts to a one-parameter family of isometries 
\begin{align}
    \Phi_\e:(\partial B_\e^{M_0}(p_0),g_0|_{\partial B_\e^{M_0}(p_0)})\to(\partial B_1^{M_1}(p_1),\e^2 g_1|_{\partial B_1^{M_1}(p_1)}).
\end{align}
Moreover, we assume that the family $\{\Phi_\e\}_\e$ is smooth with respect to $\e$ and that
\begin{itemize}
\item $\Phi_\e(p_0)=p_1$;
\item there exists $\Psi\in O(T_{p_1}M_1;T_{p_0}M_0)$ so that $F_\e:=\frac{1}{\e}(\exp_{p_0}^{M_0})^{-1}\circ \Phi_\e^{-1}\circ \exp_{p_1}^{M_1}\to \Psi$ in $C^{1}(\overline{\mathbb{B}_1};T_{p_0}M_0)$ as $\e \to 0^+$, for $\mathbb{B}_1\subset T_{p_1}M_1$. In particular, since $\Psi$ is linear, this implies that $\frac{1}{\e}\dint((\exp_{p_0}^{M_0})^{-1}\circ \Phi_\e^{-1}\circ \exp_{p_1}^{M_1})\to \Psi$ in $C^{0}(T_{p_1}M_1;T_{p_0}M_0)$ as $\e \to 0^+$.
\end{itemize}}
For any $\e \in (0,1)$ we denote
\begin{align}
    M_0(\e):=M_0 \setminus B_\e^{M_0}(p_0) \quad \andd \quad M_1(1):=M_1 \setminus B_1^{M_1}(p_1),
\end{align}
where $B_r^{M_i}(p_i)\subset M_i$ is the geodesic ball of radius $r>0$ centered at $p_i$ for $i=0,1$. We define the \textit{connected sum of $M_0$ and $M_1$} as the topological manifold
\begin{align}
    M_\e:=M_0(\e) \cup_{\Phi_\e} M_1(1),
\end{align}
where we identify $\partial B_\e^{M_0}(p_0)$ with $\partial B_1^{M_1}(p_1)$ through the isometry $\Phi_\e$. Once fixed the orientations of $M_0$ and $M_1$, we also require that $\Phi_\e$ is such that $M_\e$ is oriented.

Let us define the following continuous and piecewise smooth Riemannian metric defined on $M_\e$
\begin{align}
    g_\e := \begin{cases}
        g_0 & \inn M_0(\e)\\ \e^2 g_1 & \inn M_1(1).
    \end{cases}
\end{align}
Following \cite{Ta02} and its notations, we introduce the following Lebesgue and Sobolev spaces on $(M_\e,g_\e)$:
\begin{align}
 L^2(M_\e,g_\e):=&L^2(M_0(\e),g_0)\times L^2(M_1(1),\e^2g_1)\\
 H^1(M_\e, g_\e):=&\left\{f=(f_0,f_1)\in H^1(M_0(\e),g_0)\times H^1(M_1(1),\e^2g_1)\ :\right.\\
 &\left. \ \ f_0|_{\partial M_0(\e)}= f_1|_{\partial M_1(1)}\circ \Phi_\e\ \inn L^2(\partial M_0(\e),  g_0)\right\}\\ 
 H^2(M_\e,g_\e):=&\left\{f=(f_0,f_1)\in H^2(M_0(\e),g_0)\times H^2(M_1(1),\e^2g_1)\ :\right.\\
 &\ \ f_0|_{\partial M_0(\e)}= f_1|_{\partial M_1(1)}\circ \Phi_\e\ \inn H^1(\partial M_0(\e),  g_0),\\
 &\left. \ \ \nu_0(f_0)|_{\partial M_0(\e)}
 = -\e^{-1}\nu_1(f_1)|_{\partial M_1(1)}\circ \Phi_\e\ \inn L^2(\partial M_0(\e),  g_0)\right\},
\end{align}
where $\nu_0$ and $\nu_1$ denote the outward unit normal fields to $\partial M_0(\e)$ and $\partial M_1(1)$ respectively. These vector spaces are endowed with the natural inner products defined as the Cartesian product of those of $(M_0(\e),g_0)$ and $(M_1(1),\e^2 g_1)$.

Using the above notation, one can define the Laplacian of $(M_\e, g_\e)$, even if $g_\e$ is not smooth, as the unbounded operator
\begin{align}
    (-\Delta_\e, \mathcal{D}(-\Delta_\e)):=((-\Delta^{g_0}, -\Delta^{\e^2 g_1}), H^2(M_\e, g_\e)),
\end{align}
where $\Delta^{g}$ denotes the (non-positive definite) Laplace-Beltrami operator of the smooth manifold $(M,g)$. This operator is formally defined through the Dirichlet form $(\widetilde{Q}_\e, \mathcal{D}(\widetilde{Q}_\e))$
\begin{align}
\widetilde{Q}_\e(f,h)&=\widetilde{Q}^{g_0}(f_0,h_0) + \widetilde{Q}^{\e^2 g_1}(f_1,h_1)
\end{align}
for any $f=(f_0,f_1),h=(h_0,h_1)\in\mathcal{D}(Q_\e)=H^1(M_\e,g_\e)$, where  $\widetilde{Q}^g$ denotes the form associated with $-\Delta^g$ on the Riemannian manifold $(M,g)$, that is
\begin{align}
\widetilde{Q}^g(f,h)=\int_M g(\nabla^g f, \nabla^g h) \dvol^g \quad \forall f,h \in H^1(M,g).
\end{align}
As proved in \cite[Lemma 2.5]{Ta02}, the bilinear form $\widetilde{Q}_\e$ is, in fact, the one associated to the Laplacian $-\Delta_\e$, i.e. for any $f\in \mathcal{D}(-\Delta_\e)$ and for any $h\in \mathcal{D}(Q_\e)$
\begin{align}
\widetilde{Q}_\e(f,h)=(-\Delta_\e f, h)_{L^2(M_\e, g_\e)}.
\end{align}

In a similar way, denoted by $Q^g$ the form associated with $-\Delta^g+1$ on the Riemannian manifold $(M,g)$
\begin{align}
Q^g(f,h)=\int_M \left[g(\nabla^g f, \nabla^g h)+fh\right] \dvol^g \quad \forall f,h \in H^1(M,g),
\end{align}
we can define the unbounded operator $(-\Delta_\e+1, \mathcal{D}(-\Delta_\e+1)):=((-\Delta^{g_0}+1, -\Delta^{\e^2 g_1}+1), H^2(M_\e, g_\e))$ as the operator associated with the bilinear form
\begin{align}
Q_\e (f,h)=Q^{g_0}(g_0,h_0)+Q^{\e^2 g_1}(f_1,h_1)=(f,h)_{H^1(M_\e,g_\e)} \quad \forall f,h \in H^1(M,g).
\end{align}

An important remark is that $-\Delta_\e$, and hence $-\Delta_\e+1$, displays the same spectral properties as its smooth counterpart defined on smooth closed Riemannian manifolds (see \cite[Section 1]{colette-colbois}):
\begin{itemize}
    \item $-\Delta_\e+1 > 0$ in the spectral sense;
    \item $-\Delta_\e+1$ is self-adjoint and elliptic;
    \item $\sigma(-\Delta_\e+1)=\{\lambda_{\e,n}=\lambda_n(M_\e,g_\e)\}_n$ is discrete.
\end{itemize}

In \cite[Theorem 1.1]{Ta02} the author managed to prove the convergence of the spectrum of $\Delta_\e$ to the one of $\Delta^{g_0}$ in the limit as $\e \to 0^+$, i.e. in the limit as the manifold $(M_1(1),\e^2 g_1)$ is collapsing and $(M_0(\e),g_0)$ is exhausting $M_0$. 
Referred to the operator $-\Delta_\e+1$, at the limit as $\e \to 0^+$ we have the eigenvalue problem  
\begin{equation}
-\Delta^{g_0}\psi +\psi =\la \psi \qquad\textnormal{in }M_0,
\end{equation}
with weak formulation
\begin{equation}\label{eq_psi_connect_sum}
  \int_{M_0} \left[g_0 \left(\nabla^{g_0} \psi, \nabla^{g_0} w\right)+\psi w \right] \dvol^{g_0}= \la \int_{M_0}\psi w  \dvol^{g_0} \quad \quad \forall w \in H^1(M_0,g_0).
\end{equation}
Any eigenfunction is actually smooth since $(M_0,g_0)$ is a smooth Riemannian manifold.
We denote with $\{\la_{0,n}\}_n$ the non-decreasing sequence of the eigenvalues of \eqref{eq_psi_connect_sum} counted with multiplicity.
\bigskip

The aim of the present section is to provide a quantitative estimate of this convergence using the tools and results presented so far. Moreover, since for any $\e \in [0,1]$ the first eigenvalue of $-\Delta_\e+1$ on $M_\e$ is $1$, and the correspondent eigenfunction is constant, it follows that $\la_{1,\e}=\la_{1,0}$ for any $\e \in [0,1]$. Hence, it is not restrictive to investigate quantitative spectral stability only for higher eigenvalues.

Referring to the notation introduced in \Cref{sec_assumptions} and \Cref{sec_egin_var}, in the context we are considering, we have
\begin{multicols}{2}
\begin{itemize}
\item $H_0=L^2(M_0,g_0)$;
\item $T_0=-\Delta^{g_0}+1$;
\item $\mathcal{E}^{(0)}=Q^{g_0}$;
\item $\mathcal{F}_0=H^{1}(M_0,g_0)$;
\item $Z_0=H^{1}(M_0,g_0)$;
\item $\mathcal{Z}_0=L^{2}(M_0,g_0)$;
\item $H_\e=L^2(M_\e,g_\e)$;
\item $T_\e=-\Delta_\e+1$;
\item $\mathcal{E}^{(\e)}=Q_\e$;
\item $\mathcal{F}_\e= H^{1}(M_\e,g_\e)$;
\item $Z_\e=H^{1}(M_\e,g_\e)$;
\item $\mathcal{Z}_\e=L^{2}(M_\e,g_\e)$.
\end{itemize}
\end{multicols}
\noindent In particular, note that $\mathcal{E}^{(\e)} (\cdot,\cdot)=(\cdot,\cdot)_{H^1(M_\e,g_\e)}$.

\subsection{Assumptions} \label{subsec_assum_conn_sum}

Let $N \in \mathbb{N}\setminus \{0\}$ and let  $\la_{0,N}$ be an eigenvalue of \eqref{eq_psi_connect_sum}. 

In this section, we verify that the assumptions of \Cref{thm:main} are satisfied and we recall what it implies.
Based on what has been seen so far, \eqref{hp_compact} is implied by \cite[Section 1]{colette-colbois} and \eqref{hp_stability} is implied by \cite[Theorem 1.1]{Ta02}.

Next, we define the map $L_\e:H^{1}(M_0,g_0) \to H^{1}(M_\e,g_\e)$ in \eqref{hp_Le_exist} in this setting as 
\begin{align}\label{def_Le_conn_sum}
L_\e :u & \mapsto (u|_{M_0(\e)},h_{\e,u}),
\end{align}
where $h_{\e,u}$ is the unique weak solution to 
\begin{align}
\begin{cases}\label{prob_harmonic_extension}
(-\Delta^{g_1}+\e^2)h_{\e,u}=0 & \textnormal{in } M_1(1) \\
h_{\e,u}= u\circ \Phi^{-1}_\e  & \textnormal{on }\partial M_1(1),
\end{cases}
\end{align}
that is, $h_{\e,u}-u\in H^1_0(M_1(1),g_1)$ and
\begin{equation}\label{eq_harmonic_extension}
\int_{M_1(1)} \left[g_1(\nabla^{g_1} h_{\e,u}, \nabla^{g_1}w) +  \e^2 h_{\e,u} w \right] \dvol^{g_1}=0 
\quad \text{ for any } w \in H^1_0(M_1(1),g_1).
\end{equation}
In particular, since the bottom of the Dirichlet-spectrum of $-\Delta^{g_1}+\e^2$ is positive, by the weak maximum principle we have that
\begin{align}\label{Eq_estimate_L^infty_h}
\norm{h_{\e,\varphi}}_{L^\infty(M_1(1))}\leq \norm{\varphi\circ\Phi_\e^{-1}}_{L^\infty(\partial M_1(1))}=\norm{\varphi}_{L^\infty(\partial M_0(\varepsilon))}\leq \norm{\varphi}_{L^\infty(M_0)}.
\end{align}

\begin{lemma}\label{lemma_Le_con_sum}
With $L_\e$ defined as in \eqref{def_Le_conn_sum}, assumption  \eqref{hp_Le} holds.
\end{lemma}

\begin{proof}
For any fixed $\psi,\varphi\in E(\lambda_{0,N})$ as $\e \to 0^+$
\begin{multline}
\left|(L_\e (\psi), L_\e (\varphi))_{L^2(M_\e,g_\e)}-(\psi, \varphi)_{L^2(M_0,g_0)}\right| \\
\leq \int_{B_\e^{M_0}(p_0)} |\psi \varphi| \dvol^{g_0} + \e^m \int_{M_1(1)} |h_{\e,\psi} h_{\e,\varphi}| \dvol^{g_1}\\
\leq \norm{\psi}_{L^\infty(M_0)} \norm{\varphi}_{L^\infty(M_0)} \left(\textnormal{vol}^{g_0}(B_\e^{M_0}(p_0))+\e^m \textnormal{vol}^{g_1}(M_1) \right),
\end{multline}
where we used 
\eqref{Eq_estimate_L^infty_h} to estimate $L^\infty$-norm of both $h_{\e,\psi}$ and $h_{\e,\varphi}$. Hence,  condition \eqref{hp_Le} is satisfied.
\end{proof}

Now we focus on \eqref{hp_delta_e}. 
For the rest of this section, let 
\begin{equation}
\varphi \in  E(\lambda_{0,N})
\end{equation}
be fixed. In particular, $\varphi$ is smooth. For the sake of simplicity, unless we need to remark the dependence on $\varphi$, we
denote 
\begin{equation}
V_\e:=V_{\e,L_\e (\varphi)},
\end{equation}
where $V_{\e,L_\e (\varphi)}\in H^1(M_\e,g_\e)$ is as in \Cref{prop_J_min}.
Since $V_\e \in H^1(M_\e,g_\e)$, we may write it as $V_\e=(V_{\e,0},V_{\e,1})\in H^1(M_0(\e),g_0)\times H^1(M_1(1),\e^2g_1)$ 
and hence \eqref{eq_Ve} reads as 
\begin{align}\label{eq_Ve_connected_sum}
    \int_{M_0(\e)} & \left[g_0 \left(\nabla^{g_0} V_{\e,0}, \nabla^{g_0} u_0\right)+V_{\e,0}u_0 \right] \dvol^{g_0} \\
    + & \e^{m-2} \int_{M_1(1)} \left[g_1 \left(\nabla^{g_1} V_{\e,1}, \nabla^{g_1} u_1\right)+\e^2 V_{\e,1}u_1 \right] \dvol^{g_1}\notag \\
    = & \int_{M_0(\e)} \left[g_0 \left(\nabla^{g_0} \varphi, \nabla^{g_0} u_0\right)+\varphi u_0 \right] \dvol^{g_0} + \e^{m-2} \int_{M_1(1)} \left[g_1 \left(\nabla^{g_1} h_{\e,\varphi}, \nabla^{g_1} u_1\right)+\e^2 h_{\e,\varphi} u_1 \right] \dvol^{g_1\notag}\\
    & - \lambda_{0,N} \int_{M_0(\e)} \varphi u_0 \dvol^{g_0} - \lambda_{0,N} \e^m \int_{M_1(1)} h_{\e,\varphi} u_1 \dvol^{g_1},\notag
\end{align}
for any $u=(u_0,u_1)\in H^{1}(M_\e,g_\e)$.

To show that \eqref{hp_delta_e} holds, we need some preliminaries.

\begin{lemma}\label{lemma_he_bounded}
The family $\{h_{\e,\varphi}\}_{\e \in (0,1)}$ is bounded in $H^1(M_1(1),g_1)$. In particular
\begin{align}    \norm{h_{\e,\varphi}}_{H^1(M_1(1),\e^2}g_1)^2=O(\e^{m-2}) \quad as\ \e \to 0^+.
\end{align}
\end{lemma}
\begin{proof}
    Let us denote
    \begin{align}
        \overline{g}_0:=g_0|_{\partial M_0(\e)} \quad \andd \quad \overline{g}_1:=g_1|_{\partial M_1(1)}.
    \end{align}
    We start by observing that
    \begin{align}\label{App-3_lemma_he_bounded}
        \norm{h_{\e,\varphi}}_{H^1(M_1(1),\e^2 g_1)}^2 &= \e^{m-2} \int_{M_1(1)} \left[|\nabla^{g_1}h_{\e,\varphi}|_{g_1}^2 + \e^2 h_{\e,\varphi}^2\right]\dvol^{g_1}
            \end{align}
    By the Lax-Milgram theorem for $h_{\e,\varphi}$, there exists a constant $C>0$, only depending on $(M_1(1), g_1)$, such that
    \begin{align}\label{App-2_lemma_he_bounded}
        \norm{h_{\e,\varphi}}_{H^1(M_1(1), \e^2g_1)}^2 &\leq C\e^{m-2}\norm{h_{\e,\varphi}}_{H^{\frac{1}{2}}(\partial M_1(1), \overline{g}_1)}^2\\
        &=C\e^{m-2} \left[\int_{\partial M_1(1)} h_{\e,\varphi}^2 \dint{\sigma}^{\overline{g}_1} + \int_{\partial M_1(1)\times \partial M_1(1)} \frac{|h_{\e,\varphi}(x)-h_{\e,\varphi}(y)|^2}{(d_{\partial M_1(1)}^{\overline{g}^1}(x,y))^m} \dint{\sigma}_x^{\overline{g}_1} \dint{\sigma}_y^{\overline{g}_1}\right],
    \end{align}
    where $d_{\partial M_1(1)}^{\overline{g}^1}$ denotes the distance function on $\partial M_1(1)$ with respect to the induced Riemannian metric $\overline{g}_1$. Noticing that
    \begin{align}
        d^{\e^2 \overline{g}_1}_{\partial M_1(1)}=\e d^{\overline{g}_1}_{\partial M_1(1)} \quad \andd \quad \dint{\sigma}_x^{\e^2 \overline{g}_1}=\e^{m-1} \dint{\sigma}_x^{\overline{g}_1},
    \end{align}
    by 
    \eqref{App-2_lemma_he_bounded} we get
    \begin{align}
        &\norm{h_{\e,\varphi}}_{H^1(M_1(1), \e^2 g_1)}^2 \\
            &\leq C \left[\e^{-1} \int_{\partial M_1(1)} h_{\e,\varphi}^2 \dint{\sigma}^{\e^2 \overline{g}_1} + \int_{\partial M_1(1)\times \partial M_1(1)} \frac{|h_{\e,\varphi}(x)-h_{\e,\varphi}(y)|^2}{(d_{\partial M_1(1)}^{\e^2 \overline{g}^1}(x,y))^m} \dint{\sigma}_x^{\e^2 \overline{g}_1} \dint{\sigma}_y^{\e^2 \overline{g}_1}\right]\\
        &\leq C \e^{-1} \left[\int_{\partial M_1(1)} h_{\e,\varphi}^2 \dint{\sigma}^{\e^2 \overline{g}_1} + \int_{\partial M_1(1)\times \partial M_1(1)} \frac{|h_{\e,\varphi}(x)-h_{\e,\varphi}(y)|^2}{(d_{\partial M_1(1)}^{\e^2 \overline{g}^1}(x,y))^m} \dint{\sigma}_x^{\e^2 \overline{g}_1} \dint{\sigma}_y^{\e^2 \overline{g}_1}\right]\\
        &=C \e^{-1} \norm{h_{\e,\varphi}}_{H^\frac{1}{2}(\partial M_1(1),\e^2 \overline{g}_1)}^
        2.\label{App-1_lemma_he_bounded}
    \end{align}
    By \eqref{App-1_lemma_he_bounded}, we get
    \begin{align}
        \norm{h_{\e,\varphi}}_{H^1(M_1(1),\e^2 g_1)}^2 &\leq C \e^{-1}\norm{h_{\e,\varphi}}_{H^{\frac{1}{2}}(\partial M_1(1),\e^2 \overline{g}_1)}^2\\
        &=C \e^{-1}\norm{\varphi\circ \Phi_\e^{-1}}_{H^{\frac{1}{2}}(\partial M_1(1),\e^2 \overline{g}_1)}^2\\
        &\leq C \e^{-1}\norm{\varphi\circ \Phi_\e^{-1}}_{H^{1}(\partial M_1(1),\e^2 \overline{g}_1)}^2\\
        &= C \e^{-1}\int_{\partial M_1(1)} \left[|\nabla^{\e^2 \overline{g}_1}(\varphi\circ \Phi_{\e}^{-1})|_{\e^2 \overline{g}_1}^2+(\varphi\circ \Phi_{\e}^{-1})^2\right] \dvol^{\e^2 \overline{g}_1}\\
        &\leq C \e^{m-2} \textnormal{vol}^{\overline{g}_1}(\partial M_1(1)) \norm{\varphi\circ \Phi_\e^{-1}}_{C^{1}(\partial M_1(1),\e^2 \overline{g}_1)}^2.\label{App0_lemma_he_bounded}
    \end{align}
    Recalling that $\Phi_\e^{-1}$ is an isometry, and hence a diffeomorphism such that for every $q\in \partial M_1(1)$ and for every $v,w\in T_q \partial M_1(1)$
    \begin{align}
        (\overline{g}_0)_{\Phi_\e^{-1}(q)}(\dint{\Phi_\e^{-1}}v, \dint{\Phi_\e^{-1}}w)=(\e^2 \overline{g}_1)_q(v,w),
    \end{align}
    it follows that
    \begin{itemize}
        \item by the fact that $\Phi_\e^{-1}$ is a diffeomorphism,
        \begin{align}
            \max_{\partial M_1(1)} |\varphi \circ \Phi_\e^{-1}|=\max_{\partial M_0(\e)} |\varphi|;
        \end{align}
        \item by the fact that $\Phi_\e^{-1}$ preserves the Riemannian metric, for every $q\in \partial M_1(1)$ and for every $v\in T_p \partial M_1(1)$
        \begin{align}
            (\e^2 \overline{g}_1)_q ((\nabla^{\e^2 \overline{g}_1} (\varphi\circ \Phi_\e^{-1}))_q,v)&= \dint{_q (\varphi\circ \Phi_\e^{-1})}[v]\\
            &=\dint_{\Phi_\e^{-1}(q)}{\varphi}[\dint{_q(\Phi_\e^{-1})}[v]]\\
            &=\dint_{\Phi_\e^{-1}(q)}{\varphi}[(\dint{_{\Phi_\e^{-1}(q)}\Phi_\e})^{-1}[v]]\\
            &=(\overline{g}_0)_{\Phi_\e^{-1}(q)}(\nabla^{\overline{g}_0}\varphi,(\dint{_{\Phi_\e^{-1}(q)}\Phi_\e})^{-1}[v])\\
            &=(\e^2 \overline{g}_1)_q (\dint{_{\Phi_\e^{-1}(q)}\Phi_\e}[\nabla^{\overline{g}_{0}} \varphi],v)
        \end{align}
        i.e.
        \begin{align}
            (\nabla^{\e^2 \overline{g}_1} (\varphi\circ \Phi_\e^{-1}))_q=\dint{_{\Phi_\e^{-1}(q)}\Phi_\e}[\nabla^{\overline{g}_{0}} \varphi]
        \end{align}
        and hence, for every $q\in \partial M_1(1)$,
        \begin{align}
            |\nabla^{\e^2 \overline{g}_1}(\varphi\circ \Phi_\e^{-1})|_{\e^2\overline{g}_1}^2(q)&=(\e^2 \overline{g}_1)_q ((\nabla^{\e^2 \overline{g}_1}(\varphi\circ \Phi_\e^{-1}))_q,(\nabla^{\e^2 \overline{g}_1}(\varphi\circ \Phi_\e^{-1}))_q)\\
            &=(\e^2 \overline{g}_1)_q (\dint{_{\Phi_\e^{-1}(q)}\Phi_\e}[\nabla^{\overline{g}_{0}} \varphi],\dint{_{\Phi_\e^{-1}(q)}\Phi_\e}[\nabla^{\overline{g}_{0}} \varphi])\\
            &=(\overline{g}_0)_{\Phi_\e^{-1}(q)}((\nabla^{\overline{g}_{0}} \varphi)_{\Phi_\e^{-1}(q)},(\nabla^{\overline{g}_{0}} \varphi)_{\Phi_\e^{-1}(q)})\\
            &=|\nabla^{\overline{g}_0}\varphi|_{\overline{g}_0}(\Phi_\e^{-1}(q))
        \end{align}
        providing
        \begin{align}
            \max_{\partial M_1(1)} |\nabla^{\e^2 \overline{g}_1}(\varphi\circ \Phi_\e^{-1})|_{\e^2 \overline{g}_1}=\max_{\partial M_0(\e)} |\nabla^{\overline{g}_0}\varphi|_{\overline{g}_0}.
        \end{align}
    \end{itemize}
    Hence
    \begin{align}\label{App1_lemma_he_bounded}
        \norm{\varphi\circ \Phi_\e^{-1}}_{C^{1}(\partial M_1(1),\e^2 \overline{g}_1)}&=\max_{\partial M_1(1)} |\varphi \circ \Phi_\e^{-1}| + \max_{\partial M_1(1)} |\nabla^{\e^2 \overline{g}_1}(\varphi\circ \Phi_\e^{-1})|_{\e^2 \overline{g}_1}\\
        &=\max_{\partial M_0(\e)} |\varphi|+\max_{\partial M_0(\e)} |\nabla^{\overline{g}_0}\varphi|_{\overline{g}_0}\\
        &=\norm{\varphi}_{C^{1}(\partial M_0(\e),\overline{g}_0)}.
    \end{align}
    By putting together \eqref{App0_lemma_he_bounded} and \eqref{App1_lemma_he_bounded}, we get
    \begin{align}
        \norm{h_{\e,\varphi}}_{H^1(M_1(1),\e^2 g_1)}^2&\leq C \e^{m-2} \textnormal{vol}^{\overline{g}_1}(\partial M_1(1))\norm{\varphi\circ \Phi_\e^{-1}}_{C^{1}(\partial M_1(1),\e^2 \overline{g}_1)}^2 \\
        &=C \e^{m-2} \textnormal{vol}^{\overline{g}_1}(\partial M_1(1)) \norm{\varphi}_{C^{1}(\partial M_0(\e),\overline{g}_0)}^2.
    \end{align}
    Hence, since $\varphi$ is smooth, we conclude the proof.
\end{proof}

We now introduce a uniform extension operator from $H^1(M_0(\e),g_0)$ to $H^1(M_0,g_0)$.

\begin{lemma}\label{lemma_uniform_ext}
There exists a family of equibounded extension operators
\begin{equation}
E_\e: H^1(M_0(\e),g_0) \to H^1(M_0,g_0),
\end{equation}
that is,  there exists a constant $C>0$, that does not depend on $\e$, such that for any $\e\in(0,1)$
\begin{equation}\label{ineq_Ee}
\norm{E_\e w}_{H^1(M_0,g_0)}\le C\norm{w}_{H^1(M_0(\e),g_0)} \quad  \text{ for any } w \in H^1(M_0(\e),g_0)
\end{equation}
and $(E_\e w)_{|M_0(\e)}=w$.
\end{lemma}
\begin{proof}
By   \cite{RT_extension}, (see also \cite[Proposition 3.1]{FRS_AB}), there exists a family of equibounded extension operators 
\begin{equation}
\widetilde{E}_\e: H^1(B_1^{\R^m}(0)\setminus B_\e^{\R^m}(0)) \to H^1(B_1^{\R^m}(0)).
\end{equation}
Then we may define 
\begin{align}
E_\e w :=\begin{cases}
w, & \inn M_0 \setminus B_1^{M_0}(p_0),\\
\widetilde{E}_\e(w\circ (\exp_{p_0}^{M_0})^{-1})\circ (\exp_{p_0}^{M_0})^{-1}, & \inn B_1^{M_0}(p_0).
\end{cases}
\end{align}
Since the metric $g_0$ is smooth,  \eqref{ineq_Ee} holds. 
\end{proof}

We recall the following definition from \eqref{def_delta_e}
\begin{equation}\label{eq_delta_manifold}
    \delta_{\e,N}:=\sup\left\{\norm{V_{\e,L_\e(\varphi)}}_{L^2(M_\e,g_\e)}\colon\varphi\in E(\lambda_{0,N}),~\norm{\varphi}_{L^2(M_0,g_0)}=1\right\}.
\end{equation}
We are now in a position to show that \eqref{hp_delta_e} is satisfied.

\begin{proposition}
We have that
\begin{equation}\label{limit_Ve_connected_sum_0}
\lim_{\e\to 0^+} \delta_{\e,N} =0.
\end{equation}
\end{proposition}
\begin{proof}
Let $\varphi\in \{\varphi_{N+i-1}\}_{i=1,\dots,M_{0,N}}$. 
We have that  \eqref{eq_Ve_connected_sum} tested with $V_\e$, using \eqref{eq_psi_connect_sum}, \Cref{lemma_uniform_ext} and  the Cauchy-Schwarz inequality yield
\begin{align}
\norm{V_\e}^2_{H^1(M_\e,g_\e)}
=& (\varphi, V_{\e,0})_{H^1(M_0(\e),g_0)}+(h_{\e,\varphi}, V_{\e,1})_{H^1(M_1(1),\e^2 g_1)}\\
&- \la_{0,N}(\varphi, V_{\e,0})_{L^2(M_0(\e),g_0)} - \la_{0,N}(h_{\e,\varphi}, V_{\e,1})_{L^2(M_1(1),\e^2 g_1)}\\
=& - (\varphi, E_\e V_{\e,0})_{H^1(B_\e^{M_0}(p_0),g_0)}+(h_{\e,\varphi}, V_{\e,1})_{H^1(M_1(1),\e^2 g_1)}\\
& + \lambda_{0,N} (\varphi, E_\e V_{\e,0})_{L^2(B_\e^{M_0}(p_0),g_0)}  - \la_{0,N}(h_{\e,\varphi}, V_{\e,1})_{L^2(M_1(1),\e^2 g_1)}\\
\leq & \left(C \norm{\varphi}_{H^1(B_\e^{M_0}(p_0),g_0)} + \norm{h_{\e,\varphi}}_{H^1(M_1(1),\e^2 g_1)} \right.\\
&\left. + \lambda_{0,N} C \norm{\varphi}_{L^2(B_\e^{M_0}(p_0),g_0)} + \norm{h_{\e,\varphi}}_{L^2(M_1(1),\e^2 g_1)} \right) \norm{V_\e}_{H^1(M_\e,g_\e)},
\end{align}
by \Cref{lemma_uniform_ext}. Hence,  we have proved for any $\varphi \in E(\la_{0,N})$
\begin{multline}
\norm{V_\e}_{H^1(M_\e,g_\e)} \le \left(C \norm{\varphi}_{H^1(B_\e^{M_0}(p_0),g_0)} + \norm{h_{\e,\varphi}}_{H^1(M_1(1),\e^2 g_1)} \right.\\
\left. + \lambda_{0,N} C \norm{\varphi}_{L^2(B_\e^{M_0}(p_0),g_0)} + \norm{h_{\e,\varphi}}_{L^2(M_1(1),\e^2 g_1)} \right)
\end{multline}
Since $\textnormal{vol}^{g_0}(B_\e^{M_0}(p_0))\to 0$ as $\e \to 0^+$ and taking into account \Cref{lemma_he_bounded}, we have proved 
\eqref{limit_Ve_connected_sum_0} in view of \eqref{rem_delta_e}.
\end{proof}

In this setting, the bilinear symmetric form $h_{\e,N}$ defined in \eqref{def_h_e} is given by 
\begin{multline}\label{def_he_con_sum}
 h_{\e,N}(\varphi,\psi)=\lambda_{0,N}(V_{\e,L_\e(\varphi)},L_\e(\psi))_{H_\e}\\
 =\lambda_{0,N}\int_{M_0(\e)} V_{\e,L_\e(\varphi),0}\psi \dvol^{g_0}
 +\lambda_{0,N}\e^m \int_{M_1(1)} V_{\e,L_\e(\varphi),1}h_{\e,\psi}\dvol^{g_1}.
\end{multline}
 and we denote as usual with $\{\gamma_{\e,i}\}_{i=1,\dots,M_{0,N}}$ its eigenvalues on $E(\lambda_{0,N})$ counted with multiplicity. We also recall the following
 \begin{equation}\label{eq_tau_manifold}
     \tau_{\e,N}^2=\max_{i=1,\dots,M_{0,N}}|\gamma_{\e,i}|.
 \end{equation}
Then, by \Cref{thm:main}, we have proved the following result.
\begin{theorem}\label{theor_eigen_connected_sum}
Let $\lambda_{0,N}$ be an eigenvalue of \eqref{eq_psi_connect_sum} of multiplicity $M_{0,N}$ such that
\begin{equation}
    \lambda_{0,N-1}<\lambda_{0,N}=\cdots=\lambda_{0,N+M_{0,N}-1}<\lambda_{0,N+M_{0,N}}
\end{equation}
and let $\{\lambda_{\e,N+i-1}\}_{i=1,\dots,M_{0,N}}$ be the eigenvalues of $-\Delta_\e+1$ converging to $\lambda_{0,N}$. Then we have that
\begin{equation}\label{eq_eigen_first_order__connected_sum}
\lambda_{\e,N+i-1}-\lambda_{0,N}=\gamma_{\e,i}+O(\delta_{\e,N}^2)+o(\tau_{\e,N}^2) \quad \text{ as } \e \to 0^+
\end{equation} 
for any $i=1, \dots, M_{0,N}$, where $\{\gamma_{\e,i}\}_{i=1,\dots,M_{0,N}}$ are the eigenvalues of \eqref{def_he_con_sum} in increasing order and $\delta_{\e,N}$ and $\tau_{\e,N}$ are as in \eqref{eq_delta_manifold} and \eqref{eq_tau_manifold}, respectively.
\end{theorem}

\subsection{The asymptotics of the eigenvalue variation}
Now we turn to study the asymptotic behavior of the bilinear form $h_{\e,N}(\varphi,\psi)$
as $\e \to 0^+$ and consequently of the eigenvalues $\{\gamma_{\e,i}\}_{i=1,\dots,M_{0,N}}$ which characterize the eigenvalue variation $\lambda_{\e,N+i-1}-\lambda_{0,N}$, for $i=1,\dots,M_{0,N}$, in view of \Cref{theor_eigen_connected_sum}.

Before proceeding, we recall from \eqref{def_qe} and \eqref{def_Je} the following notation
\begin{align}
    &q_\e(u,v)=Q_\e(u,v)-\lambda_{0,N}(u,v)_{L^2(M_\e,g_\e)}, \label{eq_q_e_manifolds}\\
    &J_{\e,\varphi}(u)=\frac{1}{2}\norm{u}_{H^1(M_\e,g_\e)}^2-q_\e(L_\e(\varphi),u),\label{eq_J_e_manifolds}
\end{align}
for $u,v\in H^1(M_\e,g_\e)$ and $\varphi\in E(\lambda_{0,N})$
Again, we may drop the dependence on $\varphi$ and just denote $J_{\e,\varphi}$ with $J_\e$ when there are no ambiguities.

First of all, we prove that the $L^2(M_0(\e),g_0)$ norm of $V_{\e,0}$ is a remainder term with respect to the $H^1$ norm of $V_\e$.
\begin{proposition}\label{prop_manifolds_L2norm_smallo}
As $\e \to 0^+$,
\begin{equation}\label{eq_manifolds_L2norm_smallo_M0}
\norm{V_{\e,0}}_{L^2(M_0(\e),g_0)}=o(\norm{V_{\e,0}}_{H^1(M_0(\e),g_0)})
\end{equation}
which trivially implies
\begin{equation}\label{eq_manifolds_L2norm_smallo}
\norm{V_{\e,0}}_{L^2(M_0(\e),g_0)}=o(\norm{V_{\e}}_{H^1(M_\e,g_\e)}).
\end{equation}
\end{proposition}
\begin{proof}
Let us prove \eqref{eq_manifolds_L2norm_smallo_M0} by contradiction.  Assume that for some constant $C>0$ and for some $\e_n \to 0^+$
\begin{equation}
\frac{\norm{V_{\e_n,0}}_{H^1(M_0(\e_n),g_0)}^2}{\norm{V_{\e_n,0}}_{L^2(M_0(\e_n),g_0)}^2} \leq C.
\end{equation}
Let us introduce for any $n \in \mathbb{N}\setminus \{0\}$
\begin{equation}
W_n:=\frac{E_{\e_n}V_{\e_n,0}}{\norm{E_{\e_n}V_{\e_n,0}}_{L^2(M_0,g_0)}},
\end{equation}
where $E_{\e_n}$ is the family of uniformly bounded  extension operators introduced in \Cref{lemma_uniform_ext}.
In particular, we have that 
\begin{equation}
\norm{W_n}_{L^2(M_0,g_0)}=1 \quad \text{ and }\quad \norm{W_n}_{H^1(M_0,g_0)}\leq C\quad\textnormal{for all }n\in\N.
\end{equation}
It follows that, up to passing to a subsequence, there exists $W_0 \in H^1(M_0,g_0)$ with $\norm{W_0}_{L^2(M_0,g_0)}=1$ such that 
$W_n \rightharpoonup W_0$ weakly in $H^1(M_0,g_0)$ as $n \to \infty$. Let $(u_0,0) \in H^1(M_\e,g_\e)$ with 
$u_0 \in C^\infty_c(M_0 \setminus\{p_0\})$. We may test \eqref{eq_Ve_connected_sum} with $(u_0,0)$ thus, for $n$ large enough,  
\begin{align}
& \int_{M_0} \left[g_0 \left(\nabla^{g_0} W_n, \nabla^{g_0} u_0\right)+W_nu_0 \right] \dvol^{g_0} =\int_{M_0(\e_n)}  \left[g_0 \left(\nabla^{g_0} W_n, \nabla^{g_0} u_0\right)+W_nu_0 \right] \dvol^{g_0} \\
= &\norm{E_{\e_n}V_{\e_n,0}}_{L^2(M_0,g_0)}^{-1} \left(
\int_{M_0(\e_n)} \left[g_0 \left(\nabla^{g_0} \varphi, \nabla^{g_0} u_0\right)+\varphi u_0 \right] \dvol^{g_0} 
- \lambda_{0,N} \int_{M_0(\e_n)} \varphi u_0 \dvol^{g_0} \right)=0
\end{align}
by \eqref{eq_psi_connect_sum}. Passing to the limit as $n \to +\infty$, we obtain
\begin{equation}
\int_{M_0}\left[g_0 \left(\nabla^{g_0} W_0, \nabla^{g_0} u_0\right)+W_0u_0 \right] \dvol^{g_0}=0  
\quad \text{ for any } u_0 \in C^\infty_c(M_0 \setminus\{p_0\}).
\end{equation}
Since the Sobolev capacity of a point is $0$, it follows that 
\begin{equation}
\int_{M_0}\left[g_0 \left(\nabla^{g_0} W_0, \nabla^{g_0} u_0\right)+W_0u_0 \right] \dvol^{g_0}=0  
\quad \text{ for any } u_0 \in H^1(M_0,g_0).
\end{equation}
Hence, $W_0=0$ a contradiction with the fact that $\norm{W_0}_{L^2(M_0,g_0)}=1$. The proof is then complete.
\end{proof}

In the next result, we provide a first bound for the $H^1$ norm of $V_\e$.
\begin{proposition}\label{Proposition_norm_J}
There exists a constant $K>0$, that does not depend on $\e$, such that 
\begin{multline}\label{ineq_Ve_precise_connected_sum}
\norm{V_\e}^2_{H^1(M_\e ,g_\e)} \le K\left| \int_{B_\e^{M_0}(p_0)}\left[g_0 \left(\nabla^{g_0} \varphi, \nabla^{g_0} \varphi\right)
+|\varphi-\varphi(p_0)|^2 \right]\dvol^{g_0}\right| \\
+  K\e^{m-2}\left| \int_{M_1(1)} \left[g_1 \left(\nabla^{g_1} h_{\e,\varphi},\nabla^{g_1} h_{\e,\varphi}\right)+\e^2 |h_{\e,\varphi}-\varphi(p_0)|^2\right] \dvol^{g_1} \right|+K\e^m.
\end{multline}
\end{proposition}

\begin{proof}
Testing \eqref{eq_Ve} with $V_\e$, by definition of $J_\e$, see \eqref{eq_J_e_manifolds},
\begin{equation}\label{eq_culo2}
\norm{V_\e}^2_{H^1(M_\e ,g_\e)}=-2J_\e(V_\e).
\end{equation}
Furthermore, with a standard proof (see for example \cite[Proposition 3.5]{BS_eigen} or \cite[Lemma 3.1]{FLO_neumann}) we can show that 
\begin{equation}\label{eq_culo3}
-2J_\e(V_\e)=\sup_{0\neq u \in H^1(M_\e ,g_\e)}\frac{|q_\e(L_\e (\varphi),u)|^2}{\norm{u}^2_{H^1(M_\e ,g_\e)}},
\end{equation}
where $q_\e$ is as in \eqref{eq_q_e_manifolds}. Making it explicit, for any $(u_0,u_1) \in H^{1}(M_\e,g_\e)$, denoting 
\begin{align}
    \overline{u}:=\left(\fint_{M_0(\e)} u_0 \dvol^{g_0},\fint_{M_1(1)} u_1\dvol^{g_1}\right) \quad \andd \quad \dot{u}=(\dot{u}_0,\dot{u}_1):=(u_0,u_1)-\overline{u},
\end{align}
so to have $u=\dot{u}+\overline{u}$ and
\begin{equation}
    \int_{M_0(\e)}\dot{u}_0\dvol^{g_0}=\int_{M_1(1)}\dot{u}_1\dvol^{g_1}=0,
\end{equation}
we have
\begin{align}
q_\e(L_\e(\varphi),u)=&q_\e(L_\e(\varphi),\dot u)+q_\e(L_\e(\varphi),\overline{u})\\
=&\int_{M_0(\e)} \left[g_0 \left(\nabla^{g_0} \varphi, \nabla^{g_0} \dot{u}_0\right)+\varphi \dot{u}_0 \right] \dvol^{g_0} \\
& + \e^{m-2} \int_{M_1(1)} \left[g_1 \left(\nabla^{g_1} h_{\e,\varphi},\nabla^{g_1} \dot{u}_1\right)+\e^2 h_{\e,\varphi} \dot{u}_1 \right] \dvol^{g_1} \\
& - \lambda_{0,N} \int_{M_0(\e)} \varphi \dot{u}_0 \dvol^{g_0} - \lambda_{0,N} \e^m \int_{M_1(1)} h_{\e,\varphi} \dot{u}_1 \dvol^{g_1}\\
&+ q_\e(L_\e(\varphi),\overline{u})\\
=&\int_{M_0(\e)} \left[g_0 \left(\nabla^{g_0} \varphi, \nabla^{g_0} \dot{u}_0\right)+(\varphi-\varphi(p_0)) \dot{u}_0 \right] \dvol^{g_0} \\
& + \e^{m-2} \int_{M_1(1)} \left[g_1 \left(\nabla^{g_1} h_{\e,\varphi},\nabla^{g_1} \dot{u}_1\right)+\e^2 (h_{\e,\varphi}-\varphi(p_0))\dot{u}_1 \right] \dvol^{g_1} \\
& - \lambda_{0,N} \int_{M_0(\e)} (\varphi-\varphi(p_0)) \dot{u}_0 \dvol^{g_0}  - \lambda_{0,N} \e^m \int_{M_1(1)} (h_{\e,\varphi}-\varphi(p_0)) \dot{u}_1 \dvol^{g_1}\\
&+ q_\e(L_\e(\varphi),\overline{u}).\label{eq_culo4}
\end{align}
Furthermore, by the equation of $\varphi$ as in \eqref{eq_psi_connect_sum} we have
\begin{align}
    &\int_{M_0(\e)} \left[g_0 \left(\nabla^{g_0} \varphi, \nabla^{g_0} \dot{u}_0\right)+(\varphi-\varphi(p_0)) \dot{u}_0 \right] \dvol^{g_0} 
- \lambda_{0,N} \int_{M_0(\e)} (\varphi-\varphi(p_0)) \dot{u}_0 \dvol^{g_0} \label{eq_culo}\\
    &=-\int_{B_\e^{M_0}(p_0)}\big[g_0(\nabla^{g_0}\varphi,\nabla^{g_0}E_\e \dot{u}_0)+(\varphi-\varphi(p_0))E_\e \dot{u}_0\big]\dvol^{g_0}+\int_{B_\e^{M_0}(p_0)}(\lambda_{0,N}\varphi-\varphi(p_0))E_\e \dot{u}_0\dvol^{g_0}.
\end{align}
Then, by Cauchy-Schwarz inequality \Cref{lemma_uniform_ext} and boundedness of $\varphi$ we deduce that
\begin{equation}
    \left|\int_{B_\e^{M_0}(p_0)}\big[g_0(\nabla^{g_0}\varphi,\nabla^{g_0}E_\e \dot{u}_0)+(\varphi-\varphi(p_0))E_\e \dot{u}_0\big]\dvol^{g_0}\right|\leq C\norm{\varphi-\varphi(p_0)}_{H^1(B_\e^{M_0}(p_0),g_0)}\norm{\dot{u}_0}_{H^1(M_0(\e),g_0)}
\end{equation}
and that 
\begin{equation}
    \left|\int_{B_\e^{M_0}(p_0)}(\lambda_{0,N}\varphi-\varphi(p_0))E_\e \dot{u}_0\dvol^{g_0}\right|\leq C\e^{m/2}\norm{\dot{u}_0}_{H^1(M_0(\e),g_0)}.
\end{equation}
Plugging these two estimates into \eqref{eq_culo} we obtain
\begin{align}
&\left|\int_{M_0(\e)} \left[g_0 \left(\nabla^{g_0} \varphi, \nabla^{g_0} \dot{u}_0\right)+(\varphi-\varphi(p_0)) \dot{u}_0 \right] \dvol^{g_0} 
- \lambda_{0,N} \int_{M_0(\e)} (\varphi-\varphi(p_0)) \dot{u}_0 \dvol^{g_0}\right| \\ 
\le& C\norm{\dot{u}_0}_{H^1(B_\e^{M_0}(p_0) ,g_0)} \left( \norm{\varphi-\varphi(p_0)}_{H^1(M_0(\e),g_0)}+\e^{m/2}\right),\label{eq_psi_boundary_B}
\end{align}

for some positive constant $C>0$ that does not depend on $u$ and $\e$. 
Moreover, since
\begin{align}
    \int_{M_0} \varphi\ \dvol^{g_0}=0,
\end{align}
we have 
\begin{align}
    q_\e(L_\e(\varphi),\overline{u})&=-(1-\lambda_{0,N})\overline{u}_0\int_{B_{\e}^{M_0}(p_0)} \varphi\ \dvol^{g_0} +(1-\lambda_{0,N})\overline{u}_1 \e^m \int_{M_1(1)} h_{\e,\varphi}\ \dvol^{g_1} \\
    &=(|\overline{u}_0|+|\overline{u}_1|)O(\e^m)\label{eq_culo5}
\end{align}
as $\e\to 0$.

By the Cauchy-Schwarz inequality, we have
\begin{equation*}
    \bar{u}_0^2\leq C\norm{u_0}_{H^1(M_0(\e),g_0)}^2\quad\textnormal{and}\quad \e^m\bar{u}_1^2\leq C\norm{u_1}_{H^1(M_1(1),\e^2g_1)}^2
\end{equation*}
for some $C>0$ independent from $u$ and $\e$.
Combining these estimates with \eqref{eq_culo5} we obtain
\begin{equation}\label{eq_culo6}
    |q_\e(L_\e(\varphi),\overline{u})|^2\leq C\e^m\norm{u}_{H^1(M_\e,g_\e)}^2.
\end{equation}
By combining \eqref{eq_culo2}, \eqref{eq_culo3}, \eqref{eq_culo4}, \eqref{eq_culo6}, Cauchy-Schwarz inequality and \eqref{eq_psi_boundary_B}, it follows that there exists a constant $K>0$, that does not depend on $\e$, such that 
\begin{multline}
\norm{V_\e}^2_{H^1(M_\e ,g_\e)} \le K\left| \int_{B_\e^{M_0}(p_0)}\left[g_0 \left(\nabla^{g_0} \varphi, \nabla^{g_0} \varphi\right)
+|\varphi-\varphi(p_0)|^2 \right]\dvol^{g_0}\right| \\
+  K\e^{m-2}\left| \int_{M_1(1)} \left[g_1 \left(\nabla^{g_1} h_{\e,\varphi},\nabla^{g_1} h_{\e,\varphi}\right)+\e^2 |h_{\e,\varphi}-\varphi(p_0)|^2\right] \dvol^{g_1} \right|+K\e^m,
\end{multline}
for $\e$ sufficiently small, that is, we have proved \eqref{ineq_Ve_precise_connected_sum}.
\end{proof}

Let us define $\Theta_\e(x):=\exp_{p_0}^{M_0}(\e x)$ and
\begin{align}
&\widehat \varphi_\e(x):=\frac{1}{\e} \Big(\varphi(\Theta_\e(x))-\varphi(p_0)\Big), \quad  &&\text{ and } \quad \quad \widehat V_{\e,0}(x):=\frac{1}{\e} V_{\e,0}(\Theta_\e(x)),\\
&\widehat h_{\e,\varphi}:=\frac{1}{\e}(h_{\e,\varphi}-\varphi(p_0)), \quad  &&\text{ and } \quad \quad \widehat V_{\e,1}:=\frac{1}{\e} V_{\e,1},
\end{align}
for any $\e \in (0,1)$. We stress that on $\partial \mathbb{B}_1\subset T_{p_1}M_1$
\begin{align}
    \widehat{V}_{\e,0}(F_\e(x))=\widehat{V}_{\e,1}(\exp_{p_1}^{M_1}(x)).
\end{align}
It is not hard to see that $\widehat h_{\e,\varphi}$ is the unique solution to 
\begin{align}
\begin{cases}\label{prob_harmonic_extension_new}
(-\Delta^{g_1}+\e^2)\widehat h_{\e,\varphi}=-\e\varphi(p_0), & \textnormal{in } M_1(1), \\
\widehat h_{\e,\varphi}= \frac{1}{\e}\left(\varphi \circ \Phi_\e^{-1}-\varphi(p_0)\right),  & \textnormal{on }\partial M_1(1),
\end{cases}
\end{align}
that is,  for any $w \in H^1_0( M_1(1),g_1)$,
\begin{equation}\label{eq_harmonic_extension_new}
\int_{M_1(1)} \left[g_1(\nabla^{g_1}  \widehat h_{\e,\varphi}, \nabla^{g_1}w) +  \e^2 \widehat h_{\e,\varphi} w+ \e\varphi(p_0)w \right] \dvol^{g_1}=0.
\end{equation}
Let us also denote with $\mathbb{B}_r$ the euclidean ball of center $0$ and radius $r>0$, identified with the unit ball in the tangent space $T_{p_0}M_0$, and observe that for $p\in B_1^{M_0}(p_0)$ and $x=(\exp_{p_0}^{M_0})^{-1}(q)$
\begin{align}
    (\varphi \circ \Theta_\e)(x) &=\varphi(p_0) + \e \left[\frac{\dint{}}{\dint{\e}}\Big|_{\e=0} (\varphi \circ \Theta_\e) \right](x) + o(\e)\\
    &=\varphi(p_0)+\e [\dint_{p_0}\varphi]([\dint_{0} \exp_{p_0}^{M_0}](x))+o(\e)\\
    &=\varphi(p_0)+\e [\dint_{p_0}\varphi](x)+o(\e),
\end{align}
as $\e\to 0$. We denote
\begin{equation}\label{def_P}
P(x):=\dint_{p_0}\varphi(x)
\end{equation}
so that
\begin{equation}\label{limit_psi_0}
\widehat \varphi_\e(x):=\frac{1}{\e} \Big(\varphi(\Theta_\e(x))-\varphi(p_0)\Big)\to P(x) \quad \text{ strongly in } C^1(\mathbb{B}_R) \text{ as } \e \to 0^+,
\end{equation}
for every $R>1$. Equivalently
\begin{align}\label{limit_psi_0_ver2}
    \widehat{\varphi}_\e(F_\e(x))\to P_\Psi(x):=(P\circ \Psi)(x) \quad \text{ strongly in } C^1(\mathbb{B}_R) \text{ as } \e \to 0^+,
\end{align}
for every $R>1$.

In the next lemma, we compute the asymptotic behavior of $\widehat h_{\e,\varphi}$ as $\e \to 0^+$.

\begin{lemma}\label{the_vengence_time_consuming_lemma}
We have that 
\begin{equation}
\widehat h_{\e,\varphi} \to \widehat h_{0,\varphi} \quad \text{ strongly in } H^1(M_1(1), g_1) \text{ as } \e \to 0^+,
\end{equation} 
where $ \widehat h_{0,\varphi}$ is the unique solution to
\begin{equation}\label{eq_widehat_h0}
\int_{M_1(1)} g_1(\nabla^{g_1}  \widehat h_{0,\varphi}, \nabla^{g_1}w) \dvol^{g_1}=0
\end{equation}
for any $w \in H^1_0(M_1(1),g_1)$  with $\widehat h_{0,\varphi}=P\circ \Psi\circ (\exp_{p_1}^{M_1})^{-1}$ on $\partial M_1(1)$. In particular,
\begin{equation}\label{limit_he}
h_{\e,\varphi} \to \varphi(p_0) \quad \text{ strongly in } H^1(M_1(1),g_1) \text{ as  } \e \to 0^+.
\end{equation}
\end{lemma}
\begin{proof}
The boundary datum in \eqref{prob_harmonic_extension_new} can be expressed, for any $p=\exp_{p_1}^{M_1}(x) \in \partial B_1^{M_1}(p_1)$, as
\begin{align}
    \widehat{h}_{\e,\varphi}(\exp_{p_1}^{M_1}(x))=\widehat{\varphi}_\e(F_\e(x)),
\end{align}
which can be expanded as
\begin{equation}\label{eq_conv_P_psi}
\widehat{\varphi}_\e(F_\e(x))= P_\Psi (x)+o(1)=P(\Psi(x))+o(1)\quad\text{as }\e\to 0^+,~\text{strongly in }C^1(\mathbb{B}_1),
\end{equation}
by \eqref{limit_psi_0}.
Arguing as in \Cref{lemma_he_bounded}, it follows that  the family $\{\widehat h_{\e,\varphi}\}_{\e \in (0,1)}$ is bounded in $H^1(M_1(1), g_1)$. Hence,  there exist a sequence $\e_j \to 0^+$ and $\widehat h_{0,\varphi} \in H^1(M_1(1), g_1)$ 
such that $\widehat h_{\e_j,\varphi} \rightharpoonup \widehat h_{0,\varphi}$ in $H^1(M_1(1), g_1)$ weakly as $j \to \infty$ and, in particular, strongly in $L^2(M_1(1), g_1)$.
Furthermore, passing  to the limit in \eqref{eq_harmonic_extension_new}, we obtain
\begin{equation}
\int_{M_1(1)}g_1(\nabla^{g_1}  \widehat h_{0,\varphi}, \nabla^{g_1}w) \dvol^{g_1}=0,
\end{equation}
for any $w \in H^1_0( M_1(1),g_1)$, while $\widehat h_{0,\varphi}=P\circ \Psi \circ (\exp_{p_1}^{M_1})^{-1}$ on $\partial M_1(1)$. It is clear that the solution to this problem is unique. Next we show that the convergence of  $\widehat h_{\e_j,\varphi}$ to $\widehat h_{0,\varphi}$ is actually strong. The difference 
between $\widehat h_{0,\varphi}- \widehat h_{\e_j,\varphi}$ solves the problem
\begin{equation}
\begin{cases}\label{proop_the_vengence_time_consuming_lemma_prob_harmonic_extension}
(-\Delta^{g_1}+\e^2)(\widehat h_{0,\varphi}- \widehat h_{\e,\varphi})=\e^2 \widehat{h}_{0,\varphi}-\e\varphi(p_0), & \textnormal{in } M_1(1), \\
\widehat h_{0,\varphi}- \widehat h_{\e,\varphi}= P\circ \Psi \circ(\exp_{p_1}^{M_1})^{-1}-\frac{1}{\e}(\varphi\circ \Phi^{-1}_\e-\varphi(p_0)),  & \textnormal{on }\partial M_1(1).
\end{cases}
\end{equation}
Thus, by classical elliptic regularity theory, we have that
\begin{equation}
    \norm{\widehat{h}_{0,\varphi}-\widehat{h}_{\e,\varphi}}_{H^1(M_1(1),g_1)}^2 
    \leq C\left( \e^2\norm{\widehat{h}_{0,\varphi}}_{L^2(M_1(1),g_1)}^2+\e\varphi(p_0)+\norm*{P\circ \Psi -\widehat{\varphi}_\e\circ F_\e}_{C^1(\mathbb{B}_1)}\right).
\end{equation}
Hence, by \eqref{eq_conv_P_psi} we conclude that $\widehat h_{\e_j,\varphi} \to \widehat h_{0,\varphi}$  strongly in $H^1(M_1(1), g_1)$ as $j \to \infty$. By the Urysohn subsequence  principle,  $\widehat h_{\e,\varphi} \to \widehat h_{0,\varphi}$  strongly in $H^1(M_1(1), g_1)$ as $\e \to 0^+$.
Finally, \eqref{limit_he} trivially follows. The proof is thereby complete.
\end{proof}

Let us now define with respect to normal coordinates $(x_1,...,x_m)$ induced by $\Theta_\e(x):=\exp^{M_0}_{p_0}(\e x)$, denoting by $g_{\e,0}$ the local representation of the metric $g_0$ in such coordinates,
\begin{equation}
\sqrt{|g_{\e,0}|}(x):=\sqrt{|g_0|}(\Theta_\e(x)) \quad \text{ and } \quad  A_{\e,0}(x):= \sqrt{|g_{\e,0}|}(x)(g_{0})^{ij}(\Theta_\e(x))
\end{equation}
and let $D_0\subseteq \R^m$ the domain of the exponential map $\exp^{M_0}_{p_0}$, where we are identifying the tangent space $T_{p_0} M_0$ with $\R^m$.
We notice that 
\begin{equation}\label{eq_ge_Ae}
\sqrt{|g_{\e,0}|} \to 1  \text{ and }   A_{\e,0} \to \mathop{{\rm Id}_m} \quad \text{ as }\e \to 0^+,\quad\textnormal{uniformly on compact sets},
\end{equation}
where $\mathop{\rm Id}_m$ denotes the identity $m\times m$. We also observe that, since $M_0$ is compact, then there exists $c=c(M_0,g_0)>0$ such that
\begin{equation}\label{eq_g_A_bound}
    c^{-1}\leq \sqrt{|g_{\e,0}|}(x)\leq c\quad\text{and}\quad c^{-1}|\bm{v}|^2\leq A_{\e,0}(x)\bm{v}\cdot\bm{v}\leq c|\bm{v}|^2
\end{equation}
for all $x\in \frac{1}{\e}D_0$ and all $\bm{v}\in\R^m$. Passing to normal coordinates and  making the change of variables $x=\e y$, we obtain a bijection between the Sobolev spaces
\begin{equation}
H^{1}(M_\e(\e),g_\e)\quad  \text{ and } \quad  \mathcal{H}^{1}_\e
\end{equation}
where,  denoting $\frac{1}{\e}D_0:=\left\{\frac{x}{\e}: x \in D_0\right\}$, we set
\begin{align}
\mathcal{H}^{1}_\e:=
\left\{(u_0,u_1) \in H^1\left(\left(\frac{1}{\e}D_0\right)\setminus \mathbb{B}_1, g_{\e,0}\right)\times H^1(M_1(1),\e^2 g_1):{u_0\circ F_\e=u_1 \circ \exp_{p_1}^{M_1}} \text{ on } \partial \mathbb{B}_1 \right\}
\end{align}
endowed with the norm
\begin{equation}
\norm{(u_0,u_1)}_{\mathcal{H}^{1}_\e}^2:=\norm{u_0}^2_{ H^1\left(\left(\frac{1}{\e}D_0\right)\setminus \mathbb{B}_1, g_{\e,0}\right)}
+\norm{u_1}^2_{ H^1(M_1(1),\e^2 g_1)}.
\end{equation}
Then writing  \eqref{eq_Ve_connected_sum} with respect to the new coordinates and making the change of variables  $x=\e y$,
\begin{align}\label{eq_Ve_connected_sum_normal}
\int_{(\frac{1}{\e}D_0)\setminus \mathbb{B}_1} & \left[ (  A_{\e,0} \nabla   \widehat V_{\e,0})\cdot \nabla u_0+\e^2 \sqrt{|g_{\e,0}|} \widehat V_{\e,0} u_0 \right] \textnormal{d}x \\
+ & \int_{M_1(1)} \left[g_1 \left(\nabla^{g_1} \widehat V_{\e,1}, \nabla^{g_1} u_1\right)+\e^2 \widehat V_{\e,1}u_1 \right] \dvol^{g_1} \notag\\
= & \int_{(\frac{1}{\e}D_0)\setminus \mathbb{B}_1} \left[(  A_{\e,0}\nabla \widehat \varphi_\e) \cdot \nabla u_0+\e^2 \sqrt{|g_{\e,0}|} \widehat  \varphi_\e u_0 \right] \textnormal{d}x \notag\\
&+ \int_{M_1(1)} \left[g_1 \left(\nabla^{g_1} \widehat h_{\e,\varphi}, \nabla^{g_1} u_1\right)+\e^2 \widehat{h}_{\e,\varphi} u_1 \right] \dvol^{g_1}\notag\\
& - \lambda_{0,N} \e^2\int_{(\frac{1}{\e}D_0)\setminus \mathbb{B}_1} \sqrt{|g_{\e,0}|} \widehat \varphi_\e u_0\ \textnormal{d}x - \lambda_{0,N} \e^2 \int_{M_1(1)} \widehat{h}_{\e,\varphi} u_1 \dvol^{g_1}\notag
\end{align}
 for any $(u_0,u_1) \in \mathcal{H}^{1}_\e$, where $``\cdot"$ denotes the Euclidean scalar product on $\R^m$, and  $\nabla$ the Euclidean gradient. 
 
 We are now going to study the convergence as $\e \to 0^+$ of $ (\widehat V_{\e,0}, \widehat V_{\e,1})$ in a suitable functional setting. To this end, let us define  the space $D^{1,2}(\R^m\setminus\mathbb{B}_1)$ as the completion of $C_c^\infty(\R^m\setminus\mathbb{B}_1)$ with respect to the norm induced by the scalar product
 \begin{equation}
    (u,v)_{D^{1,2}(\R^m\setminus \mathbb{B}_1)}=\int_{\R^m\setminus \mathbb{B}_1} \nabla u\cdot\nabla v \, \textnormal{d}x .
\end{equation}
In view of the Hardy-type inequality \cite[Lemma 5.3]{FLO_neumann}, since $m\geq 3$, we have the following characterization of $D^{1,2}(\R^m\setminus\mathbb{B}_1)$ as a concrete functional space
\begin{equation*}
    D^{1,2}(\R^m\setminus \mathbb{B}_1):=\left\{ u\in L^1_{\textup{loc}}(\R^m\setminus \mathbb{B}_1)\colon \int_{\R^m \setminus \mathbb{B}_1}\left(|\nabla u|^2 +\frac{u^2}{|x|^2}\right)\, \textnormal{d}x <+\infty\right\}.
\end{equation*}

Finally, we define
\begin{equation}
\mathcal{D}:=
\left\{(u_0,u_1) \in D^{1,2}(\R^m\setminus\mathbb{B}_1)\times H^1(M_1(1),g_1): 
u_0\circ \Psi=u_1 \circ \exp_{p_1}^{M_1} \text{ on } \partial \mathbb{B}_1 \right\}.
\end{equation}
In the next proposition, we study the asymptotics of  $\widehat V_\e$ as $\e \to 0^+$.

\begin{proposition}\label{lemma_widehatV_bounded}
As $\e \to 0^+$, for any $R>2$ we have that 
\begin{align}
& \widehat V_{\e,0} \rightharpoonup  V_{0,0} \quad \text{weakly  in } H^1(\mathbb{B}_R\setminus \mathbb B_1),\\
& \widehat V_{\e,0} \to V_{0,0} \quad \text{strongly   in } L^2(\mathbb B_R\setminus \mathbb B_1),\\
& \widehat V_{\e,1} \rightharpoonup V_{0,1} \quad \text{weakly   in } H^1(M_1(1), g_1),
\end{align}
where $V_0=(V_{0,0},V_{0,1}) \in  \mc{D}$ is the unique solution of 
 \begin{align}\label{eq_V0V1}
\int_{\R^m\setminus \mathbb{B}_1} & \nabla V_{0,0} \cdot \nabla u_0 \,\textnormal{d}x +\int_{M_1(1)} g_1 \left(\nabla^{g_1} V_{0,1}, \nabla^{g_1} u_1\right) \dvol^{g_1} \\ \notag
= & -\int_{\partial \mb{B}_1}\nabla(\varphi\circ\exp_{p_0}^{M_0})(0)\cdot\nu\,  u_0 \,\textnormal{d}\sigma \notag
+ \int_{M_1(1)} g_1 \left(\nabla^{g_1} \widehat h_{0,\varphi}, \nabla^{g_1} u_1\right) \dvol^{g_1},
\end{align}
where $\nu=\frac{x}{|x|}$, for any $(u_0,u_1) \in \mc{D}$.
In particular, 
\begin{equation}\label{eq_norm_L2_nabla_manifolds_M1}
\norm{V_{\e,1}}_{L^2(M_1(1),\e^2g_1)}^2=O(\e^{m+2}), \quad \text{ as } \e \to 0^+.
\end{equation}
\end{proposition}

\begin{proof}
By \Cref{Proposition_norm_J}, \Cref{the_vengence_time_consuming_lemma}, \eqref{limit_psi_0}   passing  in normal coordinates and  making the change of variables  $x=\e y$, we obtain
\begin{align}\label{proof_lemma_widehatV_bounded_1}
\int_{(\frac{1}{\e}D_0)\setminus \mathbb{B}_1} & \left[ ({A}_{\e,0} \nabla   \widehat V_{\e,0})\cdot \nabla  \widehat V_{\e,0} +\e^2 \sqrt{| {g}_{\e,0}|} |\widehat V_{\e,0}|^2 \right] \textnormal{d}x \\
&+  \int_{M_1(1)} \left[g_1 \left(\nabla^{g_1} \widehat V_{\e,1}, \nabla^{g_1} \widehat V_{\e,1}\right)+\e^2 |\widehat V_{\e,1}|^2 \right] \dvol^{g_1} \notag\\
=& \e^{-m}\norm{V_\e}^2_{H^1(M_\e ,g_\e)}\notag \\
\leq & K \e^{-m}\left| \int_{B_\e^{M_0}(p_0)}\left[g_0 \left(\nabla^{g_0} \varphi, \nabla^{g_0} \varphi\right)
+|\varphi-\varphi(p_0)|^2 \right]\dvol^{g_0}\right| \notag\\
&+  K \e^{-2}\left| \int_{M_1(1)} \left[g_1 \left(\nabla^{g_1} h_{\e,\varphi},\nabla^{g_1} h_{\e,\varphi}\right)+\e^2 |h_{\e,\varphi}-\varphi(p_0)|^2\right] \dvol^{g_1} \right|\notag + K\\
=& K \left| \int_{\mb{B}_1}\left[({A}_{\e,0} \nabla   \widehat \varphi_\e)\cdot \nabla   \widehat \varphi_\e +\e^2 \sqrt{| {g}_{\e,0}|}\, | \widehat \varphi_\e|^2 \right]\right|\notag \\
&+  K \left| \int_{M_1(1)} \left[g_1 \left(\nabla^{g_1} \widehat h_{\e,\varphi},\nabla^{g_1}  \widehat h_{\e,\varphi}\right)+\e^2 |\widehat h_{\e,\varphi}|^2\right] \dvol^{g_1} \right|+K
\le  C,
\end{align}
for some positive constant $C$ that does not depend on $\e$, by \eqref{eq_ge_Ae}.

Fix $R>2$ and let $R_0>1$  be such that $\mb{B}_{R_0} \subset D_0$. By \cite[Lemma 5.3]{FLO_neumann} for any $\e < \frac{R_0}{2R}$ 
\begin{align}
\int_{\mathbb{B}_R\setminus \mathbb{B}_1} \left[|\nabla\widehat V_{\e,0} |^2  +\frac{|\widehat V_{\e,0}|^2}{|x|^2} \right] \, \textnormal{d}x &\le  
\int_{\mathbb{B}_{\frac{R_0}{2\e}}\setminus \mathbb{B}_1}\left[|\nabla\widehat V_{\e,0} |^2  +\frac{|\widehat V_{\e,0}|^2}{|x|^2} \right]\,  \textnormal{d}x\\
&\le C_1 \int_{\mathbb{B}_{\frac{R_0}{2\e}} \setminus \mathbb{B}_1} \left[|\nabla\widehat V_{\e,0}|^2 +\frac{\e^2}{R_0^2}{|\widehat V_{\e,0}|^2}  \right] \textnormal{d}x\\ 
&\le C_2 \int_{\mathbb{B}_{\frac{R_0}{2\e}}\setminus \mathbb{B}_1} \left[ ({A}_{\e,0} \nabla   \widehat V_{\e,0})\cdot \nabla  \widehat V_{\e,0} +\e^2 \sqrt{| {g}_{\e,0}|} |\widehat V_{\e,0}|^2 \right] \textnormal{d}x\\
&\le C_2 \int_{(\frac{1}{\e}D_0)\setminus \mathbb{B}_1} \left[ ({A}_{\e,0} \nabla   \widehat V_{\e,0})\cdot \nabla  \widehat V_{\e,0} +\e^2 \sqrt{| {g}_{\e,0}|} |\widehat V_{\e,0}|^2 \right] \textnormal{d}x \\
&\le C_2 C,\label{eq_V_bound}
\end{align}
for some positive constants $C,C_1,C_2$ depending only on $m,R_0$ and $g_0$ but not $\e$ and $R$, in view of \eqref{eq_g_A_bound}. 
In conclusion, $\{\widehat V_{\e,0}\}_{\e \in (0,R_0/2R)}$ is bounded in $H^1(\mathbb{B}_R\setminus \mathbb{B}_1)$ for any $R>2$.
With a diagonal argument,
there exists a subsequence $\e_j \to 0^+$ as $j \to \infty$ and a function $V_{0,0} \in H^1(\mathbb{B}_R\setminus\mathbb{B}_1)$ such that 
\begin{align}
& \widehat V_{\e_j,0} \rightharpoonup  V_{0,0} \quad \text{weakly  in } H^1(\mathbb{B}_R\setminus \mathbb B_1),\label{eq_conv_V_0_H1}\\
& \widehat V_{\e_j,0} \to V_{0,0} \quad \text{strongly   in } L^2(\mathbb B_R\setminus \mathbb B_1),\label{eq_conv_V_0_L2}\\
& \widehat V_{\e_j,0} \to V_{0,0} \quad \text{strongly   in } L^2(\partial  \mathbb B_1),
\end{align}
as $j\to \infty$ for any $R>2$. Furthermore, by weak lower semicontinuity and letting $R\to+\infty$ in \eqref{eq_V_bound}, we have that
\begin{equation}
    \int_{\R^m\setminus \mathbb{B}_1} \left[|\nabla V_{0,0} |^2  +\frac{| V_{0,0}|^2}{|x|^2} \right] \, \textnormal{d}x<+\infty
\end{equation}
which implies that $V_{0,0}\in D^{1,2}(\R^m\setminus \mathbb{B}_1)$.

Furthermore, by classical Sobolev trace theory,   $\{\widehat V_{\e,0}\}_{\e \in (0,R_0/2R)}$ is bounded in $L^2(\partial \mathbb{B}_1)$. 
Then, since $\widehat V_{\e,1}\circ \exp_{p_1}^{M_1}=\widehat V_{\e,0}\circ F_\e \text{ on } \partial M_1(1)$, also 
$\{\widehat V_{\e,1}\}_{\e \in (0,R_0/2R)}$ is bounded in $L^2(\partial M_1(1))$. We conclude that, in view of \eqref{proof_lemma_widehatV_bounded_1},  $\{\widehat V_{\e,1}\}_{\e \in (0,R_0/2R)}$ is bounded in
$H^1(M_1(1),g_1)$. Indeed,  there exists a constant $c>0$ such that for any $w \in H^1(M_1(1),g_1)$
\begin{equation}
\int_{M_1(1)} |w|^2 \dvol^{g_1} \le  c\left(\int_{M_1(1)} g_1(\nabla^{g_1}w,\nabla^{g_1}  w) \dvol^{g_1}+\int_{\partial M_1(1)} |w|^2  \textnormal{d}\sigma^{g_1}\right).
\end{equation}
The constant $c$ is actually the inverse of the first Robin eigenvalue of the Laplacian. Hence, there exists a subsequence  $\{\widehat V_{\e_j,1}\}_{j \in \mb N}$ with $\e_j \to 0^+$ and $V_{0,1} \in H^1(M_1(1),g_1)$ such that
\begin{align}
    &\widehat V_{\e_j,1} \rightharpoonup V_{0,1}\quad \text{weakly in }H^1(M_1(1),g_1) \label{eq_conv_V_1_H1}\\
    &\widehat V_{\e_j,1} \rightharpoonup V_{0,1}\quad\text{strongly in }L^2(M_1(1),g_1)\label{eq_conv_V_1_L2}
\end{align}
as $j \to \infty$. 
Let $(u_0,u_1)\in\mathcal{D}$ be such that $u_0\in C_c^\infty(\R^m\setminus\mathbb{B}_1)$ and let $R>2$ be such that $\mathrm{supp}\, u_0\subseteq \mathbb{B}_R$. Hence, by \eqref{eq_Ve_connected_sum_normal} we have that 
\begin{align}
\int_{\mathbb{B}_R\setminus \mathbb{B}_1} & \left[ (  A_{\e,0} \nabla   \widehat V_{\e,0})\cdot \nabla u_0+\e^2 \sqrt{|g_{\e,0}|} \widehat V_{\e,0} u_0 \right] \textnormal{d}x \\
+ & \int_{M_1(1)} \left[g_1 \left(\nabla^{g_1} \widehat V_{\e,1}, \nabla^{g_1} u_1\right)+\e^2 \widehat V_{\e,1}u_1 \right] \dvol^{g_1} \notag\\
= & \int_{\mathbb{B}_R\setminus \mathbb{B}_1} \left[(  A_{\e,0}\nabla \widehat \varphi_\e) \cdot \nabla u_0+\e^2 \sqrt{|g_{\e,0}|} \widehat  \varphi_\e u_0 \right] \textnormal{d}x \notag\\
&+ \int_{M_1(1)} \left[g_1 \left(\nabla^{g_1} \widehat h_{\e,\varphi}, \nabla^{g_1} u_1\right)+\e^2 \widehat{h}_{\e,\varphi} u_1 \right] \dvol^{g_1}\notag\\
& - \lambda_{0,N} \e^2\int_{\mathbb{B}_R\setminus \mathbb{B}_1} \sqrt{|g_{\e,0}|} \widehat \varphi_\e u_0\ \textnormal{d}x - \lambda_{0,N} \e^2 \int_{M_1(1)} \widehat{h}_{\e,\varphi} u_1 \dvol^{g_1}.\notag
\end{align}
Therefore, in view of \eqref{limit_psi_0}, \eqref{eq_ge_Ae}, \eqref{eq_conv_V_0_H1}, \eqref{eq_conv_V_0_L2}, \eqref{eq_conv_V_1_H1}, \eqref{eq_conv_V_1_L2} and \Cref{the_vengence_time_consuming_lemma} we get that $(V_{0,0},V_{0,1})$ satisfy
\begin{align}
\int_{\R^m\setminus \mathbb{B}_1} & \nabla V_{0,0} \cdot \nabla u_0 \,\textnormal{d}x +\int_{M_1(1)} g_1 \left(\nabla^{g_1} V_{0,1}, \nabla^{g_1} u_1\right) \dvol^{g_1} \\
= & \int_{\R^m\setminus \mathbb{B}_1} \nabla P  \cdot \nabla u_0 \,\textnormal{d}x 
+ \int_{M_1(1)} g_1 \left(\nabla^{g_1} \widehat h_{0,\varphi}, \nabla^{g_1} u_1\right) \dvol^{g_1},
\end{align}
for any $(u_0,u_1)\in\mathcal{D}$ such that $u_0\in C_c^\infty(\R^m\setminus\mathbb{B}_1)$. Now, integrating by parts, since $\Delta P=0$ and $\nabla P$ can be identified with $\nabla(\varphi\circ \exp_{p_0}^{M_0})(0)$, one can see that
\begin{equation}
   \int_{\R^m\setminus \mathbb{B}_1} \nabla P  \cdot \nabla u_0 \,\textnormal{d}x =-\int_{\partial \mathbb{B}_1}u_0 \nabla(\varphi\circ \exp_{p_0}^{M_0})(0)\cdot \nu\diff\sigma, 
\end{equation}
where $\nu:=x/|x|$. Hence, \eqref{eq_V0V1} holds, by density, for all $(u_0,u_1)\in\mathcal{D}$. Furthermore, \eqref{eq_V0V1} admits a unique solution, as it is easy to check by taking the difference of two solutions, and so the Urysohn subsequence principle allows us to conclude that the convergence of $V_\e$ to $V_0$ holds as $\e \to 0^+$ and not just along a subsequence.

Finally  
\begin{equation}
\norm{V_{\e,1}}^2_{L^2(M_1(\e),\e^2 g_1)}=\e^{m+2}\norm{\hat V_{\e,1}}^2_{L^2(M_1(1),g_1)}
=O(\e^{m+2}), \quad \text{ as } \e \to 0^+,
\end{equation}
since $\hat V_{\e,1}$ is bounded in $H^1(M_1(1),g_1)$, that is, we have proved \eqref{eq_norm_L2_nabla_manifolds_M1}.
\end{proof}

We are now ready to prove the main result of the present section.
\begin{theorem}\label{theor_eigen_connected_sum_sharp}
Let $h_{\e,N}$ be as in \eqref{def_he_con_sum}, $V_0=(V_{0,0},V_{0,1})$ be the unique solution to \eqref{eq_V0V1}, and $\widehat h_{0,\varphi}$ the  unique solution to \eqref{eq_widehat_h0}.
Then, as $\e \to 0^+$, 
\begin{multline}\label{eq_hen_connect_sum_final}
h_{\e,N}(\varphi,\varphi)= \e^m\Bigg(-(|\nabla^{g_0}\varphi(p_0)|^2 +  (1-\la_{0,N}) |\varphi(p_0)|^2) |\mathbb{B}_1|
+\int_{\partial \mb{B}_1} V_{0,0}\nabla (\varphi \circ \exp_{p_0}^{M_0})(0) \cdot \nu  \, \textnormal{d} \sigma\\
+(1-\la_{0,N}) \textnormal{vol}^{g_1}(M_1(1))  |\varphi(p_0)|^2 
+\int_{M_1(1)} g_1 \left(\nabla^{g_1} V_{0,1}, \nabla^{g_1} \widehat h_{0,\varphi}\right) \dvol^{g_1}\\
+\int_{M_1(1)}g_1 \left(\nabla^{g_1} \widehat h_{0,\varphi}, \nabla^{g_1} \widehat h_{0,\varphi}\right) \dvol^{g_1}\Bigg) +o(\e^m).
\end{multline}
Moreover, if we define 
\begin{equation}
h_{0,N}(\varphi,\varphi):=
\lim_{\e\to 0^+}\e^{-m}h_{\e,N}(\varphi,\varphi)
\end{equation}
and we denote by $\{\gamma_{0,i}\}_{i=1,\dots,M_{0,N}}$  the eigenvalues of $h_{0,N}$ counted with multiplicity, then, as $\e \to 0^+$, 
\begin{equation}\label{eq_eige_var_conn_sum_sharp}
\lambda_{\e,N+i-1}-\lambda_{0,N}=\e^m\gamma_{0,i}+o(\e^m).
\end{equation} 
Finally, we can write
\begin{align}
    -h_{0,N}(\varphi,\varphi)&=\mathbb{A}(\nabla(\varphi\circ \exp_{p_0}^{M_0})(0),\nabla(\varphi\circ \exp_{p_0}^{M_0})(0)) \\
    &\quad-(\lambda_{0,N}-1)\big[| \mathbb{B}_1|-\mathrm{vol}^{g_1}(M_1(1))\big]|\varphi(p_0)|^2,
\end{align}
where $\mathbb{A}\colon \R^m\times\R^m\to \R$ is a bilinear symmetric form, in view of the fact that
\begin{equation}
    \nabla (\varphi\circ \exp_{p_0}^{M_0})\quad \text{and} \quad\nabla^{g_0}\varphi(p_0)
\end{equation}
can be identified.
\end{theorem}
\begin{proof}
By \eqref{eq_psi_connect_sum} tested with $E_\e V_{\e,0}$ and \eqref{eq_Ve_connected_sum} tested with $\varphi$,
\begin{multline}
\la_{0,N}(V_\e,L_\e (\varphi))_{L^2(M_\e,g_\e)}=\la_{0,N}\int_{M_0} E_\e V_{\e,0} \varphi \dvol^{g_0} \\
-\la_{0,N}\int_{B_\e^{M_0}} E_\e V_{\e,0} \varphi\dvol^{g_0}+ \la_{0,N} \e^m \int_{M_1(1)}  V_{\e,1}  h_{\e,\varphi}  \dvol^{g_1}\\
= \int_{M_0} [g_0(\nabla^{g_0} E_\e V_{\e,0}, \nabla^{g_0}  \varphi)  +  E_\e V_{\e,0} \varphi] \dvol^{g_0}-\la_{0,N}\int_{B_\e^{M_0}} E_\e V_{\e,0} \varphi \dvol^{g_0}+ \la_{0,N} \e^m \int_{M_1(1)}  V_{\e,1}  h_{\e,\varphi}  \dvol^{g_1}\\
= \int_{M_0(\e)} [g_0(\nabla^{g_0}  V_{\e,0}, \nabla^{g_0}  \varphi)  +   V_{\e,0} \varphi] \dvol^{g_0}\\
+\int_{B_\e^{M_0}} [g_0(\nabla^{g_0}  E_\e V_{\e,0}, \nabla^{g_0}  \varphi)  +  (1-\la_{0,N})  E_\e V_{\e,0} \varphi] \dvol^{g_0}       
+ \la_{0,N} \e^m \int_{M_1(1)}  V_{\e,1}  h_{\e,\varphi}  \dvol^{g_1} \\
=- \e^{m-2} \int_{M_1(1)} \left[g_1 \left(\nabla^{g_1} V_{\e,1}, \nabla^{g_1} h_{\e,\varphi}\right)+\e^2 V_{\e,1}h_{\e,\varphi}\right] \dvol^{g_1} \\
+\int_{M_0(\e)} \left[g_0 \left(\nabla^{g_0} \varphi, \nabla^{g_0} \varphi\right)+|\varphi|^2 \right] \dvol^{g_0}
+ \e^{m-2} \int_{M_1(1)} \left[g_1 \left(\nabla^{g_1} h_{\e,\varphi}, \nabla^{g_1} h_{\e,\varphi}\right)+\e^2 |h_{\e,\varphi}|^2 \right] \dvol^{g_1}\\
+\int_{B_\e^{M_0}} [g_0(\nabla^{g_0}  E_\e V_{\e,0}, \nabla^{g_0}  \varphi)  +  (1-\la_{0,N})  E_\e V_{\e,0} \varphi] \dvol^{g_0}\\
- \lambda_{0,N} \int_{M_0(\e)} |\varphi|^2 \dvol^{g_0} + \la_{0,N} \e^m \int_{M_1(1)}  V_{\e,1}  h_{\e,\varphi}  \dvol^{g_1}
- \la_{0,N} \e^m \int_{M_1(1)}  |h_{\e,\varphi}|^2  \dvol^{g_1}\\
=-\int_{B_\e^{M_0}} [g_0(\nabla^{g_0}  \varphi, \nabla^{g_0}  \varphi)  +  (1-\la_{0,N})|\varphi|^2] \dvol^{g_0}
-\int_{\partial B_\e^{M_0}} \nu_0(\varphi) V_{\e,0}\ \textnormal{d}\sigma^{g_0}\\
- \e^{m-2} \int_{M_1(1)} \left[g_1 \left(\nabla^{g_1} V_{\e,1}, \nabla^{g_1} h_{\e,\varphi}\right)+\e^2 V_{\e,1}h_{\e,\varphi}\right] \dvol^{g_1} \\
+ \e^{m-2} \int_{M_1(1)} \left[g_1 \left(\nabla^{g_1} h_{\e,\varphi}, \nabla^{g_1} h_{\e,\varphi}\right)+\e^2 |h_{\e,\varphi}|^2 \right] \dvol^{g_1}\\
+ \la_{0,N} \e^m \int_{M_1(1)}  V_{\e,1}  h_{\e,\varphi}  \dvol^{g_1} - \la_{0,N} \e^m \int_{M_1(1)}  |h_{\e,\varphi}|^2  \dvol^{g_1}.
\end{multline}
Since $\varphi(p)=\varphi(p_0)+o(1)$ and  $\nabla^{g_0}\varphi(p)=\nabla^{g_0}\varphi(p_0)+o(1)$, as $\e \to 0^+$ uniformly for any $p \in B_\e^{M_0}$,
\begin{align}
&\int_{B^{M_0}_\e} [g_0(\nabla^{g_0}  \varphi, \nabla^{g_0}  \varphi)  +  (1-\la_{0,N})|\varphi|^2] \dvol^{g_0}\\
=& \e^m\left([|\nabla^{g_0}\varphi(p_0)|^2+  (1-\la_{0,N})|\varphi(p_0)|^2]  \, |\mathbb{B}_1|\right)+o(\e^m).
\end{align}
Furthermore, passing to normal coordinates
\begin{align}
&\int_{\partial B^{M_0}_\e} \nu_0(\varphi) V_{\e,0}\ \textnormal{d}\sigma^{g_0}=\int_{ \partial\mb{B_\e}} (V_{\e,0}\circ \exp_{p_0}^{M_0}) \left(A_{\e,0} \nabla(\varphi\circ \exp_{p_0}^{M_0})\right) \cdot \nu \diff  \sigma \\
= &\e^{m}\int_{\partial \mb{B}_1} \widehat V_{\e,0} \left( A_{\e,0} \nabla \widehat \varphi_\e\right) \cdot \nu \,   \diff  \sigma
= \e^{m}\int_{\partial \mb{B}_1}  V_{0,0} \nabla (\varphi\circ\exp_{p_0}^{M_0})(0) \cdot \nu  \diff  \sigma +o(\e^{m}),
\end{align}
as $\e \to 0^+$, where $\nu:=\frac{x}{|x|}$.
Furthermore, as $\e \to 0^+$ by \Cref{the_vengence_time_consuming_lemma} and \Cref{lemma_widehatV_bounded}
\begin{align}
&\e^{m-2}\int_{M_1(1)} \left[g_1 \left(\nabla^{g_1} V_{\e,1}, \nabla^{g_1} h_{\e,\varphi}\right)+\e^2 V_{\e,1}h_{\e,\varphi}\right] \dvol^{g_1}\\
=&\e^{m}\int_{M_1(1)} g_1 \left(\nabla^{g_1} V_{0,1}, \nabla^{g_1} \widehat h_{0,\varphi}\right) \dvol^{g_1}+o(\e^m).
\end{align}
Moreover, as $\e \to 0^+$
\begin{align}
&\e^{m-2} \int_{M_1(1)} \left[g_1 \left(\nabla^{g_1} h_{\e,\varphi}, \nabla^{g_1} h_{\e,\varphi}\right)
+\e^2 |h_{\e,\varphi}|^2 \right] \dvol^{g_1}   \\
&=\e^{m}\int_{M_1(1)} \left(g_1 \left(\nabla^{g_1} \widehat h_{0,\varphi}, \nabla^{g_1} \widehat h_{0,\varphi}\right) + |\varphi(p_0)|^2\right) \dvol^{g_1}+o(\e^{m}).
\end{align}
Similarly, as $\e \to 0^+$ 
\begin{equation}
\e^{m} \int_{M_1(1)}  V_{\e,1}  h_{\e,\varphi}  \dvol^{g_1}= \e^{m+1}  \varphi(p_0) \int_{M_1(1)}  V_{0,1}  \dvol^{g_1} +o(\e^{m+1})
\end{equation}
and 
\begin{equation}
\e^m \int_{M_1(1)}  |h_{\e,\varphi}|^2  \dvol^{g_1} = \e^m |\varphi(p_0)|^2 \textnormal{vol}^{g_1}(M_1(1))+o(\e^m).
\end{equation}
In conclusion, putting everything together we have proved \eqref{eq_hen_connect_sum_final}. Finally, we apply \Cref{theor_eigen_connected_sum_sharp}: in view of \eqref{eq_hen_connect_sum_final} we have that
\begin{equation}
    \gamma_{\e,i}=\e^m \gamma_i+o(\e^m)\quad\text{as }\e\to 0^+
\end{equation}
and, by definition, $\tau_{\e,N}^2=O(\e^m)$ as $\e\to 0^+$; in addition $\delta_{\e,N}^2=o(\e^m)$ thanks to \eqref{eq_manifolds_L2norm_smallo} and \eqref{eq_norm_L2_nabla_manifolds_M1}. Hence, \eqref{eq_eige_var_conn_sum_sharp} holds.
\end{proof}

\section*{AI statement}
The authors acknowledge the use of AI tools for proofchecking some of the results. The strategy and the arguments are human-generated and the paper does not contain AI written text. The authors take full responsibility for the whole content of the paper.

\section*{Acknowledgements}
A. Bisterzo, R. Ognibene, and G. Siclari are partially supported by the 2026 INdAM-GNAMPA project \emph{Asymptotic analysis of variational problems}, n. \texttt{CUP\_E53C25002010001}. A. Bisterzo and  G. Siclari are also supported by ``Centro di Ricerca Matematica Ennio De Giorgi''. P.~Roychowdhury is grateful to IIT Hyderabad for its warm hospitality and excellent research environment during the completion of this work.

\bibliographystyle{acm}
\bibliography{references}	
\end{document}